\documentclass[a4paper,reqno,10pt,oneside]{amsart}
\usepackage[left=2.61cm,right=2.61cm,top=2.72cm,bottom=2.72cm]{geometry}

\usepackage{hyperref}
\usepackage[english]{babel}
\usepackage{mathrsfs,amssymb,enumitem}

\allowdisplaybreaks[1]

\numberwithin{equation}{section}

\newtheorem{theorem}{Theorem}[section]
\newtheorem{proposition}[theorem]{Proposition}
\newtheorem{corollary}[theorem]{Corollary}
\newtheorem{lemma}[theorem]{Lemma}
\theoremstyle{definition}
\newtheorem{remark}[theorem]{Remark}

\newcommand{\N}{\mathbb{N}}
\newcommand{\E}{\mathcal{E}}
\newcommand{\beq}{\begin{equation}}
\newcommand{\eeq}{\end{equation}}
\newcommand{\R}{\mathbb R}
\newcommand{\RN}{\mathbb{R}^N}
\newcommand{\divg}{\operatorname{div}}
\newcommand{\tr}{\operatorname{tr}}
\newcommand{\SN}{\mathbb{S}^{N-1}}
\newcommand{\Y}{\mathcal{Y}}
\newcommand{\Id}{\mathrm{id}}
\newcommand{\dd}{\mathrm{d}}
\newcommand{\DiffQ}{\operatorname{Diff}^{3,\alpha}_Q(\SN)}
\newcommand{\Param}{\mathscr{U}}
\newcommand{\SpecNeu}{\operatorname{Spec}_{\mathrm{Neu}}}
\makeatletter
\@namedef{subjclassname@2020}{\textup{2020} Mathematics Subject Classification}
\makeatother
\begin{document}

\title[Generic nondegeneracy of CMC radial graphs]{Generic nondegeneracy of constant mean curvature\\ radial graphs}

\author{A\MakeLowercase{lessandro} Iacopetti, F\MakeLowercase{ilomena} Pacella}

\subjclass[2020]{35J93, 53A10, 35B20}
\keywords{Radial graphs, CMC hypersurfaces, Genericity}
\thanks{\emph{Acknowledgements.} Research partially supported by Gruppo Nazionale per l'Analisi Matematica, la Pro\-ba\-bi\-li\-t\`a e le loro Applicazioni (GNAMPA) of the Istituto Nazionale di Alta Matematica (INdAM). Alessandro Iacopetti is partially supported by the GNAMPA 2025 project E53C25002010001: ``Strutture Analitiche e Geometriche in PDEs: Regolarit\`a, Fenomeni Critici e Dinamiche Complesse".}

\address[Alessandro Iacopetti]{Dipartimento di Matematica ``Giuseppe Peano", Universit\`a degli Studi di Torino, Via Carlo Alberto 10, 10123 Torino, Italy}
\email{alessandro.iacopetti@unito.it}

\address[Filomena Pacella]{Dipartimento di Matematica, Sapienza Universit\`a di Roma,  P.le Aldo Moro 2, 00185 Roma, Italy}
\email{filomena.pacella@mat.uniroma1.it}

\begin{abstract}
Given a smooth domain $D\subset\mathbb S^{N-1}$, $N\geq 3$, and a positive constant $H$, we study radial graphs over $D$ with constant mean curvature $H$ which meet orthogonally the boundary of the cone spanned by $D$. 
The corresponding problem is given by a quasilinear elliptic equation on $D$ with homogeneous Neumann
boundary condition.

We investigate the structure of the solution set under smooth perturbations of the domain which are parametrized by an open set of
$C^{3,\alpha}$-diffeomorphisms of the sphere. After pulling the equation back to a fixed reference domain, we apply a classical transversality theorem and prove that, in the class of domains for which $N-1$ does not belong to the Neumann spectrum of the Laplace--Beltrami operator, there is a residual set for which every solution is nondegenerate. 
A consequence of our result is that, for every such domain and every $R>0$, the set of solutions whose $C^{2,\alpha}$-norm is bounded by $R$ is finite.
\end{abstract}

\maketitle
\section{Introduction}

Let $D$ be a smooth proper domain on the unit sphere $\SN$, $N\geq3$.
We consider hypersurfaces in $\RN$ which are radial graphs over $D$ and satisfy a geometric boundary condition. More precisely, for a function $\varphi\in C^2(\overline D)$ we define the graph
$$
\Gamma_\varphi:=\{e^{\varphi(p)}p;\ p\in D\}
$$
and assume that
\begin{equation}\label{eq1:intro}
\frac{\partial\varphi}{\partial\nu_{g_0}}=0 \qquad\text{on }\partial D,
\end{equation}
where $\nu_{g_0}$ is the outward unit normal to $\partial D$ with respect to the standard metric $g_0$ on $\SN$.
Let us observe that, if we denote by $\Sigma_D$ the cone spanned by $D$, i.e.
$$
\Sigma_D:=\{rp;\ r>0,\ p\in D\},
$$
then the hypersurface $\Gamma_\varphi$ is strictly star-shaped with respect
to the vertex of the cone, its boundary lies on $\partial\Sigma_D\setminus\{0\}$ and, by \eqref{eq1:intro}, intersects $\partial\Sigma_D$ orthogonally.

In this paper we study constant mean curvature (CMC in short) radial graphs of this type. 
Thus, for a fixed number $H>0$, we consider the following problem:
\begin{equation}\label{eq:intro-CMC}
\begin{cases}
-\divg_{g_0}\!\left(\dfrac{\nabla_{g_0}\varphi}{\sqrt{1+|\nabla_{g_0}\varphi|_{g_0}^2}}\right) +\dfrac{N-1}{\sqrt{1+|\nabla_{g_0}\varphi|_{g_0}^2}}-(N-1)He^\varphi=0 &\text{in }D,\\[1.2ex]
\dfrac{\partial\varphi}{\partial\nu_{g_0}}=0 &\text{on }\partial D.
\end{cases}
\end{equation}
Note that we have assumed $H>0$ because, as shown in Remark \ref{rem:Hpositivenecessary}, problem \eqref{eq:intro-CMC}
can have a solution only for positive $H$. It is easy to see that the constant function
$$\varphi_H:=\log\frac{1}{H}$$
is a solution of \eqref{eq:intro-CMC}, and the corresponding radial graph is the spherical cap
$$
\Gamma_{\varphi_H}=\partial B_{1/H}\cap\Sigma_D,
$$
where $B_\rho$ denotes the ball of radius $\rho>0$ centered at the origin.

A rigidity result obtained in \cite{ChoePark2011}, and in \cite{PacellaTralli2020} with a detailed proof, shows that if $D$ is a
convex domain on $\SN$, then the spherical cap $\Gamma_{\varphi_H}$ corresponds to the only solution of \eqref{eq:intro-CMC} and, actually, it is the only smooth CMC hypersurface intersecting $\partial\Sigma_D$ orthogonally (i.e.\ not only in the class of radial graphs, unless $\partial D$ has some flat parts, in which case half balls should also be considered).

Thus the interest is to study \eqref{eq:intro-CMC} for nonconvex domains $D$.

An important characterization of these CMC hypersurfaces in cones is obtained by considering the relative perimeter $P(E;\Sigma_D)$ of any
set $E\subset\Sigma_D$ of finite relative perimeter. We recall that, for a smooth set $E$, $P(E;\Sigma_D)$ is the Hausdorff measure of the
part of the boundary $\partial E$ which lies inside $\Sigma_D$ (see \cite{RitoreRosales2004} for more details).
The functional $P(E;\Sigma_D)$ can be differentiated, and it is not difficult to prove that the smooth CMC hypersurfaces which intersect
$\partial\Sigma_D$ orthogonally are the relative boundaries of domains $E$ which are critical points of $P(\cdot\,;\Sigma_D)$ under a volume
constraint (see \cite{RitoreRosales2004} and \cite{IPW} for radial graphs).

Exploiting this variational formulation of CMC radial graphs in cones, it is proved in \cite{IPW} that the spherical cap $\Gamma_{\varphi_H}$ is not the only CMC hypersurface intersecting $\partial\Sigma_D$ orthogonally. 
More precisely, denoting by $\lambda_1(D)$ the first nontrivial Neumann eigenvalue of the Laplace--Beltrami operator $-\Delta_{g_0}$ on the domain $D$, by $\mathcal{H}^{N-1}$ the $(N-1)$-dimensional Hausdorff measure, and by $\SN_+$ a half unit sphere, \cite[Theorem 1.3]{IPW} shows that when $\lambda_1(D)<N-1$ and $\mathcal{H}^{N-1}(D)<\mathcal{H}^{N-1}(\SN_+)$ other such CMC hypersurfaces exist in $\Sigma_D$.

This is related to the crucial fact that, when
\begin{equation}\label{eq4:intro}
\lambda_1(D)=N-1,
\end{equation}
then the hypograph
$$\Omega_{\varphi_H}:=\{rp\in\Sigma_D;\, \ 0<r<1/H,\ p\in D\}$$
is a degenerate critical point of the perimeter functional.
This is equivalent to saying that, under condition \eqref{eq4:intro}, the constant function $\varphi_H$ is a degenerate solution of
problem \eqref{eq:intro-CMC}, i.e.\ the linearization of \eqref{eq:intro-CMC} at $\varphi_H$ admits zero as an eigenvalue.

Moreover, the degeneracy of $\varphi_H$ at a domain $D_0$ for which $\lambda_1(D_0)=N-1$ could induce a bifurcation phenomenon. By this we mean that, considering a family of domains $D_\gamma$ parametrized by $\gamma\in(-\varepsilon,\varepsilon)$ such that $\lambda_1(D_0)=N-1$ and $\lambda_1(D_\gamma)$ crosses the value $N-1$ as $\gamma$ varies, a branch of CMC radial graphs could emanate from the constant graph $\Gamma_{\varphi_H}$. Thus it is clear that the number of solutions of \eqref{eq:intro-CMC} and their degeneracy depend on the domain $D$
on the sphere.

In the present paper we study the structure of the solution set of \eqref{eq:intro-CMC} as $D$ varies, with the aim of showing that, for ``generic'' domains $D$ (in a sense that will be made precise), all solutions of \eqref{eq:intro-CMC} are nondegenerate. As a consequence, we obtain that, ``generically'', problem \eqref{eq:intro-CMC} admits only a finite number of solutions for any H\"older bound we may fix.

To state our result we need to describe a suitable setting. Let $Q$ be a proper open set in $\SN$ and fix a domain $D_0\Subset Q$ of class $C^{3,\alpha}$, with $\alpha\in(0,1)$. We consider the set of $C^{3,\alpha}$ diffeomorphisms of $\SN$ onto itself that are the identity outside $Q$, namely
$$
\DiffQ:=\left\{h\in\operatorname{Diff}^{3,\alpha}(\SN);\ h=\Id\text{ on }\SN\setminus Q \right\}.
$$
This set carries a natural Banach manifold structure, which will be
described in Section 2. For $h\in\DiffQ$ we set
$$
D_h:=h(D_0)\Subset Q,
$$
and, for fixed $H>0$, we consider problem \eqref{eq:intro-CMC} in each domain $D_h$, i.e.
\begin{equation}\label{eq:introprobh}
\begin{cases}
-\divg_{g_0}\!\left(\dfrac{\nabla_{g_0}\varphi}{\sqrt{1+|\nabla_{g_0}\varphi|_{g_0}^2}}\right)+\dfrac{N-1}{\sqrt{1+|\nabla_{g_0}\varphi|_{g_0}^2}}
-(N-1)He^\varphi=0 &\text{in }D_h,\\[1.2ex]
\dfrac{\partial\varphi}{\partial\nu_{g_0}}=0&\text{on }\partial D_h.
\end{cases}
\end{equation}

We then consider the open subset $\mathscr A\subset\DiffQ$ defined by
\begin{equation}\label{eq5:intro}
\mathscr A:=\left\{h\in\DiffQ;\ N-1\notin\SpecNeu(-\Delta_{g_0};D_h)\right\},
\end{equation}
where $\SpecNeu(-\Delta_{g_0};D_h)$ denotes the spectrum of the Laplace--Beltrami operator $-\Delta_{g_0}$ on the domain $D_h$ with
homogeneous Neumann boundary conditions.

We recall that a subset $A$ of a topological space $X$ is called \emph{residual} if its complement is a \emph{meager set}, i.e.\ $X\setminus A$ is a countable union of nowhere dense sets $C_n$, $n\in\N$, which means that $\overline{C_n}$ has empty interior. If $X$ is a Baire space, for instance a complete metric space or, as in our case, an open subset of a Banach manifold (see Section 8), every residual set is dense in $X$. We use the word ``generic'' in the following sense: a property is generic in $X$ if it holds on a residual set $A\subset X$.

Now let us recall the definition of a nondegenerate solution of \eqref{eq:introprobh}. Let us define the map
$$\E:C^{2,\alpha}(\overline{D_h})\longrightarrow C^{0,\alpha}(\overline{D_h}), \qquad \E(\varphi):=-\divg_{g_0}\!\left(\frac{\nabla_{g_0}\varphi}{W_\varphi}\right)+\frac{N-1}{W_\varphi}-(N-1)He^\varphi,$$
where
$$W_\varphi:=\sqrt{1+|\nabla_{g_0}\varphi|_{g_0}^2},$$
and consider its derivative at a solution $\varphi$ of \eqref{eq:introprobh}, i.e.\ the linear operator
$$L_{\varphi}v:=D\E(\varphi)[v], \qquad v\in C^{2,\alpha}(\overline{D_h}).$$

This derivative, with the homogeneous Neumann boundary condition, is computed and studied in a more general setting in Section 3, where
it is shown to be a Fredholm operator of index zero. We say that $\varphi$ is a nondegenerate solution of \eqref{eq:introprobh} if $L_{\varphi}$, with this boundary condition, does not admit zero as an eigenvalue.

Our result is the following.

\begin{theorem}\label{thm:main}
Let $N\geq 3$, $\alpha\in (0,1)$ and let $D_0$ be a $C^{3,\alpha}$ domain on $\SN$, with $D_0\Subset Q$, for a fixed proper open set $Q\subset\SN$, and such that $N-1 \notin \SpecNeu(-\Delta_{g_0};D_0)$.
 Let $\mathscr A$ be the set of diffeomorphisms defined in \eqref{eq5:intro} and, for $h\in\mathscr A$, consider the CMC problem
\eqref{eq:introprobh} in $D_h$, for a fixed positive number $H$. Then the set $\mathscr G$ of all diffeomorphisms $h\in\mathscr A$
for which all solutions of \eqref{eq:introprobh} are nondegenerate is a residual subset of $\mathscr A$. Moreover, for every $h\in\mathscr G$ and $R>0$, problem \eqref{eq:introprobh} in $D_h$ has only finitely many solutions $\varphi$ with $\|\varphi\|_{C^{2,\alpha}(\overline{D_h})}\leq R$.
\end{theorem}

Hence, in short, we prove that, generically with respect to varying the domain $D$ on $\SN$, all solutions of \eqref{eq:intro-CMC}
are nondegenerate, and there are only finitely many of them for any $C^{2,\alpha}$ bound we may fix.

Generic questions for PDEs have long been studied, in particular in connection with spectral problems, for which we refer to some papers by Micheletti \cite{Micheletti1972,Micheletti1973} and Uhlenbeck \cite{Uhlenbeck1972,Uhlenbeck1976}, where the beautiful result about the generic simplicity of the eigenvalues of some elliptic operators is proved. These results are obtained by a topological approach which relies on
so-called transversality theorems, which go back to Quinn \cite{Quinn1970} and others. We refer to the book \cite{Henry} for them and to \cite{SautTemam1979} for generic results for semilinear elliptic problems.

To our knowledge, this is the first time that a generic result is obtained for the CMC problem \eqref{eq:intro-CMC}.

\medskip
To prove Theorem \ref{thm:main}, we use a version of an abstract transversality theorem, \cite[Theorem 5.4]{Henry}, which is suitable
for our problem (see Theorem \ref{thm:Henryspecial}). The first step is to pull back problem \eqref{eq:introprobh} to an
equivalent one in the fixed domain $D_0$, endowed with the metric $g_h$, which is the pull-back of the metric $g_0$ by the diffeomorphism $h\in\mathscr A$. Then we consider a map $F$ (see \eqref{eq:defF}) whose zeros correspond to the solutions of \eqref{eq:introprobh}, or equivalently of problem \eqref{eq:prob0} (see Lemma \ref{lem1:sect2}), and apply Theorem \ref{thm:Henryspecial} to $F$.
Note that the map $F$ has two components, $F=(F_1,F_2)$, because, when the metric changes, the normal vector changes, so that the
Neumann boundary condition must be included in the definition of $F$. Thus the core of the proof of Theorem \ref{thm:main} consists in
showing that $F$ satisfies the hypotheses of Theorem \ref{thm:Henryspecial}.

The hardest part is to prove that a transversality condition holds, namely hypothesis (ii) of Theorem \ref{thm:Henryspecial}. This amounts to proving that the total differential of $F$ is surjective at every zero of $F$. This is the statement of Proposition \ref{prop1sect8}, whose proof relies on the nontrivial results of the whole of Section 7. It is here that the crucial hypothesis \eqref{eq5:intro} on the class of diffeomorphisms considered enters. At this stage we cannot claim that the spectral condition \eqref{eq5:intro} is necessary, but we tend to believe that it is so.

Note that, when $N-1$ is a Neumann eigenvalue of the Laplace--Beltrami operator $-\Delta_{g_0}$ on $D$, the constant
function $\varphi_H$ is a degenerate solution of \eqref{eq:intro-CMC}.

The paper is organized as follows. In Section 2 we fix the functional setting in which we work and provide some preliminary results. In
Sections 3 and 4 we compute and study the derivatives of the map $F$ with respect to the function $\psi$ and to the diffeomorphism $h$,
respectively. In Section 5 we study the kernel of the adjoint of the derivative of $F$ with respect to $\psi$, while in Section 6 we prove
that $F$ satisfies a compactness condition called $\sigma$-properness. In Section 7 we prove what is needed to verify the transversality
condition for the map $F$ at its zeros. Finally, in Section 8 we state the transversality theorem we use and prove Theorem \ref{thm:main}.
In the Appendix we collect some useful preliminaries from Riemannian geometry and some technical lemmas.

\subsection*{Notation and conventions}

\begin{itemize}
\item For a Riemannian metric $g$, we use interchangeably the notations
$g(\cdot,\cdot)$ and $\langle\cdot,\cdot\rangle_g$, and $|\cdot|_g$ to denote the associated pointwise norm.

\item In local coordinates $(x^1,\ldots,x^{N-1})$ we denote by
$$
\partial_i:=\frac{\partial}{\partial x^i}
$$
the $i$-th coordinate vector field, and we write $\partial_{ij}$
for the second-order partial derivatives. 

\item We denote by $\nabla$ the Levi-Civita connection of the metric under consideration, which will always be clear from the context, and we set $\nabla_i:=\nabla_{\partial_i}$. We stress that a subscript consisting of a vector field or of an index always refers to the connection, whereas a subscript consisting of a metric, as in $\nabla_g$, refers to the gradient.

\item We denote by $d\mu_g$ the Riemannian volume measure. If $D$ is a domain with $C^1$ boundary, $d\sigma_g$ denotes the hypersurface measure induced by $g$ on $\partial D$.

\item If $D$ is a domain with $C^1$ boundary, $\nu_g$ denotes the outward unit normal to $\partial D$ with respect to $g$, and
$$
\partial_{\nu_g}u:=du(\nu_g)= \langle\nabla_g u,\nu_g\rangle_g.
$$
The tangential gradient of $u$ along $\partial D$ is denoted by
$\nabla_g^\top u$; thus
$$
\nabla_g^\top u=\nabla_g u-\partial_{\nu_g}u\,\nu_g.
$$
\item We use the Einstein summation convention for repeated upper and lower indices.
\end{itemize}

\section{Functional setting and preliminary results}

Throughout the paper, we fix $N\geq 3$ and $\alpha\in(0,1)$, we denote by $\SN$ the unit sphere of $\RN$, by $\langle\cdot,\cdot\rangle$ the Euclidean scalar product of $\RN$, and by $g_0$ the standard metric on $\SN$, which is the restriction of $\langle\cdot,\cdot\rangle$ to each tangent space $T_p\SN$. We also fix a proper open set $Q\subset\SN$ and a domain $D_0\subset\SN$ of class $C^{3,\alpha}$ such that
\begin{equation}\label{eq:standing-assumptions-D0}
D_0\Subset Q,
\end{equation}
that is, $\overline{D_0}$ is a compact subset of $Q$.

Let $p_* \in \SN\setminus Q$ and let 
\beq\label{eq:defsterproj}
\pi:\SN\setminus\{p_*\}\to \R^{N-1}
\eeq
be the stereographic projection from the pole $p_*$.
We fix once and for all the atlas of $\SN$ given by the two stereographic charts $\pi$ and $\pi_-$, from the poles $p_*$ and $-p_*$ respectively, and we define accordingly, in the standard way, the H\"older and Sobolev spaces of functions, vector fields and tensor fields on $\SN$ and on $\overline{D}$, $\partial D$, for every domain $D\Subset Q$ of class $C^{3,\alpha}$. 
Note that, since $\overline{D_0}\subset Q \subset\SN\setminus\{p_*\}$, all local computations on $\overline{D_0}$ can be, and will be, performed in the single chart $\pi$. Moreover, we set
\beq\label{eq:defomega0}
\Omega_0:=\pi(D_0).
\eeq
Since $\overline{D_0}$ is a compact subset of $\SN\setminus\{p_*\}$, $D_0$ is of class $C^{3,\alpha}$ and $\pi$ is smooth, $\Omega_0\subset\R^{N-1}$ is a bounded domain of class $C^{3,\alpha}$.

We consider the closed linear subspace of $C^{3,\alpha}(\SN,T\SN)$ defined by
$$
\Y:= \left\{ Y \in C^{3,\alpha}(\SN,T\SN);\, Y=0\text{ on }\SN\setminus Q\right\}
$$
endowed with the usual $C^{3,\alpha}$ norm, where  $T\SN$ is the tangent bundle of $\SN$. In particular, $\Y$ is a Banach space.

We denote by $\DiffQ$ the set of $C^{3,\alpha}$ diffeomorphisms of $\SN$ onto itself that coincide with the identity on $\SN\setminus Q$, namely
$$
\DiffQ:=\left\{h\in\operatorname{Diff}^{3,\alpha}(\SN);\, h=\Id\text{ on }\SN\setminus Q\right\}.
$$
$\DiffQ$ carries a natural Banach-manifold structure, which we briefly describe here, limiting ourselves to the features needed in the sequel.

For $h\in\DiffQ$, let $h^*T\SN$ denote the pull-back of $T\SN$ by $h$, whose fiber over $p\in\SN$ is $T_{h(p)}\SN$, and set
$$
\Y_h:=\bigl\{V\in C^{3,\alpha}(\SN,h^*T\SN);\ V=0\ \text{on }\SN\setminus Q\bigr\}.
$$
Then $\Y_h$ is a Banach space, and $I_h:\Y\to\Y_h$, $I_h(Y):=Y\circ h$, is a Banach-space isomorphism.

As local parametrizations  we use the exponential map of $g_0$, denoted by $\exp$, where throughout we identify $T_p\SN=\{v\in\RN;\ \langle v,p\rangle=0\}$. By \cite[Proposition 5.27]{LeeRM}, the nonconstant maximal geodesics of $(\SN,g_0)$ are precisely the periodic constant-speed curves whose image is a great circle, and $(\SN,g_0)$ is geodesically complete. In particular $\exp$ is defined on the
whole of $T\SN$. For $v\in T_p\SN\setminus\{0\}$ the great circle through $p$ tangent to $v$ is $\SN\cap\operatorname{span}\{p,v/|v|\}$, and its constant-speed parametrization with initial point $p$ and initial velocity $v$ is $t\mapsto\cos(|v|t)\,p+\sin(|v|t)\,v/|v|$. Evaluating at $t=1$, we obtain
\beq\label{eq:expexplicit}
\exp_pv=\cos|v|\;p+\frac{\sin|v|}{|v|}\,v, \qquad p\in\SN,\ v\in T_p\SN,
\eeq
with the convention $\sin|v|/|v|=1$ for $v=0$. Since $t\mapsto\cos t$ and $t\mapsto\sin t/t$ are entire and even, the right-hand side of
\eqref{eq:expexplicit} is a linear combination of $p$ and $v$ whose coefficients are entire functions of $|v|^2$. In particular it defines a $C^\infty$ map $E:\RN\times\RN\to\RN$ whose restriction to $T\SN$ takes values in $\SN$ and coincides with $\exp$. Moreover, differentiating \eqref{eq:expexplicit} along $t\mapsto tv$ at $t=0$ we obtain the well known property
\beq\label{eq:dexp0}
d(\exp_p)_0=\mathrm{id}_{T_p\SN},\qquad p\in\SN .
\eeq
For $V\in\Y_h$ we then set
\beq\label{defThetahV}
\Theta_h(V)(p):=\exp_{h(p)}\bigl(V(p)\bigr)=E\bigl(h(p),V(p)\bigr),\qquad p\in\SN,
\eeq
so that $\Theta_h(0)=h$. Since $E$ is smooth and $h,V$ are of class
$C^{3,\alpha}$, the map $\Theta_h(V):\SN\to\SN$ is of class $C^{3,\alpha}$ for
every $V\in\Y_h$. Since $E$ is smooth, $\Theta_h$ is continuous from $\Y_h$ into $C^1(\SN,\SN)$, endowed with the $C^1$ topology. Then, as the set of $C^1$ diffeomorphisms of $\SN$ is open in $C^1(\SN,\SN)$ (see \cite[Chapter 2, Theorem 1.7]{Hirsch}), and $\Theta_h(0)=h$, there is an open neighborhood $\mathcal O_h$ of $0$ in $\Y_h$ such that $\Theta_h(V)$ is a $C^1$ diffeomorphism of $\SN$ for every $V\in\mathcal O_h$. 

Since $\Theta_h(V)$ is a $C^{3,\alpha}$ bijection with everywhere invertible differential, the inverse function theorem gives $(\Theta_h(V))^{-1}\in C^{3,\alpha}$, so that $\Theta_h(V)\in\operatorname{Diff}^{3,\alpha}(\SN)$. 
Shrinking $\mathcal O_h$ if necessary, we may further assume that $\max_{\SN}|V|_{g_0}<\pi$ for every $V\in\mathcal O_h$. 
By \eqref{eq:expexplicit} we have $\langle p,\exp_pv\rangle=\cos|v|$, so that if $\exp_pv=\exp_pw$ with $|v|,|w|<\pi$ then $|v|=|w|$, and hence $v=w$ by \eqref{eq:expexplicit} again, since $\sin r/r\neq0$ for $r\in[0,\pi)$. 
Thus each $\exp_p$ is injective on $\{v\in T_p\SN;\ |v|<\pi\}$, and it follows that $\Theta_h$ is one-to-one on $\mathcal O_h$.

Furthermore, the constraint defining $\DiffQ$ is preserved by $\Theta_h$: if $p\in\SN\setminus Q$, then $h(p)=p$ and
$V(p)=0$, so that $\Theta_h(V)(p)=\exp_p(0)=p$. 

Hence $\Theta_h$ maps
$\mathcal O_h$ bijectively onto a subset $\mathcal U_h:=\Theta_h(\mathcal O_h)$ of $\DiffQ$ containing $h$, and
$$\DiffQ=\bigcup_{h\in\DiffQ}\mathcal U_h.$$
Setting
$$
\phi_h:=I_h^{-1}\circ\Theta_h^{-1}: \mathcal U_h\longrightarrow I_h^{-1}(\mathcal O_h)\subset\Y,
$$
we can check that the pairs $(\mathcal U_h,\phi_h)$ define on $\DiffQ$ a structure of $C^1$ Banach manifold modeled on $\Y$. 
We omit the details of the proof, but we only observe that the $C^1$ regularity of the transition maps can be proved directly since the inverse $\Theta_h^{-1}$ can also be computed explicitly (actually, for the analogous structure on $C^k(\SN,\SN)$, these transition maps are smooth, see \cite[Theorem 4.4]{Wittmann}, however higher regularity is not needed here). Let us also note that for $h=\Id$ we have $\phi_{\Id}=\Theta_{\Id}^{-1}$.

Since $I_h$ is a Banach-space isomorphism, in the computations it is convenient to work with the parametrization $\Theta_h$ itself, thereby identifying
\beq\label{eq:tangent_space_DiffQ}
T_h\DiffQ=\Y_h=\bigl\{Y\circ h;\ Y\in\Y\bigr\}, \qquad\text{so that}\qquad T_{\Id}\DiffQ=\Y .
\eeq
This is the identification used throughout. Finally, let $t\mapsto V_t$ be a differentiable curve in $\Y_h$ with $V_0=0$, set
$\dot V_0:=\frac{d}{dt}\big|_{t=0}V_t$ and $h_t:=\Theta_h(V_t)$, which is well defined for $|t|$ small. 
Then $t\mapsto h_t$ is a differentiable curve in $\DiffQ$ with $h_0=h$ and $\frac{d}{dt}\big|_{t=0}h_t=\dot V_0$, which can also be computed pointwise: by \eqref{eq:dexp0}, for every $p\in\SN$,
\beq\label{eq:pointwisetangent}
\left.\frac{d}{dt}\right|_{t=0}h_t(p) =\left.\frac{d}{dt}\right|_{t=0}\exp_{h(p)}\bigl(V_t(p)\bigr)=d\bigl(\exp_{h(p)}\bigr)_{0}\bigl[\dot V_0(p)\bigr]
=\dot V_0(p).
\eeq
\medskip

For $h\in\DiffQ$ we set
$$D_h:=h(D_0), \qquad g_h:=h^*g_0,$$
where $h^*g_0$ is the pull-back of $g_0$ by $h$ (see \eqref{eq:defpullbackmetric} for the definition).
Since $h=\Id$ on $\SN\setminus Q$, we have $h(Q)=Q$ and hence $D_h\Subset Q$.

Let $H\in\R$ and consider, on $D_h$, the problem
\begin{equation}\label{eq:probh}
\begin{cases}
-\divg_{g_0}\!\left(\dfrac{\nabla_{g_0}\varphi}{\sqrt{1+|\nabla_{g_0}\varphi|_{g_0}^2}}\right)+\dfrac{N-1}{\sqrt{1+|\nabla_{g_0}\varphi|_{g_0}^2}}
-(N-1)He^\varphi=0&\text{in }D_h,\\[1.2ex]
\dfrac{\partial\varphi}{\partial\nu_{g_0}}=0&\text{on }\partial D_h.
\end{cases}
\end{equation}
\begin{remark}\label{rem:Hpositivenecessary}
Note that for any proper domain $D$ of $\SN$ of class $C^1$, a necessary condition for the solvability of the equation \eqref{eq:intro-CMC} (and thus in particular of \eqref{eq:probh}) is that
$$
H>0.
$$
Indeed, if $\varphi \in C^2(\overline D)$ is a solution of \eqref{eq:intro-CMC}, setting
$$W_\varphi:=\sqrt{1+|\nabla_{g_0}\varphi|_{g_0}^2},$$ 
integrating the equation over $D$ and applying the divergence theorem, we obtain
$$
-\int_{D}\divg_{g_0}\!\left(\frac{\nabla_{g_0}\varphi}{W_\varphi}\right)\,d\mu_{g_0}+(N-1)\int_{D}\frac{1}{W_\varphi}\,d\mu_{g_0}
-(N-1)H\int_{D}e^\varphi\,d\mu_{g_0}=0.
$$
Since
$$
\int_{D}\divg_{g_0}\!\left(\frac{\nabla_{g_0}\varphi}{W_\varphi}\right)\,d\mu_{g_0}=\int_{\partial D}\frac{\partial_{\nu_{g_0}}\varphi}{W_\varphi}\,d\sigma_{g_0}=0,$$
it follows that
$$
\int_{D}\frac{1}{W_\varphi}\,d\mu_{g_0}=H\int_{D}e^\varphi\,d\mu_{g_0}.$$
Now $W_\varphi\ge 1$, hence $W_\varphi^{-1}>0$ in $D$, and clearly $e^\varphi>0$ in $D$. Therefore both integrals are strictly positive, and we conclude that $H$ is positive.
\end{remark}

The following lemma allows us to pull back problem
\eqref{eq:probh} from the moving domain $D_h$ to the fixed domain
$D_0$, and is therefore central to our strategy.
\begin{lemma}\label{lem1:sect2}
Let $h\in\DiffQ$ and $\varphi\in C^2(\overline{D_h})$, and set $\psi:=\varphi\circ h$. Then $\varphi$ solves \eqref{eq:probh} if and only if $\psi$ solves
\begin{equation}\label{eq:prob0}
\begin{cases}
-\divg_{g_h}\!\left(\dfrac{\nabla_{g_h}\psi}{\sqrt{1+|\nabla_{g_h}\psi|_{g_h}^{2}}}\right) +\dfrac{N-1}{\sqrt{1+|\nabla_{g_h}\psi|_{g_h}^{2}}}
-(N-1)He^\psi=0&\text{in }D_0,\\[1.2ex]
\dfrac{\partial\psi}{\partial\nu_{g_h}}=0&\text{on }\partial D_0.
\end{cases}
\end{equation}
\end{lemma}
\begin{proof}
By the definition of $g_h$, the map
$$
h:(D_0,g_h)\longrightarrow(D_h,g_0)
$$
is an isometry. In particular, we have
$$
dh(\nu_{g_h})=\nu_{g_0}\circ h \qquad\text{on }\partial D_0,
$$
and hence
$$
\partial_{\nu_{g_h}}(\varphi\circ h)=(\partial_{\nu_{g_0}}\varphi)\circ h,
$$
from which we readily deduce that $\partial_{\nu_{g_0}}\varphi=0$ on $\partial D_h$ if and only if $\partial_{\nu_{g_h}}\psi=0$ on $\partial D_0$.

Moreover, denoting by $h^*Z:=(dh)^{-1}\bigl(Z\circ h\bigr)$ the pull-back by $h$ of a vector field $Z$ on $D_h$, and as $h$ is an isometry, we check that
\beq\label{eq1:addlem2:2}
\nabla_{g_h}(\varphi\circ h)=h^*\bigl(\nabla_{g_0}\varphi\bigr),\qquad
\bigl|\nabla_{g_h}(\varphi\circ h)\bigr|_{g_h}=|\nabla_{g_0}\varphi|_{g_0}\circ h,
\eeq
and
\beq\label{eq2:addlem2:2}
(\divg_{g_0}Z)\circ h=\divg_{g_h}(h^*Z).
\eeq
Then, applying \eqref{eq2:addlem2:2} with
$Z=\nabla_{g_0}\varphi/\sqrt{1+|\nabla_{g_0}\varphi|_{g_0}^2}$, and exploiting \eqref{eq1:addlem2:2}, we readily infer that $\varphi$ solves \eqref{eq:probh} if and only if
$\psi$ solves \eqref{eq:prob0}. The proof is complete.
\end{proof}

Let $\Param$ be an open neighborhood of $\Id$ in $\DiffQ$. For $(\psi,h)\in C^{2,\alpha}(\overline{D_0})\times\Param$ let
\begin{equation}\label{eq:notationsimplification}
W_{\psi,h}:=\sqrt{1+|\nabla_{g_h}\psi|_{g_h}^2},
\end{equation}
and define the map
\begin{equation}\label{eq:defF}
F:C^{2,\alpha}(\overline{D_0})\times\Param\longrightarrow C^{0,\alpha}(\overline{D_0})
\times C^{1,\alpha}(\partial D_0),
\qquad
F(\psi,h):=(F_1(\psi,h),F_2(\psi,h)),
\end{equation}
where
\begin{eqnarray*}
F_1(\psi,h) &:=&-\divg_{g_h} \left(\frac{\nabla_{g_h}\psi}{W_{\psi,h}}\right)+\frac{N-1}{W_{\psi,h}}-(N-1)He^\psi,\\
F_2(\psi,h)&:=&\partial_{\nu_{g_h}}\psi.
\end{eqnarray*}
By Lemma \ref{lem1:sect2}, $F(\psi,h)=(0,0)$ if and only if
$\varphi=\psi\circ h^{-1}$ solves \eqref{eq:probh}. We first establish
the regularity of $F$.

\begin{proposition}\label{prop:FC1}
The map $F:C^{2,\alpha}(\overline{D_0})\times\Param\to C^{0,\alpha}(\overline{D_0})
\times C^{1,\alpha}(\partial D_0)$ is of class $C^1$.
\end{proposition}

\begin{proof}
We first show that the map $h\mapsto g_h$, from $\Param$ into $C^{2,\alpha}(\overline{D_0},T^*\SN\otimes T^*\SN)$, is of class $C^1$. We work in the chart $\pi$ (see \eqref{eq:defsterproj}) and define $\Omega_0$ as in \eqref{eq:defomega0}. 
In the corresponding local coordinates $x=(x^1,\ldots,x^{N-1})$ the components of the pull-back metric $g_h=h^*g_0$ are
\begin{equation}\label{eq:gh-extrinsic}
(g_h)_{ij}(x)=\bigl\langle\partial_ih(x),\,\partial_jh(x)\bigr\rangle,
\end{equation}
where $\partial_ih(x)$ stands for $\frac{\partial}{\partial x^i}(h\circ\pi^{-1})(x)$,
regarded as a vector of $\RN$.

Let us fix $\bar h\in\Param$, let $h\in\mathcal U_{\bar h}\cap\Param$ and let $V\in\mathcal O_{\bar h}\subset\Y_{\bar h}$ be its representative under the parametrization $\Theta_{\bar h}$ (see \eqref{defThetahV}), so that $h=E(\bar h,V)$ on $\overline{\Omega_0}$. By \eqref{eq:expexplicit}, $E(q,v)$ is a linear combination of $q$ and $v$ whose coefficients are entire functions of $|v|^2$.
Since $C^{3,\alpha}(\overline{\Omega_0})$ is a Banach algebra and $\bar h\in C^{3,\alpha}$, it follows that $V\mapsto h$ is of class $C^1$ from
$\mathcal O_{\bar h}$ into $C^{3,\alpha}(\overline{\Omega_0},\RN)$.

Consequently $V\mapsto\partial_ih$ is of class $C^1$ with values in $C^{2,\alpha}(\overline{\Omega_0},\RN)$ for every $i=1,\ldots,N-1$ and, using again that $C^{2,\alpha}(\overline{\Omega_0})$ is a Banach algebra, \eqref{eq:gh-extrinsic} yields that $V\mapsto(g_h)_{ij}$ is of class $C^1$ as well. Since $\bar h\in\Param$ is arbitrary, we conclude that $h\mapsto g_h$ is of class $C^1$ from $\Param$ into $C^{2,\alpha}(\overline{D_0},T^*\SN\otimes T^*\SN)$.

Since $\overline{D_0}$ is compact and $g_h$ is a Riemannian metric, then $\det g_h$ is bounded away from zero, and, as the determinant and the adjoint matrix, denoted respectively by $\det g_h$ and $\operatorname{adj}g_h$, depend polynomially on $g_h$, we infer that
$$
h\longmapsto\det g_h, \qquad h\longmapsto g_h^{-1}=\frac{\operatorname{adj}g_h}{\det g_h}
$$
are of class $C^1$ from $\Param$ into $C^{2,\alpha}(\overline{D_0})$, $C^{2,\alpha}(\overline{D_0},T\SN\otimes T\SN)$, respectively. Now, by the expression of the gradient and the divergence in local coordinates (see \eqref{eq1bis:appendix1} and \eqref{eq:divlocalcoordinates}) we have
$$
\nabla_{g_h}\psi=g_h^{ij}\,\partial_j\psi\,\partial_i, \quad \divg_{g_h}Z=\partial_iZ^i+\frac{\partial_i(\det g_h)}{2\det g_h}\,Z^i .
$$
Then, using again that $C^{k,\alpha}$ is a Banach algebra, we conclude that
$$
(\psi,h)\longmapsto\nabla_{g_h}\psi,
\qquad
(Z,h)\longmapsto\divg_{g_h}Z
$$
are of class $C^1$, from $C^{2,\alpha}(\overline{D_0})\times\Param$ to $C^{1,\alpha}(\overline{D_0},T\SN)$, and from
$C^{1,\alpha}(\overline{D_0},T\SN)\times\Param$ to $C^{0,\alpha}(\overline{D_0})$, respectively. Moreover,
since $1+|\nabla_{g_h}\psi|_{g_h}^2\ge1$ and $t\mapsto(1+t)^{-1/2}$ is smooth on
$[0,\infty)$, also
$$
(\psi,h)\longmapsto
W_{\psi,h}^{-1}=\bigl(1+|\nabla_{g_h}\psi|_{g_h}^2\bigr)^{-1/2},
$$
is of class $C^1$. Hence
$$(\psi,h)\longmapsto\frac{\nabla_{g_h}\psi}{W_{\psi,h}}$$
is of class $C^1$ from $C^{2,\alpha}(\overline{D_0})\times\Param$ to $C^{1,\alpha}(\overline{D_0},T\SN)$.
Summing up, taking into account that $\psi\mapsto e^{\psi}$ is of class $C^1$ from $C^{2,\alpha}(\overline{D_0})$ into
$C^{0,\alpha}(\overline{D_0})$, we conclude that
$$
F_1:C^{2,\alpha}(\overline{D_0})\times\Param\longrightarrow
C^{0,\alpha}(\overline{D_0})
$$
is of class $C^1$.

We now prove the regularity of $F_2$. Since $\partial D_0$ is of class
$C^{3,\alpha}$, we may fix a defining function $\rho\in C^{3,\alpha}(\SN)$ for
$D_0$, that is,
\beq\label{eq:deffunctrho}
D_0=\{p\in\SN;\,\rho(p)>0\}, \qquad\partial D_0=\{p\in\SN;\,\rho(p)=0\}\quad\text{and }\ d\rho\neq0 \qquad\text{on }\partial D_0 .
\eeq
With this sign convention, the outward unit normal is
$$
\nu_{g_h}=-\frac{\nabla_{g_h}\rho}{|\nabla_{g_h}\rho|_{g_h}} \qquad\text{on }\partial D_0 .
$$
As above, $h\mapsto q_h:=g_h^{ij}\,\partial_i\rho\,\partial_j\rho$ is of class $C^1$ with values in $C^{2,\alpha}(\partial D_0)$. Moreover $q_h$ is bounded away from zero, since $g_h^{-1}$ is positive definite and $d\rho\neq0$ on the compact set $\partial D_0$. 
As $t\mapsto t^{-1/2}$ is smooth on $(0,\infty)$, it follows that
$$
h\longmapsto\nu_{g_h}
$$
is of class $C^1$, with values in $C^{2,\alpha}(\partial D_0,T\SN)$. Finally, since
$$
F_2(\psi,h)=d\psi\bigl(\nu_{g_h}\bigr)
\qquad\text{on }\partial D_0,
$$
and $\psi\mapsto d\psi|_{\partial D_0}$ is a bounded linear operator, we conclude that $F_2$
is of class $C^1$ from $C^{2,\alpha}(\overline{D_0})\times\Param$ to $C^{1,\alpha}(\partial D_0)$. The proof is complete.
\end{proof}

\section{Derivative with respect to $\psi$  and Fredholm property}
In this section we compute the partial derivative of the map $F=(F_1, F_2)$ defined in \eqref{eq:defF} with respect to $\psi$ and prove that $D_\psi F$ is a Fredholm operator of index zero.
\begin{lemma}\label{lem:DFpsi}
For any $(\psi,h) \in C^{2,\alpha}(\overline{D_0})\times \Param$ it holds
\begin{equation}\label{eq:derF1psi}
\begin{array}{lll}
\displaystyle D_\psi F_1(\psi,h)[v]&=&\displaystyle
-\divg_{g_h}\left(\frac{\nabla_{g_h}v}{\sqrt{1+|\nabla_{g_h}\psi|_{g_h}^2}}
-\frac{\nabla_{g_h}\psi\,\langle \nabla_{g_h}\psi,\nabla_{g_h}v\rangle_{g_h}}{(1+|\nabla_{g_h}\psi|_{g_h}^2)^{3/2}}\right)\\[21pt]
&&\displaystyle-\frac{N-1}{(1+|\nabla_{g_h}\psi|_{g_h}^2)^{3/2}}\langle \nabla_{g_h}\psi,\nabla_{g_h}v\rangle_{g_h} -(N-1)He^{\psi}v
\end{array}
\end{equation}
\beq\label{eq:derF2psi}
D_\psi F_2(\psi,h)[v]= \frac{\partial v}{\partial \nu_{g_h}}, \quad\quad\quad\quad\quad\quad\quad\quad\quad\quad\quad\quad\quad\quad\quad\quad\quad\quad\quad
\eeq
for any $v\in C^{2,\alpha}(\overline{D_0})$, where $g_h=h^*g_0$ is the pull-back of $g_0$ by $h$.
\end{lemma}
\begin{proof}
We first observe that for any fixed $h\in\Param$ the map $\psi\mapsto F_2(\psi,h)= \frac{\partial \psi}{\partial \nu_{g_h}}$ is linear, and thus \eqref{eq:derF2psi} is trivial. Let us prove \eqref{eq:derF1psi}.

For any $\psi\in C^{2,\alpha}(\overline{D_0})$ and $h\in\Param$ we easily check that
$$\frac{\partial}{\partial \psi}\bigl(W_{\psi,h}\bigr)[v]=\frac{\langle \nabla_{g_h}\psi,\nabla_{g_h} v\rangle_{g_h}}{W_{\psi,h}},
\qquad \frac{\partial}{\partial \psi}\left(\frac{1}{W_{\psi,h}}\right)[v]=-\frac{\langle \nabla_{g_h}\psi,\nabla_{g_h} v\rangle_{g_h}}{W_{\psi,h}^3},
$$
where $W_{\psi,h}$ is defined by \eqref{eq:notationsimplification}. Then, as the divergence (at fixed metric) is linear, it follows that
\begin{equation}\label{eq1:lem1sect3}
\frac{\partial}{\partial \psi}\left(\divg_{g_h}\left(\frac{\nabla_{g_h}\psi}{W_{\psi,h}}\right)\right)[v]
=\divg_{g_h}\left(\frac{\nabla_{g_h} v}{W_{\psi,h}} - \frac{\nabla_{g_h}\psi}{W_{\psi,h}^3}\langle \nabla_{g_h}\psi,\nabla_{g_h} v\rangle_{g_h}
\right).
\end{equation}
Moreover, for the lower order terms, taking into account that $H$ is a constant, we get
\begin{equation}\label{eq2:lem1sect3}
\displaystyle \frac{\partial}{\partial \psi}\left(\frac{N-1}{W_{\psi,h}}\right)[v]= \displaystyle -\frac{N-1}{W_{\psi,h}^3}\langle \nabla_{g_h}\psi,\nabla_{g_h} v\rangle_{g_h}, \quad
\displaystyle \frac{\partial}{\partial \psi}\bigl((N-1)He^{\psi}\bigr)[v] =\displaystyle (N-1)He^{\psi}v.
\end{equation}
Then putting together \eqref{eq1:lem1sect3} and \eqref{eq2:lem1sect3} we readily conclude. The proof is complete.
\end{proof}

\begin{remark}\label{rem:unifelliptApsi}
For any fixed $(\psi,h) \in C^{2,\alpha}(\overline{D_0})\times \Param$ the second-order linear partial differential operator $$L_{\psi,h}:C^{2,\alpha}(\overline{D_0})\to C^{0,\alpha}(\overline{D_0})$$ defined by
\begin{equation}\label{eq:defLpsih}
L_{\psi,h}v:=  D_\psi F_1(\psi,h)[v]
\end{equation}
is uniformly elliptic in $D_0$, with respect to $g_h$. 

Indeed, in view of \eqref{eq:derF1psi}, the second-order part, in divergence form, of $L_{\psi,h}$ is
$$
P_{\psi,h}v:=-\divg_{g_h}\bigl(A_{\psi,h}\nabla_{g_h}v\bigr),
$$
where  $A_{\psi,h}$ is the $g_h$-symmetric $(1,1)$-tensor field defined by
\beq\label{eq:defApsih}
A_{\psi,h}X:=\frac{X}{W_{\psi,h}}-\frac{\langle\nabla_{g_h}\psi,X\rangle_{g_h}}{W_{\psi,h}^3}\nabla_{g_h}\psi
\eeq
for every vector field $X$, and we readily check that, for any $p\in D_0$, for any $\xi\in T_pD_0$
$$
\langle A_{\psi,h}(p)\xi,\xi\rangle_{g_h(p)}\ge
\frac{|\xi|_{g_h(p)}^2}{(1+\kappa_{\psi,h}^2)^{3/2}},
$$
where 
$$
\kappa_{\psi,h}:= \max_{\overline{D_0}}|\nabla_{g_h}\psi|_{g_h}.
$$
\end{remark}

We now conclude this section with the following result.

\begin{proposition}\label{prop:Fredholm-general-parameter}
Let $(\psi,h)\in C^{2,\alpha}(\overline{D_0})\times\Param$ and consider
the linear operator
$$
D_\psi F(\psi,h):C^{2,\alpha}(\overline{D_0}) \longrightarrow C^{0,\alpha}(\overline{D_0}) \times C^{1,\alpha}(\partial D_0)
$$
defined by
$$
D_\psi F(\psi,h)v=\bigl(L_{\psi,h}v,\partial_{\nu_{g_h}}v\bigr),
$$
where
$$
L_{\psi,h}v=D_\psi F_1(\psi,h)[v]
$$
is given by \eqref{eq:defLpsih}. Then $D_\psi F(\psi,h)$ is a Fredholm
operator of index zero.
\end{proposition}

\begin{proof}
We work in the chart given by the stereographic projection $\pi$ from a pole $p_*\in \SN\setminus Q$ (see \eqref{eq:defsterproj}) and we define $\Omega_0$ as in \eqref{eq:defomega0}.
By Remark \ref{rem:unifelliptApsi}, the local coordinate expressions for the gradient and the divergence (see \eqref{eq1bis:appendix1} and \eqref{eq:divlocalcoordinates}), and since $g_h=h^*g_0$, we infer that, in local coordinates, ${D_\psi F}(\psi,h)$ is represented by
$$
\bigl(\widetilde L_{\psi,h},\widetilde M_h\bigr),
$$
where
$$
\widetilde L_{\psi,h}v=-\widetilde a^{ij}\partial_{ij}v+\widetilde b^i\partial_{i}v+\widetilde cv
$$
is a uniformly elliptic second-order linear operator in $\Omega_0$, with
$$
\widetilde a^{ij}\in C^{1,\alpha}(\overline{\Omega_0}),
\qquad
\widetilde b^i,\widetilde c
\in C^{0,\alpha}(\overline{\Omega_0}),
$$
for all $i,j=1,\ldots,N-1$, and
\begin{equation}\label{eq:Mpsiform}
\widetilde M_hv =\widetilde\beta\cdot \nabla v
\end{equation}
is a uniformly oblique first-order boundary operator, with
$$
\widetilde\beta
\in C^{1,\alpha}
(\partial\Omega_0,\mathbb R^{N-1}),
$$
where $\cdot$ and $\nabla$ denote here, respectively, the standard Euclidean scalar product and the gradient in $\mathbb R^{N-1}$. We point out that, in general, $\widetilde M_h$ is not a Neumann boundary operator.

It is standard to see that taking $\tau>0$ sufficiently large, the operator
$$
(\widetilde L_{\psi,h},\widetilde M_h)+\tau\mathcal K
$$
is an isomorphism, where
$$
\mathcal K:C^{2,\alpha}(\overline{\Omega_0})\longrightarrow C^{0,\alpha}(\overline{\Omega_0})\times C^{1,\alpha}(\partial\Omega_0)
$$
is defined by
$$
\mathcal Kv:=\bigl(v,v|_{\partial\Omega_0}\bigr).
$$
Indeed, with the sign conventions of \cite{LiebermanODP}, for $\tau$ large the pair $-(\widetilde L_{\psi,h},\widetilde M_h)-\tau\mathcal K$ satisfies the hypotheses of \cite[Theorem 2.27]{LiebermanODP}, the shift making the zero order coefficients of the interior and of the boundary operators strictly negative.

Since $\mathcal K$ is compact, by the compact H\"older embeddings, $(\widetilde L_{\psi,h},\widetilde M_h)$ is a compact perturbation of an isomorphism. Hence it is Fredholm of index zero. 
Finally, since $u\mapsto u\circ\pi^{-1}$ is an isomorphism of Banach spaces, and conjugation by isomorphisms preserves the Fredholm property and the index, the same conclusion holds for $D_\psi F(\psi,h)$. The proof is complete.
\end{proof}

\section{Derivative with respect to $h$}

In this section we compute the derivative of $F$ with respect to the domain parameter $h$. We begin with the derivative of $F$ at
$(\psi,\Id)$. We first establish a general first-variation formula for pull-back metrics.

\begin{lemma}\label{lem1:sect4}
Let  $t\mapsto h_t$ be a differentiable curve in $\DiffQ$ that satisfies
\beq\label{eqstat:lem1:sect4}
h_0=\Id, \qquad \left.\frac{d}{dt}\right|_{t=0}h_t=Y\in\Y.
\eeq
Then
\begin{equation}\label{eq:first-variation-pull-back-metric}
\left.\frac{d}{dt}\right|_{t=0}h_t^*g_0=\mathcal L_Yg_0,
\end{equation}
where $h_t^*g_0$ is the pull-back of $g_0$ by $h_t$, and $\mathcal L_Yg_0$ is the Lie derivative of $g_0$ along $Y$ (see \eqref{eq:defLieDerivative}).
\end{lemma}

\begin{proof}
In local coordinates $x=(x^1,\ldots,x^{N-1})$, identifying $\partial_ih_t(x)$, $\partial_jh_t(x)$ with the corresponding vectors in $\RN$, the components of
$h_t^*g_0$ are 
$$
(h_t^*g_0)_{ij}(x)=\langle\partial_ih_t(x),\partial_jh_t(x)\rangle .
$$
Since $h_0=\Id$, we have $\partial_ih_0=\partial_i$, while by \eqref{eqstat:lem1:sect4} it holds  $\left.\frac{d}{dt}\right|_{t=0}\partial_ih_t=\partial_iY$. Hence
$$
\left.\frac{d}{dt}\right|_{t=0}(h_t^*g_0)_{ij}=\langle\partial_iY,\partial_j\rangle+\langle\partial_i,\partial_jY\rangle .
$$
Since $g_0$ is the metric induced by $\langle\cdot,\cdot\rangle$, the Levi-Civita
connection of $g_0$ is the tangential part of the derivative in $\RN$ (see
\cite[Theorem 8.2]{LeeRM}), so that
$\langle\partial_iY,\partial_j\rangle=g_0(\nabla_{\partial_i}Y,\partial_j)$, the
coordinate vector fields being tangent to $\SN$. Therefore
$$
\left.\frac{d}{dt}\right|_{t=0}(h_t^*g_0)_{ij}=g_0(\nabla_{\partial_i}Y,\partial_j)+g_0(\partial_i,\nabla_{\partial_j}Y)=(\mathcal L_Yg_0)_{ij},
$$
where the last equality follows from \eqref{eq:liederiv} applied to $\partial_i$ and $\partial_j$.
This proves \eqref{eq:first-variation-pull-back-metric} and the proof is complete.
\end{proof}

We now apply Lemma \ref{lem1:sect4} to the curves in $\DiffQ$ provided by the exponential parametrization $\Theta_{\Id}$ centered at
$\Id$ (see \eqref{defThetahV}), and derive some variation formulas. For $Y\in\Y$ we
set
\beq\label{eq:defdeltaYg_0}
\delta_Yg_0:=\mathcal L_Yg_0 .
\eeq

\begin{lemma}\label{lem2sect4}
Let $\psi\in C^{2,\alpha}(\overline{D_0})$ and $Y\in\Y$. For $|t|$ sufficiently small, set
$$
h_t:=\Theta_{\Id}(tY),
\qquad
g_t:=h_t^*g_0,
$$
and
$$
W_{\psi,t}:=\sqrt{1+|\nabla_{g_t}\psi|_{g_t}^2}, \qquad W_\psi:=W_{\psi,0}=\sqrt{1+|\nabla_{g_0}\psi|_{g_0}^2}.
$$
Then
$$
\left.\frac{d}{dt}\right|_{t=0}\nabla_{g_t}\psi=-\bigl(\delta_Yg_0(\nabla_{g_0}\psi,\cdot)\bigr)^\sharp,\quad
\left.\frac{d}{dt}\right|_{t=0}|\nabla_{g_t}\psi|_{g_t}^2=-\delta_Yg_0(\nabla_{g_0}\psi,\nabla_{g_0}\psi),
$$
and
$$
\left.\frac{d}{dt}\right|_{t=0}\left(\frac{1}{W_{\psi,t}}\right)=\frac{\delta_Yg_0(\nabla_{g_0}\psi,\nabla_{g_0}\psi)}{2W_\psi^3},
$$
where $\bigl(\delta_Yg_0(\nabla_{g_0}\psi,\cdot)\bigr)^\sharp$ is the $g_0$-metric dual vector field of the one-form $\delta_Yg_0(\nabla_{g_0}\psi,\cdot)$ (see \eqref{eq:defomegasharp} for the definition).
\end{lemma}
\begin{proof}
Applying \eqref{eq:pointwisetangent} with $h=\Id$ and $V_t=tY$, the
map
\beq\label{eq1lem2sect4}
t\mapsto h_t=\Theta_{\Id}(tY),
\eeq
defined for $|t|$ sufficiently small, is a differentiable curve in $\DiffQ$ and satisfies
\eqref{eqstat:lem1:sect4}. Then, by
Lemma \ref{lem1:sect4},
\beq\label{eq2:lem2sect4}
\left.\frac{d}{dt}\right|_{t=0}g_t=\mathcal L_Yg_0=\delta_Yg_0.
\eeq
Moreover, as seen in the proof of Proposition \ref{prop:FC1}, the maps $t\mapsto\nabla_{g_t}\psi$ and
$t\mapsto W_{\psi,t}^{-1}$ are differentiable curves in
$C^{1,\alpha}(\overline{D_0},T\SN)$ and in $C^{1,\alpha}(\overline{D_0})$,
respectively.

Now, for every vector field $Z$ on $D_0$, by definition of gradient (see \eqref{eq:defgradient}) we have
$$
g_t(\nabla_{g_t}\psi,Z)=d\psi(Z).
$$
Then, differentiating at $t=0$, taking into account \eqref{eq2:lem2sect4}, we obtain
$$
\delta_Yg_0(\nabla_{g_0}\psi,Z)+g_0\left(\left.\frac{d}{dt}\right|_{t=0}\nabla_{g_t}\psi,Z\right)
=0.$$
Since $Z$ is arbitrary, by definition of metric dual vector field (see \eqref{eq:defomegasharp}), it follows that
$$
\left.\frac{d}{dt}\right|_{t=0}\nabla_{g_t}\psi=-\bigl(\delta_Yg_0(\nabla_{g_0}\psi,\cdot)\bigr)^\sharp.
$$
Moreover, since
$$
|\nabla_{g_t}\psi|_{g_t}^2=d\psi(\nabla_{g_t}\psi),
$$
differentiating at $t=0$ and using the previous identity, we get
\beq\label{eq3:lem2sect4}
\left.\frac{d}{dt}\right|_{t=0}|\nabla_{g_t}\psi|_{g_t}^2=-\delta_Yg_0(\nabla_{g_0}\psi,\nabla_{g_0}\psi).
\eeq
Finally, differentiating
$$
W_{\psi,t}^{-1}=\left(1+|\nabla_{g_t}\psi|_{g_t}^2\right)^{-1/2}
$$
and using \eqref{eq3:lem2sect4}, we obtain
$$
\left.\frac{d}{dt}\right|_{t=0}\left(\frac{1}{W_{\psi,t}}\right)=\frac{\delta_Yg_0(\nabla_{g_0}\psi,\nabla_{g_0}\psi)}{2W_\psi^3}.
$$
The proof is complete.
\end{proof}

We now compute $D_hF(\psi,\Id)$.

\begin{lemma}\label{lem3sect4}
Let $\psi\in C^{2,\alpha}(\overline{D_0})$ and let $Y\in\Y$. It holds
$$
\begin{aligned}
D_hF_1(\psi,\Id)[Y]
=& \divg_{g_0}\left(\frac{\bigl(\delta_Yg_0(\nabla_{g_0}\psi,\cdot)\bigr)^\sharp}{W_\psi}\right)-
\divg_{g_0}\left(\frac{\delta_Yg_0(\nabla_{g_0}\psi,\nabla_{g_0}\psi)}{2W_\psi^3}\nabla_{g_0}\psi\right)\\[3pt]
&-\frac{\left\langle\nabla_{g_0}(\tr_{g_0}\delta_Yg_0),\nabla_{g_0}\psi\right\rangle_{g_0}}{2W_\psi}+\frac{(N-1)\delta_Yg_0(\nabla_{g_0}\psi,\nabla_{g_0}\psi)}{2W_\psi^3},
\end{aligned}
$$
and
$$
D_hF_2(\psi,\Id)[Y]=-\delta_Yg_0 \bigl(\nu_{g_0},\nabla_{g_0}^{\top}\psi\bigr)-\frac12\delta_Yg_0(\nu_{g_0},\nu_{g_0}) \partial_{\nu_{g_0}}\psi,\quad\quad\quad\quad\quad
$$
where $\nu_{g_0}$ is the outward unit normal to $\partial D_0$ with respect to $g_0$, $\nabla_{g_0}^{\top}\psi$ is the tangential
component of $\nabla_{g_0}\psi$ on $\partial D_0$, $\delta_Y g_0$ is defined by \eqref{eq:defdeltaYg_0}, and $\bigl(\delta_Yg_0(\nabla_{g_0}\psi,\cdot)\bigr)^\sharp$, $W_\psi$ are as in Lemma \ref{lem2sect4}.
\end{lemma}

\begin{proof}
Let $Y\in\Y$ and, for $|t|$ sufficiently small, consider the curve $h_t$ defined in \eqref{eq1lem2sect4}.
As seen in the proof of Lemma \ref{lem2sect4}, $h_t$ satisfies \eqref{eqstat:lem1:sect4}. Moreover, since $\Param$ is an open neighborhood of $\Id$, we also have
$h_t\in\Param$ for $|t|$ sufficiently small, and $F$ is of class $C^1$ by Proposition \ref{prop:FC1}. Hence
$$
D_hF_1(\psi,\Id)[Y]=\left.\frac{d}{dt}\right|_{t=0}F_1(\psi,h_t),
\qquad D_hF_2(\psi,\Id)[Y]=\left.\frac{d}{dt}\right|_{t=0}F_2(\psi,h_t).
$$
Let $g_t$, $W_{\psi,t}$ be as in Lemma \ref{lem2sect4}.
By Lemma \ref{lem1:sect4}, it holds
$$
\left.\frac{d}{dt}\right|_{t=0}g_t=\mathcal L_Yg_0=\delta_Yg_0.
$$
Setting for simplicity
$$
X_t:=\frac{\nabla_{g_t}\psi}{W_{\psi,t}}, \qquad \qquad \nu_t:=\nu_{g_t},
$$
we write
$$
F_1(\psi,h_t)=-\divg_{g_t}(X_t)+\frac{N-1}{W_{\psi,t}}-(N-1)He^\psi
$$
and
$$
F_2(\psi,h_t)=d\psi(\nu_t).
$$

\medskip

We first compute the derivative of $F_1$.\\

By Lemma \ref{lem2appendix}, we have
$$
\left.\frac{d}{dt}\right|_{t=0} \divg_{g_t}(X_t)=\divg_{g_0}(\dot X_0)+\frac12\left\langle\nabla_{g_0}(\tr_{g_0}\delta_Yg_0), X_0\right\rangle_{g_0},
$$
where
$$
\dot X_0:=\left.\frac{d}{dt}\right|_{t=0}X_t.
$$
Using the product rule and Lemma \ref{lem2sect4}, we obtain
$$
\dot X_0=-\frac{\bigl(\delta_Yg_0(\nabla_{g_0}\psi,\cdot)\bigr)^\sharp}{W_\psi}
+\frac{\delta_Yg_0(\nabla_{g_0}\psi,\nabla_{g_0}\psi)}{2W_\psi^3}\nabla_{g_0}\psi
$$
and
$$
\left.\frac{d}{dt}\right|_{t=0}\left(\frac{1}{W_{\psi,t}}\right)=\frac{\delta_Yg_0(\nabla_{g_0}\psi,\nabla_{g_0}\psi)}{2W_\psi^3}.
$$
Since
$$
X_0=\frac{\nabla_{g_0}\psi}{W_\psi},
$$
combining these identities and observing that $He^\psi$ is independent
of $t$, we obtain
$$
\begin{aligned}
D_hF_1(\psi,\Id)[Y]
=& \divg_{g_0}\left(\frac{\bigl(\delta_Yg_0(\nabla_{g_0}\psi,\cdot)\bigr)^\sharp}{W_\psi}\right)
-\divg_{g_0}\left(\frac{\delta_Yg_0(\nabla_{g_0}\psi,\nabla_{g_0}\psi)}{2W_\psi^3}\nabla_{g_0}\psi\right)\\[3pt]
&-\frac{\left\langle\nabla_{g_0}(\tr_{g_0}\delta_Yg_0),\nabla_{g_0}\psi\right\rangle_{g_0}}{2W_\psi}+
\frac{(N-1)\delta_Yg_0(\nabla_{g_0}\psi,\nabla_{g_0}\psi)}{2W_\psi^3},
\end{aligned}
$$
which proves the first identity of Lemma \ref{lem3sect4}.\\

We now compute the derivative of $F_2$.\\

As seen in the proof of Proposition \ref{prop:FC1}, the curve $t\mapsto\nu_t$ is differentiable.
Setting
$$
\dot\nu_0:=\left.\frac{d}{dt}\right|_{t=0}\nu_t,
$$
since $\psi$ is independent of $t$, we have
$$
\left.\frac{d}{dt}\right|_{t=0}F_2(\psi,h_t)=d\psi(\dot\nu_0)=\langle\nabla_{g_0}\psi,\dot\nu_0\rangle_{g_0}.
$$
Differentiating the identity
$$
g_t(\nu_t,\nu_t)=1
$$
at $t=0$, we obtain
$$
\delta_Yg_0(\nu_{g_0},\nu_{g_0})+2\langle\dot\nu_0,\nu_{g_0}\rangle_{g_0}=0,
$$
and hence
$$
\langle\dot\nu_0,\nu_{g_0}\rangle_{g_0}=-\frac12\delta_Yg_0(\nu_{g_0},\nu_{g_0}).
$$
Let $\tau$ be a fixed vector field tangent to $\partial D_0$. Differentiating
$$
g_t(\nu_t,\tau)=0
$$
at $t=0$, we get
$$
\delta_Yg_0(\nu_{g_0},\tau)+\langle\dot\nu_0,\tau\rangle_{g_0}=0.
$$
Therefore,
$$
\langle\dot\nu_0,\tau\rangle_{g_0}
=-\delta_Yg_0(\nu_{g_0},\tau).
$$
Since 
$$
\nabla_{g_0}\psi=\nabla_{g_0}^{\top}\psi+\partial_{\nu_{g_0}}\psi\,\nu_{g_0} \quad \hbox{on $\partial D_0$},
$$
it follows that
$$
\begin{aligned}
D_hF_2(\psi,\Id)[Y]&=\langle\nabla_{g_0}\psi,\dot\nu_0\rangle_{g_0}\\
&=-\delta_Yg_0\bigl(\nu_{g_0},\nabla_{g_0}^{\top}\psi\bigr)-\frac12 \delta_Yg_0(\nu_{g_0},\nu_{g_0}) \partial_{\nu_{g_0}}\psi.
\end{aligned}
$$
This proves the second identity of Lemma \ref{lem3sect4} and the proof is complete.
\end{proof}

We conclude this section by showing the corresponding formulas for $D_h F_1(\psi,h)$, $D_h F_2(\psi,h)$ at an arbitrary point $(\psi,h)\in C^{2,\alpha}(\overline{D_0})\times\Param$. 
However, we point out that these general formulas are not needed in the sequel, since, as we will see later in Section 8, we can always reduce to the case $h=\Id$.

Let $(\psi,h)\in C^{2,\alpha}(\overline{D_0})\times\Param$ and $Y\in\Y$. Since $Y\circ h\in\Y_h$, applying \eqref{eq:pointwisetangent} with $V_t=tY\circ h$ we obtain that the differentiable curve in $\DiffQ$
$$
t\mapsto h_t:=\Theta_h(tY\circ h),
$$
defined for $|t|$ sufficiently small, satisfies
$$
h_0=h, \qquad \left.\frac{d}{dt}\right|_{t=0}h_t=Y\circ h,
$$
and, since $\Param$ is open in $\DiffQ$, we have $h_t\in\Param$ for all sufficiently small $|t|$. Moreover, denoting by
$k_t$ the curve defined as in \eqref{eq1lem2sect4}, i.e. $t\mapsto k_t:=\Theta_{\Id}(tY)$, we have
$$
h_t=k_t\circ h,
$$
and by the composition rule for pull-backs, it holds
$$
g_{h_t}=h_t^*g_0=(k_t\circ h)^*g_0=h^*(k_t^*g_0).
$$
Since $k_t$ satisfies \eqref{eqstat:lem1:sect4}, by
Lemma \ref{lem1:sect4} we have
$$
\left.\frac{d}{dt}\right|_{t=0}k_t^*g_0=\mathcal L_Yg_0,
$$
and thus, as $h^*$ is linear and continuous for fixed $h$, we deduce that
$$
\left.\frac{d}{dt}\right|_{t=0}g_{h_t}=h^*(\mathcal L_Yg_0).
$$
Let us set
$$
\delta_{Y,h}g_h:=h^*(\mathcal L_Yg_0).
$$
We notice that the proofs of Lemma \ref{lem2sect4} and of
Lemma \ref{lem3sect4} only use the metric at $t=0$ and its first variation, as does
Lemma \ref{lem2appendix}. Hence they apply with $g_0$ replaced by $g_h$ and
$\delta_Yg_0$ replaced by $\delta_{Y,h}g_h$. We therefore obtain
$$
\begin{aligned}
D_hF_1(\psi,h)[Y\circ h]
=&\divg_{g_h}\left(\frac{\bigl(\delta_{Y,h}g_h(\nabla_{g_h}\psi,\cdot)\bigr)^{\sharp_{g_h}}}{W_{\psi,h}}\right)-\divg_{g_h}
\left(\frac{\delta_{Y,h}g_h(\nabla_{g_h}\psi,\nabla_{g_h}\psi)}{2W_{\psi,h}^3}\nabla_{g_h}\psi\right)\\[3pt]
&-\frac{\left\langle\nabla_{g_h}\bigl(\tr_{g_h}\delta_{Y,h}g_h\bigr),\nabla_{g_h}\psi\right\rangle_{g_h}}{2W_{\psi,h}}+\frac{(N-1)\delta_{Y,h}g_h(\nabla_{g_h}\psi,\nabla_{g_h}\psi)}{2W_{\psi,h}^3},\\
D_hF_2(\psi,h)[Y\circ h]=&-\delta_{Y,h}g_h\bigl(\nu_{g_h},\nabla_{g_h}^{\top}\psi\bigr)-\frac12
\delta_{Y,h}g_h(\nu_{g_h},\nu_{g_h})
\partial_{\nu_{g_h}}\psi,
\end{aligned}
$$
where $W_{\psi,h}$ is given by \eqref{eq:notationsimplification}, $\bigl(\delta_{Y,h}g_h(\nabla_{g_h}\psi,\cdot)\bigr)^{\sharp_{g_h}}$ is the $g_h$-metric dual of the one-form $\delta_{Y,h}g_h(\nabla_{g_h}\psi,\cdot)$ (see \eqref{eq:defomegasharp} for the definition), and $\nabla_{g_h}^{\top}\psi$ denotes the tangential component of $\nabla_{g_h}\psi$ on $\partial D_0$ with respect to $g_h$.

\section{Kernel of the adjoint of $D_\psi F(\psi,\Id)$}

In this section we characterize the kernel of the adjoint of
$D_\psi F(\psi,\Id)$. We only consider it at $(\psi, \Id)$, instead of considering a generic diffeomorphism $h$, to keep the notation simple. Indeed the proof depends on $D_0$ only through its regularity and the inclusion $D_0\Subset Q$. 
Consequently, the case of a general diffeomorphism $h$ can be treated by taking $D_h$ as the reference domain and replacing $\psi$ with $\psi\circ h^{-1}$, thereby reducing the problem to the identity. 
This reduction will be used in Section 8.
\medskip

Let $\psi\in C^{2,\alpha}(\overline{D_0})$ and set, for simplicity,
\begin{equation}\label{eq:simplnotation}
A_\psi:=A_{\psi,\Id}, \qquad W_\psi:=\sqrt{1+|\nabla_{g_0}\psi|_{g_0}^2}, \qquad b_\psi:=\frac{N-1}{W_\psi^3}\nabla_{g_0}\psi, \qquad c_\psi:=(N-1)He^\psi,
\end{equation}
where, for $h\in\Param$, $A_{\psi,h}$ is the $g_h$-symmetric $(1,1)$-tensor field defined in \eqref{eq:defApsih}.

In view of Lemma \ref{lem:DFpsi}, the linear operator
$D_\psi F(\psi,\Id)$ naturally extends to a bounded linear operator from
$H^2(D_0)$ to $L^2(D_0)\times L^2(\partial D_0)$, and we still denote its
extension by $D_\psi F(\psi,\Id)$. Namely,
$$
D_\psi F(\psi,\Id):H^2(D_0)\longrightarrow L^2(D_0)\times L^2(\partial D_0), \qquad D_\psi F(\psi,\Id)u =(L_\psi u,\partial_{\nu_{g_0}}u),
$$
where
$$
L_\psi u:=D_\psi F_1(\psi,\Id)[u]=-\divg_{g_0}(A_\psi\nabla_{g_0}u)-\langle b_\psi,\nabla_{g_0}u\rangle_{g_0}-c_\psi u,
$$
and $L^2(D_0)\times L^2(\partial D_0)$ is endowed with the inner product
\begin{equation}\label{eq:defscalarprodL2L2}
\bigl((v_1,v_2),(w_1,w_2)\bigr)_{L^2(D_0)\times L^2(\partial D_0)}:=\int_{D_0}v_1w_1\,d\mu_{g_0}+\int_{\partial D_0}v_2w_2\,d\sigma_{g_0}.
\end{equation}
Let $L_\psi^*$ be the formal adjoint of $L_\psi$ (see \cite[Chapter 2, Sect. 2]{LionsMagenes}), which is defined as the adjoint of $L_\psi$ in the sense of distributions in $D_0$, i.e. $L_\psi^*$ satisfies
$$
\int_{D_0}(L_\psi u)w\,d\mu_{g_0}=\int_{D_0}u (L_\psi^* w) \,d\mu_{g_0},
$$
for all $u,w \in C_c^\infty(D_0)$. We begin with the following preliminary result.

\begin{lemma}\label{lem1:sect5}
Let $\psi\in C^{2,\alpha}(\overline{D_0})$ and let $L_\psi^*$ be the formal adjoint of $L_\psi$. 
We have that
\begin{equation}\label{eq:formaladjointLpsi}
L_\psi^*w=-\divg_{g_0}(A_\psi\nabla_{g_0}w)+\divg_{g_0}(b_\psi w)-c_\psi w,
\end{equation}
where $A_\psi$, $b_\psi$, $c_\psi$ are as in \eqref{eq:simplnotation}.
If, in addition, $\psi$ satisfies
\beq\label{eq:Neumpsi}
\partial_{\nu_{g_0}}\psi=0
\qquad\text{on }\partial D_0,
\eeq
and defining $L_\psi^*:H^2(D_0)\to L^2(D_0)$ by \eqref{eq:formaladjointLpsi}, then, for every $u,w\in H^2(D_0)$, we have
\begin{equation}\label{eq1:statlem1sect5}
\begin{aligned}
\int_{D_0}(L_\psi u)w\,d\mu_{g_0}=&\int_{D_0}u(L_\psi^*w)\,d\mu_{g_0}\\
&+\int_{\partial D_0}\frac{u}{W_\psi}\, \partial_{\nu_{g_0}}w\,d\sigma_{g_0}-\int_{\partial D_0}\frac{w}{W_\psi}\, \partial_{\nu_{g_0}}u\,d\sigma_{g_0}.
\end{aligned}
\end{equation}
\end{lemma}

\begin{proof}
For any $u,w\in H^2(D_0)$, we have
\begin{equation}\label{eq0:lem1sect5}
\begin{aligned}
&\int_{D_0}\bigl[-\divg_{g_0}(A_\psi\nabla_{g_0}u)\bigr]w\,d\mu_{g_0}\\[3pt]
=&\int_{D_0}u\bigl[-\divg_{g_0}(A_\psi\nabla_{g_0}w)\bigr]\,d\mu_{g_0}+\int_{\partial D_0}u\langle A_\psi\nabla_{g_0}w,\nu_{g_0}\rangle_{g_0}\,d\sigma_{g_0}\\
&-\int_{\partial D_0}w \langle A_\psi\nabla_{g_0}u,\nu_{g_0}\rangle_{g_0}\,d\sigma_{g_0}.
\end{aligned}
\end{equation}
Moreover, integrating by parts the first-order term, we obtain
\begin{equation}\label{eq2:lem1sect5}
-\int_{D_0}\langle b_\psi,\nabla_{g_0}u\rangle_{g_0}w\,d\mu_{g_0}=\int_{D_0}u\,\divg_{g_0}(b_\psi w)\,d\mu_{g_0}-\int_{\partial D_0}uw \langle b_\psi,\nu_{g_0}\rangle_{g_0}
\,d\sigma_{g_0}.
\end{equation}
If in particular $u,w \in C_c^\infty(D_0)$, then, from \eqref{eq0:lem1sect5} and \eqref{eq2:lem1sect5} we get that
$$\int_{D_0}(L_\psi u)w\,d\mu_{g_0}=\int_{D_0}u \left(-\divg_{g_0}(A_\psi\nabla_{g_0}w)+\divg_{g_0}(b_\psi w)-c_\psi w \right)\,d\mu_{g_0},$$
and thus, by definition, it follows that $L_\psi^*$ is given by \eqref{eq:formaladjointLpsi}.

Finally, to prove \eqref{eq1:statlem1sect5}, let $u,w\in H^2(D_0)$ and assume that \eqref{eq:Neumpsi} holds. By definition
of $b_\psi$, we have
\begin{equation}\label{eq3:lem1sect5}
\langle b_\psi,\nu_{g_0}\rangle_{g_0}=\frac{N-1}{W_\psi^3}\partial_{\nu_{g_0}}\psi=0 \qquad\text{on }\partial D_0.
\end{equation}
In addition, using the definition of $A_\psi$ and again
\eqref{eq:Neumpsi}, we readily check that
\begin{equation}\label{eq4:lem1sect5}
\langle A_\psi\nabla_{g_0}u,\nu_{g_0}\rangle_{g_0}=\frac{1}{W_\psi}\partial_{\nu_{g_0}}u,\quad
\langle A_\psi\nabla_{g_0}w,\nu_{g_0}\rangle_{g_0}=\frac{1}{W_\psi}\partial_{\nu_{g_0}}w
\qquad\text{on }\partial D_0.
\end{equation}
Combining \eqref{eq0:lem1sect5} and
\eqref{eq2:lem1sect5}, and taking into account \eqref{eq:formaladjointLpsi},
\eqref{eq3:lem1sect5} and \eqref{eq4:lem1sect5}, we obtain
\eqref{eq1:statlem1sect5}. The proof is complete.
\end{proof}

The main result of this section is the following.

\begin{theorem}\label{theo1sect5}
Let $\psi\in C^{2,\alpha}(\overline{D_0})$ satisfy \eqref{eq:Neumpsi}.
Consider the bounded linear operator
$$
D_\psi F(\psi,\Id):H^2(D_0) \longrightarrow L^2(D_0)\times L^2(\partial D_0), \qquad u\longmapsto(L_\psi u,\partial_{\nu_{g_0}}u),
$$
and let
$$
(D_\psi F(\psi,\Id))': L^2(D_0)\times L^2(\partial D_0) \longrightarrow(H^2(D_0))'
$$
be its adjoint, after the canonical Riesz identification
$$
\bigl(L^2(D_0)\times L^2(\partial D_0)\bigr)' \cong L^2(D_0)\times L^2(\partial D_0).
$$
Then, for $(w,\eta)\in L^2(D_0)\times L^2(\partial D_0)$, the following
conditions are equivalent:

\begin{itemize}
\item[(i)]
$(w,\eta)\in\ker\bigl((D_\psi F(\psi,\Id))'\bigr)$;

\item[(ii)]
$w\in H^2(D_0)$ and
\begin{equation}\label{eq1:statteo1sect5}
\begin{cases}
-\divg_{g_0}(A_\psi\nabla_{g_0}w)+\divg_{g_0}(b_\psi w)-c_\psi w=0
&\text{in }D_0,\\[0.4em]
\partial_{\nu_{g_0}}w=0
&\text{on }\partial D_0,\\[0.4em]
\displaystyle\eta=\frac{w}{W_\psi}
&\text{on }\partial D_0.
\end{cases}
\end{equation}
\end{itemize}
In particular, recalling \eqref{eq:formaladjointLpsi}, we have
$$
\ker\bigl((D_\psi F(\psi,\Id))'\bigr)=\left\{\left(w,\frac{w}{W_\psi}\right);\ w\in H^2(D_0),\, L_\psi^*w=0\text{ in }D_0,\, \partial_{\nu_{g_0}}w=0\text{ on }\partial D_0\right\}.
$$
\end{theorem}

\begin{proof}
Let $\psi\in C^{2,\alpha}(\overline{D_0})$ satisfy \eqref{eq:Neumpsi}. Following \cite[Chapter 2, Sect. 8.4]{LionsMagenes}, we define the linear
operator
$$
\mathcal L:\mathcal D(\mathcal L)\subset L^2(D_0) \longrightarrow L^2(D_0), \qquad \mathcal Lu:=L_\psi u,
$$
with domain
$$
\mathcal D(\mathcal L):= \left\{u\in H^2(D_0);\, \partial_{\nu_{g_0}}u=0\text{ on }\partial D_0 \right\}.
$$
The domain $\mathcal D(\mathcal L)$ is dense in $L^2(D_0)$, since $C_c^\infty(D_0)\subset\mathcal D(\mathcal L)$ and $C_c^\infty(D_0)$ is
dense in $L^2(D_0)$. We refer to $\mathcal L$ as the $L^2$-realization of $L_\psi$ under the homogeneous Neumann boundary condition.

We divide the proof into three steps.

\medskip

\noindent
\textbf{Step 1}: Let $\mathcal L^\prime$ be the adjoint of
$\mathcal L$. For $w\in L^2(D_0)$, we have
\begin{equation}\label{eq:kerLstar-selfcontained}
w\in\ker(\mathcal L^\prime)
\iff
w\in H^2(D_0)\text{ and }
\begin{cases}
L_\psi^*w=0 &\text{in }D_0,\\
\partial_{\nu_{g_0}}w=0 &\text{on }\partial D_0.
\end{cases}
\end{equation}

To prove \eqref{eq:kerLstar-selfcontained}, consider the realization in
$L^2(D_0)$ of the formal adjoint $L_\psi^*$, namely
$$
\widehat{\mathcal L}:\mathcal D(\widehat{\mathcal L})\subset L^2(D_0)\longrightarrow L^2(D_0), \qquad \widehat{\mathcal L}w:=L_\psi^*w,
$$
with domain
$$
\mathcal D(\widehat{\mathcal L}):= \left\{w\in H^2(D_0);\, \partial_{\nu_{g_0}}w=0\text{ on }\partial D_0 \right\}.
$$
We show that
$$
\widehat{\mathcal L}=\mathcal L^\prime.
$$
We first claim that there exists $\tau_0>0$ such that, for every
$\tau\geq\tau_0$, the shifted realizations
\begin{equation}\label{eq:shiftedproblems}
\mathcal L_\tau:=\mathcal L+\tau I \qquad\text{and}\qquad \widehat{\mathcal L}_\tau:=\widehat{\mathcal L}+\tau I
\end{equation}
are bijections from their domains onto $L^2(D_0)$, where $I$ denotes the
identity map.

To prove the claim, we work in the chart given by the stereographic projection $\pi$ from a pole $p_*\in \SN\setminus Q$ (see \eqref{eq:defsterproj}) and we define $\Omega_0$ as in \eqref{eq:defomega0}. In local coordinates, the pairs
$(L_\psi,\partial_{\nu_{g_0}})$ and $(L_\psi^*,\partial_{\nu_{g_0}})$
are represented, respectively, by
$$
(\widetilde{L_\psi},\widetilde M) \qquad\text{and}\qquad (\widetilde{L_\psi^*},\widetilde M),$$
where $\widetilde{L_\psi}$ and $\widetilde{L_\psi^*}$ are uniformly elliptic second-order operators of the form
$$
\begin{aligned}
\widetilde{L_\psi}v&=-\widetilde a^{ij}\partial_{ij}v+\widetilde b^i\partial_iv+\widetilde cv,\\
\widetilde{L_\psi^*}v&=-\widetilde a^{ij}\partial_{ij}v+\widetilde b_1^i\partial_iv+\widetilde c_1v,
\end{aligned}
$$
with
$$
\widetilde a^{ij}\in C^{1,\alpha}(\overline{\Omega_0}), \qquad \widetilde b^i,\widetilde c, \widetilde b_1^i,\widetilde c_1 \in 
C^{0,\alpha}(\overline{\Omega_0})
$$
for every $i,j=1,\ldots,N-1$, and $\widetilde M$ is a first-order uniformly oblique boundary operator of the form \eqref{eq:Mpsiform}, with
$$
\widetilde\beta \in C^{1,\alpha}(\partial\Omega_0,\mathbb R^{N-1}).
$$

Since $\partial\Omega_0$ is of class $C^{3,\alpha}$, in particular $\partial\Omega_0$ is of class $C^{1,\gamma}$ for every
$\gamma\in(0,1)$. Choose $\gamma\in(1/2,1)$. Then
$$
2<\frac{1}{1-\gamma}.
$$
We may therefore apply \cite[Theorem~6.29]{LiebermanODP}, with $p=2$, taking into account the opposite sign conventions for the elliptic operator and the orientation of the normal vector. From this we obtain a
constant $\tau_0>0$ such that, for every $\tau\geq\tau_0$ and every
$f\in L^2(\Omega_0)$, each of the problems
$$
\begin{cases}
\widetilde{L_\psi}v+\tau v=f&\text{in }\Omega_0,\\
\widetilde Mv=0&\text{on }\partial\Omega_0,
\end{cases}
$$
and
$$
\begin{cases}
\widetilde{L_\psi^*}v+\tau v=f &\text{in }\Omega_0,\\
\widetilde Mv=0 &\text{on }\partial\Omega_0
\end{cases}
$$
has a unique solution in $H^2(\Omega_0)$. If $v$ is a solution on $\Omega_0$, then $u:=v\circ\pi$ is the corresponding solution on $D_0$,
and vice versa. Hence, for every $\tau\geq\tau_0$, both operators in \eqref{eq:shiftedproblems} are bijections from their domains onto
$L^2(D_0)$, which proves the claim.

For every $u\in\mathcal D(\mathcal L)$ and
$w\in\mathcal D(\widehat{\mathcal L})$, we have
$$
\partial_{\nu_{g_0}}u=\partial_{\nu_{g_0}}w=0 \qquad\text{on }\partial D_0.
$$
Thus, by \eqref{eq1:statlem1sect5}, it holds
$$
\int_{D_0}(L_\psi u)w\,d\mu_{g_0}=\int_{D_0}u(L_\psi^*w)\,d\mu_{g_0}.
$$
Adding
$$
\tau\int_{D_0}uw\,d\mu_{g_0}
$$
to both sides gives
\begin{equation}\label{eq:adjrelation}
(\mathcal L_\tau u,w)_{L^2(D_0)}=(u,\widehat{\mathcal L}_\tau w)_{L^2(D_0)}
\end{equation}
for every $u\in\mathcal D(\mathcal L)$ and
$w\in\mathcal D(\widehat{\mathcal L})$.

Since $\mathcal D(\mathcal L_\tau)=\mathcal D(\mathcal L)$ is dense in
$L^2(D_0)$, the adjoint
$$
\mathcal L_\tau^\prime: \mathcal D(\mathcal L_\tau^\prime)\subset L^2(D_0) \longrightarrow L^2(D_0)
$$
is well defined, with domain
$$
\begin{aligned}
\mathcal D(\mathcal L_\tau^\prime):=\bigl\{
&w\in L^2(D_0);\, \text{there exists }z\in L^2(D_0)\text{ such that}\\
&(\mathcal L_\tau u,w)_{L^2(D_0)}=(u,z)_{L^2(D_0)} \text{ for every }u\in\mathcal D(\mathcal L)\bigr\}.
\end{aligned}
$$
Such $z$ is necessarily unique, and one sets
$\mathcal L_\tau^\prime w:=z$. By \eqref{eq:adjrelation},
$$
\mathcal D(\widehat{\mathcal L}) \subset \mathcal D(\mathcal L_\tau^\prime)
$$
and
\begin{equation}\label{eq:LLstar}
\mathcal L_\tau^\prime v=\widehat{\mathcal L}_\tau v \qquad\text{for every }v\in\mathcal D(\widehat{\mathcal L}).
\end{equation}

We now prove the reverse inclusion. Let
$v\in\mathcal D(\mathcal L_\tau^\prime)$ and set
$$
f:=\mathcal L_\tau^\prime v\in L^2(D_0).
$$
Since $\widehat{\mathcal L}_\tau$ is onto, there exists
$w\in\mathcal D(\widehat{\mathcal L})$ such that
$$
\widehat{\mathcal L}_\tau w=f.
$$
By \eqref{eq:LLstar},
\begin{equation}\label{eq:LtildecalLtilde}
\mathcal L_\tau^\prime w=f=\mathcal L_\tau^\prime v.
\end{equation}
Since $\mathcal L_\tau$ is surjective,
$$
\ker(\mathcal L_\tau^\prime)=\operatorname{Range}(\mathcal L_\tau)^\perp=\{0\}.
$$
Therefore \eqref{eq:LtildecalLtilde} implies that $v=w$. In particular,
$v\in\mathcal D(\widehat{\mathcal L})$, and hence
$$
\mathcal L_\tau^\prime=\widehat{\mathcal L}_\tau.
$$
Finally,
$$
\mathcal L_\tau^\prime=(\mathcal L+\tau I)^\prime=\mathcal L^\prime+\tau I,
$$
and therefore
$$
\mathcal L^\prime=\widehat{\mathcal L}.
$$
This proves \eqref{eq:kerLstar-selfcontained}.

\medskip

\noindent
\textbf{Step 2}: $(i)\Rightarrow(ii)$.\\

Assume that
$$
(w,\eta)\in\ker\bigl((D_\psi F(\psi,\Id))'\bigr).
$$
By the definition of adjoint and the Riesz identification, this
means that
$$
\bigl(D_\psi F(\psi,\Id)u,(w,\eta)\bigr)_{L^2(D_0)\times L^2(\partial D_0)}=0 \qquad\text{for every }u\in H^2(D_0),
$$
that is,
\begin{equation}\label{eq1:step2teo1sect5}
\int_{D_0}(L_\psi u)w\,d\mu_{g_0}+\int_{\partial D_0}\eta\,\partial_{\nu_{g_0}}u\,d\sigma_{g_0}=0
\qquad\forall u\in H^2(D_0).
\end{equation}
If $u\in\mathcal D(\mathcal L)$, then
$\partial_{\nu_{g_0}}u=0$ on $\partial D_0$, and hence
$$
(\mathcal Lu,w)_{L^2(D_0)}=\int_{D_0}(L_\psi u)w\,d\mu_{g_0}=0=(u,0)_{L^2(D_0)}.
$$
By the definition of the Hilbert adjoint,
$$
w\in\mathcal D(\mathcal L^\prime)\qquad\text{and}\qquad \mathcal L^\prime w=0.
$$
Thus $w\in\ker(\mathcal L^\prime)$, and \textit{Step 1} gives
$$
w\in H^2(D_0), \qquad L_\psi^*w=0\text{ in }D_0, \qquad \partial_{\nu_{g_0}}w=0\text{ on }\partial D_0.
$$
Since $w\in H^2(D_0)$, then by \eqref{eq1:statlem1sect5} and \eqref{eq1:step2teo1sect5}, we obtain
$$
0=\int_{D_0}u(L_\psi^*w)\,d\mu_{g_0}+\int_{\partial D_0}\frac{u}{W_\psi}\partial_{\nu_{g_0}}w\,d\sigma_{g_0}+
\int_{\partial D_0} \left(\eta-\frac{w}{W_\psi}\right) \partial_{\nu_{g_0}}u\,d\sigma_{g_0} \qquad \forall u\in H^2(D_0).
$$
Therefore
$$
\int_{\partial D_0}\left(\eta-\frac{w}{W_\psi}\right)\partial_{\nu_{g_0}}u\,d\sigma_{g_0}=0\qquad \forall u\in H^2(D_0).
$$
Now, recall that the normal trace map
$$
H^2(D_0)\longrightarrow H^{1/2}(\partial D_0), \qquad u\longmapsto\partial_{\nu_{g_0}}u,
$$
is bounded and surjective: this follows from \cite[Theorem 1.5.1.2]{Grisvard}, applied in the chart $\pi$, taking into account that $\pi$ is conformal. Then, as $H^{1/2}(\partial D_0)$ is dense in $L^2(\partial D_0)$ and $\eta-\frac{w}{W_\psi}\in L^2(\partial D_0)$, we conclude that
$$
\eta=\frac{w}{W_\psi} \qquad\text{a.e. on }\partial D_0.
$$
Thus (ii) holds.

\medskip

\noindent
\textbf{Step 3}: $(ii)\Rightarrow(i)$.\\

Assume that $w\in H^2(D_0)$ and that \eqref{eq1:statteo1sect5} holds. Then, from \eqref{eq1:statlem1sect5} and \eqref{eq1:statteo1sect5}, we obtain
$$
\int_{D_0}(L_\psi u)w\,d\mu_{g_0}+\int_{\partial D_0}\eta\,\partial_{\nu_{g_0}}u\,d\sigma_{g_0}=0\qquad \forall u\in H^2(D_0).
$$
By definition (see \eqref{eq:defscalarprodL2L2}) this means that
$$
\bigl(D_\psi F(\psi,\Id)u,(w,\eta)\bigr)_{L^2(D_0)\times L^2(\partial D_0)}=0 \qquad\forall u\in H^2(D_0),
$$
that is,
$$
(w,\eta)\in\ker\bigl((D_\psi F(\psi,\Id))'\bigr).
$$
Thus (i) holds and the proof is complete.
\end{proof}

\begin{corollary}\label{cor:regwKernel}
Under the assumptions of Theorem \ref{theo1sect5}, if
$$
(w,\eta)\in\ker\bigl((D_\psi F(\psi,\Id))'\bigr),
$$
then
$$
(w,\eta)\in C^{2,\alpha}(\overline{D_0}) \times C^{1,\alpha}(\partial D_0).
$$
\end{corollary}

\begin{proof}
By Theorem \ref{theo1sect5} we have $w\in H^2(D_0)$, $\eta=w/W_\psi$ on $\partial D_0$ in the sense of traces, and
\begin{equation}\label{eq:adjoint-homogeneous-problem}
L_\psi^*w=0\ \text{ in }D_0,\qquad \partial_{\nu_{g_0}}w=0\ \text{ on }\partial D_0 ,
\end{equation}
where $L_\psi^*$ is defined by \eqref{eq:formaladjointLpsi}. Given $\tau>0$, by \eqref{eq:adjoint-homogeneous-problem} the function $w$ solves
\begin{equation}\label{eq1:cor1sect5}
L_\psi^*u+\tau u=\tau w\ \text{ in }D_0,\qquad \partial_{\nu_{g_0}}u=0\ \text{ on }\partial D_0.
\end{equation}
On the other hand, arguing as in \textit{Step 1} of Theorem \ref{theo1sect5}, for $\tau$ large \eqref{eq1:cor1sect5} has a unique solution
$u\in W^{2,p}(D_0)$ whenever $w\in L^p(D_0)$, $2\leq p<\infty$ (see \cite[Theorem 6.29]{LiebermanODP}, taking into account that $\partial D_0$ is of class $C^{1,\gamma}$ for every $\gamma\in(0,1)$). Moreover \eqref{eq1:cor1sect5} has a unique solution
$u\in C^{2,\alpha}(\overline{D_0})$ whenever $w\in C^{0,\alpha}(\overline{D_0})$ (see \cite[Corollary 4.41]{LiebermanODP}). Since both classes are contained in $H^2(D_0)$, where \eqref{eq1:cor1sect5} is uniquely solvable by \textit{Step 1} of Theorem \ref{theo1sect5}, in either
case $u=w$, and therefore the following two implications hold
\begin{equation}\label{eq:BI}
w\in L^p(D_0)\ \Rightarrow\ w\in W^{2,p}(D_0), \qquad w\in C^{0,\alpha}(\overline{D_0})\ \Rightarrow\ w\in C^{2,\alpha}(\overline{D_0}).
\end{equation}
We then conclude by a standard bootstrap argument. Namely, starting from $w\in H^2(D_0)$ and iterating the first implication in \eqref{eq:BI}
together with the Sobolev embedding, after finitely many steps we get that $w\in W^{2,q}(D_0)$, for some $q>\frac{N-1}{2-\alpha}$, whence
$w\in C^{0,\alpha}(\overline{D_0})$ by the Sobolev--Morrey embedding. 
Then, by the second implication in \eqref{eq:BI} we deduce that $w\in C^{2,\alpha}(\overline{D_0})$. Finally, as $W_\psi$ is positive and of class $C^{1,\alpha}$, because $\psi\in C^{2,\alpha}(\overline{D_0})$, we conclude that $\eta=w/W_\psi\in C^{1,\alpha}(\partial D_0)$. 
The proof is complete.
\end{proof}

\begin{remark}\label{rem:Schauder-lifting}
The same argument as in the proof of Corollary \ref{cor:regwKernel}, with minor modifications, gives the
following regularity property. Assume that $\psi \in C^{2,\alpha}(\overline{D_0})$ satisfies \eqref{eq:Neumpsi}. 
If
$$
(f_1,f_2) \in C^{0,\alpha}(\overline{D_0}) \times C^{1,\alpha}(\partial D_0)
$$
and $u\in H^2(D_0)$ solves
$$
\begin{cases}
L_\psi^*u=f_1&\text{in }D_0,\\
\partial_{\nu_{g_0}}u=f_2 &\text{on }\partial D_0,
\end{cases}
$$
then
$$
u\in C^{2,\alpha}(\overline{D_0}).
$$
The same conclusion holds for the solutions of
$$
\begin{cases}
L_\psi u=f_1&\text{in }D_0,\\
\partial_{\nu_{g_0}}u=f_2 &\text{on }\partial D_0.
\end{cases}
$$
\end{remark}

\begin{remark}\label{rem:H12-realization}
Theorem \ref{theo1sect5} also holds when $L^2(\partial D_0)$ is replaced by $H^{1/2}(\partial D_0)$ in the target space. More precisely, since the normal trace $u\mapsto\partial_{\nu_{g_0}}u$ is bounded from $H^2(D_0)$ into $H^{1/2}(\partial D_0)$ (see \textit{Step 2} of the proof of
Theorem \ref{theo1sect5}), the operator $D_\psi F(\psi,\Id)$ takes values in the smaller space $L^2(D_0)\times H^{1/2}(\partial D_0)$ and remains bounded when regarded as an operator into this space. 
We denote this operator by $\mathbb D_\psi F(\psi,\Id)$, namely
$$
\mathbb D_\psi F(\psi,\Id):H^2(D_0)\longrightarrow L^2(D_0)\times H^{1/2}(\partial D_0), \qquad \mathbb D_\psi F(\psi,\Id)u:=\bigl(L_\psi u,\partial_{\nu_{g_0}}u\bigr).
$$
Its adjoint $(\mathbb D_\psi F(\psi,\Id))'$ is defined on $L^2(D_0)\times H^{-1/2}(\partial D_0)$ with respect to the duality pairing
$$
\bigl\langle (f,g),(w,\eta)\bigr\rangle:=\int_{D_0}fw\,d\mu_{g_0} +\langle g,\eta\rangle_{H^{1/2}(\partial D_0),H^{-1/2}(\partial D_0)}.
$$
We denote by $\iota:L^2(\partial D_0)\hookrightarrow H^{-1/2}(\partial D_0)$ the canonical continuous injection, given by
\beq\label{eq:caninj}
\langle g,\iota\eta\rangle_{H^{1/2}(\partial D_0),H^{-1/2}(\partial D_0)}:=\int_{\partial D_0}g\eta\,d\sigma_{g_0},
\qquad g\in H^{1/2}(\partial D_0),
\eeq
so that the above pairing is compatible with \eqref{eq:defscalarprodL2L2}.
Let $(w,\eta)\in\ker\bigl((\mathbb D_\psi F(\psi,\Id))'\bigr)$. 
Testing with $u\in\mathcal D(\mathcal L)$, for which $\partial_{\nu_{g_0}}u=0$, the first part of \textit{Step 2} of the proof of Theorem \ref{theo1sect5} applies verbatim and gives $w\in H^2(D_0)$, $L_\psi^*w=0$ in $D_0$ and $\partial_{\nu_{g_0}}w=0$ on $\partial D_0$. 
Then \eqref{eq1:statlem1sect5} yields
$$
\bigl\langle\partial_{\nu_{g_0}}u,\ \eta-\iota(w/W_\psi) \bigr\rangle_{H^{1/2}(\partial D_0),H^{-1/2}(\partial D_0)}=0 \qquad\forall u\in H^2(D_0),
$$
and, since the normal trace maps $H^2(D_0)$ onto $H^{1/2}(\partial D_0)$, we
get $\eta=\iota(w/W_\psi)$. Conversely, for $(w,\eta)\in L^2(D_0)\times L^2(\partial D_0)$ the two pairings coincide, so
that 
$$(w,\eta)\in\ker\bigl((D_\psi F(\psi,\Id))'\bigr) \Longleftrightarrow (w,\iota\eta)\in\ker\bigl((\mathbb D_\psi F(\psi,\Id))'\bigr).$$ 
Therefore $\mathrm{id}\times\iota$ maps $\ker\bigl((D_\psi F(\psi,\Id))'\bigr)$ bijectively onto $\ker\bigl((\mathbb D_\psi F(\psi,\Id))'\bigr)$, and, by Corollary \ref{cor:regwKernel}, both kernels are identified with
$$
\Bigl\{\Bigl(w,\frac{w}{W_\psi}\Bigr);\ w\in C^{2,\alpha}(\overline{D_0}),\ L_\psi^*w=0\text{ in }D_0,\ \partial_{\nu_{g_0}}w=0\text{ on }\partial D_0\Bigr\}.
$$
Moreover, $\mathbb D_\psi F(\psi,\Id)$ is a Fredholm operator of index zero.
Indeed, arguing as in \textit{Step 1} of Theorem \ref{theo1sect5}, but allowing an arbitrary boundary datum $g\in H^{1/2}(\partial D_0)$ in \cite[Theorem 6.29]{LiebermanODP}, it follows that, for $\tau$ sufficiently large, the map $u\mapsto\bigl(L_\psi u+\tau u,\partial_{\nu_{g_0}}u\bigr)$ is a bijection of $H^2(D_0)$ onto $L^2(D_0)\times H^{1/2}(\partial D_0)$, and in fact an isomorphism, the continuity of the inverse following from estimate \cite[(6.47)]{LiebermanODP}. 
Since the embedding $H^2(D_0)\hookrightarrow L^2(D_0)$ is compact, $\mathbb D_\psi F(\psi,\Id)$ is a compact perturbation of an isomorphism, and the conclusion follows.
\end{remark}

\section{Proof that $F$ is $\sigma$-proper}
In this section we prove that the map $F$ defined in \eqref{eq:defF} is $\sigma$-proper. More precisely, following \cite[Theorem 5.4-(3)]{Henry}, we prove that the projection
$$
\Pi:F^{-1}(0,0)\longrightarrow\Param, \qquad \Pi(\psi,h):=h,
$$
is $\sigma$-proper, namely, there exists a sequence of subsets 
$M_l\subset F^{-1}(0,0)$, $l\in\mathbb N^+$, such that
$$F^{-1}(0,0)=\bigcup_{l=1}^{\infty}M_l$$
and, for every $l\in\mathbb N^+$, and for any sequence
$$(\psi_n,h_n)\in M_l \qquad\text{such that}\qquad h_n\longrightarrow h\quad\text{in }\Param$$
there exists a subsequence $(\psi_{n_j},h_{n_j})$ converging in $C^{2,\alpha}(\overline{D_0})\times\Param$ to an element $(\psi,h)\in M_l$.
\begin{proposition}\label{prop:sigmaproper}
The projection $\Pi:F^{-1}(0,0)\longrightarrow\Param$ is $\sigma$-proper. 
\end{proposition}
\begin{proof}
For $l\in \N^+$ we denote by $\overline{\mathcal{B}_l(0)}$  the closed ball in $C^{2,\alpha}(\overline{D_0})$ of radius $l$, centered at $0$, namely
$$\overline{\mathcal{B}_l(0)}=\{\psi \in C^{2,\alpha}(\overline{D_0}); \ \|\psi\|_{C^{2,\alpha}(\overline{D_0})} \leq l\},$$ 
and we set $$M_l:=(\overline{\mathcal{B}_l(0)}\times\Param)\cap F^{-1}(0,0).$$
Clearly, it holds that
\beq\label{eq:Fpreimageunion}
F^{-1}(0,0)=\bigcup_{l=1}^\infty M_l.
\eeq
Let us fix $l \in \N^+$, and consider a sequence $(\psi_n,h_n)_n$ in $M_l$ such that $h_n\to h$ in $\Param$, for some $h \in \Param$. By definition, for each $n\in \N$, $\psi_n \in \overline{\mathcal{B}_l(0)}$ is a solution of
\begin{equation}\label{eq:probm}
\begin{cases}
-\divg_{g_{h_n}}\!\left(\dfrac{\nabla_{g_{h_n}}\psi}{\sqrt{1+|\nabla_{g_{h_n}}\psi|_{g_{h_n}}^{2}}}\right)
+ \dfrac{N-1}{\sqrt{1+|\nabla_{g_{h_n}}\psi|_{g_{h_n}}^{2}}} - (N-1)e^{\psi}H = 0 & \text{in } D_0,\\[3pt]
\dfrac{\partial \psi}{\partial \nu_{g_{h_n}}} = 0 & \text{on } \partial D_0.
\end{cases}
\end{equation}
We aim to prove that, up to a subsequence
\beq\label{eq:conpsinpsi}
\psi_n\to\psi \quad \text{in $C^{2,\alpha}(\overline{D_0})$}.
\eeq
We set
\begin{equation}\label{eq0:prop1sect8}
g_n:=g_{h_n},\quad B_n:=\nabla_{g_{n}}\psi_n,\quad W_{n}:= \sqrt{1+|B_n|_{g_{n}}^{2}} , \quad f_n:=(N-1)W_n^2(e^{\psi_n} W_{n} H-1).
\end{equation}
Putting the equation in \eqref{eq:probm} in non-divergence form, after some standard computations we check that $\psi_n$ is also a solution of the following
\begin{equation}\label{eq:probmnondiv}
\begin{cases}
-W_n^2\Delta_{g_{n}} \psi  + D_{g_{n}}^2\psi\, (B_n, B_n)=f_n  & \text{in } D_0,\\[6pt]
\dfrac{\partial \psi}{\partial \nu_{g_{n}}} = 0 & \text{on } \partial D_0.
\end{cases}
\end{equation}
Observe that, once the functions $B_n$, $W_n$ and $f_n$ are fixed, the equation in \eqref{eq:probmnondiv} can be viewed as a linear second-order equation in the unknown $\psi$.

To prove \eqref{eq:conpsinpsi} we pass to local coordinates and apply the regularity results for uniformly elliptic linear equations in
bounded domains of $\R^{N-1}$, with oblique derivative boundary condition (see e.g. \cite[Chapter 2]{LiebermanODP}). 
This will yield a global $C^{2,\alpha}(\overline{D_0})$ estimate for $(\psi_n)_n$, uniform with respect to $n$.

From now on we work in the chart $\pi$ given by the stereographic projection from a pole $p_*\in\SN\setminus Q$ (see \eqref{eq:defsterproj}), and we keep the same notation for functions, vector fields and tensor fields on $\overline{D_0}$
and for their local expressions on $\overline{\Omega_0}$, where $\Omega_0:=\pi(D_0)$.

By the expression of the Laplacian and the Hessian in local coordinates (see \eqref{eq:Laplacianlocalcoord}, \eqref{eq:Hesslocalcoord}) we have
$$
\Delta_{g_n}\psi = g_n^{ij} \left(\partial_{ij}\psi-\Gamma_{ij}^k(g_n)\partial_k\psi\right),\quad D^2_{g_n}\psi(B_n,B_n)=B_n^iB_n^j
\left(\partial_{ij}\psi-\Gamma_{ij}^k(g_n)\partial_k\psi\right),
$$
where the $\{\Gamma_{ij}^k\}_{i,j,k=1,\ldots,N-1}$ are the Christoffel symbols, and by definition (see \eqref{eq0:prop1sect8}) and the expression of the gradient in local coordinates
$$B_n=B_n^i\partial_i, \qquad B_n^i=g_n^{ij}\partial_j\psi_n.$$
Therefore
$$
-W_n^2\Delta_{g_n}\psi +D^2_{g_n}\psi(B_n,B_n)=-a_n^{ij}\partial_{ij}\psi+b_n^k\partial_k\psi,
$$
where
$$a_n^{ij}:=W_n^2 g_n^{ij}-B_n^iB_n^j,\quad b_n^k:=\left(W_n^2g_n^{ij}-B_n^iB_n^j\right) \Gamma_{ij}^k(g_n),$$
and thus, in $\Omega_0$, the PDE in \eqref{eq:probmnondiv} takes the form
\begin{equation}\label{eq:probmnondivlocal}
\widetilde{\mathcal{L}}_n \psi:=-a_n^{ij}\partial_{ij}\psi+b_n^k\partial_k\psi=f_n.
\end{equation}

As a first step we establish uniform H\"older bounds, with respect to $n$, for the coefficients $a_n^{ij}$, $b_n^k$. More precisely we prove the following.\\

\noindent\textbf{Claim I}: There exists a positive constant $C_0$ such that 
$$\sum_{i,j=1}^{N-1} \|a_n^{ij}\|_{C^{1,\alpha}(\overline{\Omega_0})} + \sum_{k=1}^{N-1} \|b_n^k\|_{C^{1,\alpha}(\overline{\Omega_0})} \leq C_0, \quad\forall n\in \N.$$

To prove the claim we begin with the metric $g_n$. To avoid repetitions, all the constants $c_1,\ldots,c_9$ appearing below are positive and independent of $n$ (they may depend on $l$, which is fixed).
Recalling that $g_n=h_n^*g_0$, we write
$$
g_n=(g_n)_{ij}\,dx^i\otimes dx^j.
$$
As seen in the proof of Proposition \ref{prop:FC1}, the maps $h\mapsto g_h$ and $h\mapsto g_h^{-1}$ are
continuous from $\Param$ into $C^{2,\alpha}(\overline{D_0},T^*\SN\otimes T^*\SN)$
and into $C^{2,\alpha}(\overline{D_0},T\SN\otimes T\SN)$, respectively. Since
$h_n\to h$ in $\Param$, in the chart $\pi$ we deduce that
\begin{equation}\label{eq3:prop1sect8}
(g_n)_{ij}\longrightarrow(g_h)_{ij},
\quad
g_n^{ij}\longrightarrow g_h^{ij}
\quad\text{in }C^{2,\alpha}(\overline{\Omega_0}),
\qquad
\|(g_n)_{ij}\|_{C^{2,\alpha}(\overline{\Omega_0})}+\|g_n^{ij}\|_{C^{2,\alpha}(\overline{\Omega_0})}\le c_1
\quad\forall n\in\N,
\end{equation}
for some constant $c_1$. Moreover $g_h$ is a Riemannian metric and $\overline{\Omega_0}$ is compact, so that
$(g_h)_{ij}\xi^i\xi^j\ge c_2|\xi|^2,$ for some constant $c_2$. By \eqref{eq3:prop1sect8}, and since each $g_n$ is a Riemannian metric, we find two constants $c_3,c_4$ such that, for every $x\in\overline{\Omega_0}$,
$\xi\in\R^{N-1}$ and $n\in\N$,
$$c_3|\xi|^2\le (g_n)_{ij}(x)\xi^i\xi^j\le c_4|\xi|^2,$$
where $|\cdot|$ is the Euclidean norm of $\R^{N-1}$. Consequently $G_n(x)=((g_n)_{ij}(x))_{ij}$ and $G_n^{-1}(x)=(g_n^{ij}(x))_{ij}$ are uniformly positive definite in $\overline{\Omega_0}$, and
\begin{equation}\label{eq:unifposdefinvG_n}
\frac1{c_4}|\xi|^2\le g_n^{ij}(x)\xi_i\xi_j\le\frac1{c_3}|\xi|^2 .
\end{equation}

Thanks to \eqref{eq3:prop1sect8} and since $\psi_n \in \overline{\mathcal{B}_l(0)}$, there exists a constant $c_5$ such that for any $i=1,\ldots,N-1$ it holds
\begin{equation}\label{eq5:prop1sect8}
\|B_n^i\|_{C^{1,\alpha}(\overline{\Omega_0})}\le c_5 \qquad  \forall n\in \N.
\end{equation}
In addition, as
$$
|B_n|_{g_n}^2=(g_n)_{ij}B_n^iB_n^j,
$$
and thanks to \eqref{eq3:prop1sect8}, \eqref{eq5:prop1sect8}, we conclude that
\begin{equation}\label{eq6:prop1sect8}
\|W_n\|_{C^{1,\alpha}(\overline{\Omega_0})}\leq c_6  \qquad\forall n\in \N.
\end{equation}
Hence, combining \eqref{eq3:prop1sect8}, \eqref{eq5:prop1sect8} and \eqref{eq6:prop1sect8} we find a constant $c_7$ such that for any $i,j=1,\ldots,N-1$
\beq\label{eq7bis:prop1sect8}
\|a_n^{ij}\|_{C^{1,\alpha}(\overline{\Omega_0})}\leq c_7 \qquad\forall n\in \N.
\eeq
Now, since the Christoffel symbols are given by
$$\Gamma_{ij}^k(g_n)=\frac12 g_n^{ks}\left(\partial_i(g_n)_{js}+\partial_j(g_n)_{is}-\partial_s(g_n)_{ij}\right),$$
(see \eqref{eq:christoffel}), then, from \eqref{eq3:prop1sect8}, we readily deduce that
\begin{equation}\label{eq7:prop1sect8}
\|\Gamma_{ij}^k(g_n)\|_{C^{1,\alpha}(\overline{\Omega_0})}\leq c_8  \qquad\forall n\in \N.
\end{equation}
Finally, as $b_n^k=a_n^{ij}\Gamma_{ij}^k(g_n)$, from \eqref{eq7bis:prop1sect8}
and \eqref{eq7:prop1sect8} we conclude that
$$
\|b_n^k\|_{C^{1,\alpha}(\overline{\Omega_0})}\le c_{9}  \qquad\forall n\in \N,
$$
and \textit{Claim I} is thus proved.\\

As a next step we check that the equation \eqref{eq:probmnondivlocal} is uniformly elliptic in $\Omega_0$ with ellipticity constants independent of $n$. This is the content of the following.\\

\noindent\textbf{Claim II}: There exist two positive constants $\lambda, \Lambda$ such that for any $\xi\in\mathbb R^{N-1}$
$$ \lambda |\xi|^2\leq a_n^{ij}(x)\xi_i\xi_j \leq \Lambda |\xi|^2 \quad \forall x\in \overline{\Omega_0},\ \forall n \in \N.$$

Indeed, for every $\xi=(\xi_1,\ldots,\xi_{N-1})\in\mathbb R^{N-1}$, we have
$$a_n^{ij}\xi_i\xi_j=W_n^2 g_n^{ij}\xi_i\xi_j-B_n^iB_n^j\xi_i\xi_j=W_n^2|\xi|_{g_n}^2-(B_n^i\xi_i)^2,$$
where, by some abuse of notation, we set
$$|\xi|_{g_n}^2:=g_n^{ij}\xi_i\xi_j.$$
In particular, by \eqref{eq6:prop1sect8} and \eqref{eq:unifposdefinvG_n}, we find a positive constant $\Lambda$ such that 
\begin{equation}\label{eq:2aparticular}
a_n^{ij}(x)\xi_i\xi_j \leq W_n^2(x)|\xi|_{g_n(x)}^2\leq \Lambda |\xi|^2\quad  \forall x\in\overline{\Omega_0},\ \forall n \in \N.
\end{equation}
On the other hand, since $W_n^2=1+|B_n|_{g_n}^2$, we have
$$a_n^{ij}\xi_i\xi_j=|\xi|_{g_n}^2+|B_n|_{g_n}^2|\xi|_{g_n}^2-(B_n^i\xi_i)^2.$$
By the Cauchy--Schwarz inequality, associated with the positive definite matrix $G_n(x)$, we get that
$$(B_n^i\xi_i)^2\leq ((g_n)_{ij}B_n^i B_n^j) (g_n^{rs}\xi_r \xi_s)=|B_n|_{g_n}^2|\xi|_{g_n}^2,$$
and, exploiting \eqref{eq:unifposdefinvG_n}, we conclude that there exists a positive constant $\lambda$ such that for any $\xi \in \R^{N-1}$
\begin{equation}\label{eq:2particular}
a_n^{ij}(x)\xi_i\xi_j\geq |\xi|_{g_n(x)}^2\geq \lambda |\xi|^2 \quad \forall x\in\overline{\Omega_0},\ \forall n \in \N,
\end{equation}
and \textit{Claim II} readily follows from \eqref{eq:2aparticular} and \eqref{eq:2particular}.\\

We now consider the boundary operator
$\psi\mapsto\frac{\partial \psi}{\partial \nu_{g_n}}$, appearing in
\eqref{eq:probmnondiv}. Recall that $\nu_{g_n}$ is the outward unit normal
to $\partial D_0$ with respect to the metric $g_n$. Writing in local coordinates
\begin{equation}\label{eq10:prop1sect8}
\nu_{g_n}=\beta_n^i\partial_i ,
\end{equation}
we get that
$$
\frac{\partial\psi}{\partial\nu_{g_n}}=d\psi(\nu_{g_n})=\beta_n^i\partial_i\psi
\qquad\text{on }\partial\Omega_0.
$$
Hence, in local coordinates, the boundary operator
$\psi\mapsto\frac{\partial \psi}{\partial \nu_{g_n}}$ takes the form
$$
\widetilde M_n\psi:=\beta_n\cdot \nabla\psi=\beta_n^i\partial_i\psi,
$$
where $\cdot$, $\nabla$ denote here, respectively, the standard Euclidean scalar product and the gradient in
$\mathbb R^{N-1}$, and
$$
\beta_n(x):=(\beta_n^1(x),\ldots,\beta_n^{N-1}(x)),
\qquad x\in \partial\Omega_0.
$$
In the following we show that $\widetilde M_n$ satisfies a uniformly
oblique boundary condition, with constants independent of $n$, and we
prove a uniform H\"older bound on the coefficients $\beta_n^i$.\\

\noindent\textbf{Claim III}: There exists a positive constant $\chi$ such that
$$
\beta_n(x)\cdot\nu_0(x)\ge \chi
\qquad
\forall x\in\partial\Omega_0,\quad \forall n\in\mathbb N,
$$
where $\nu_0$ is the Euclidean outward unit normal to $\partial\Omega_0$.
Moreover, there exists a positive constant $C_1$ such that
$$
\sum_{i=1}^{N-1} \| \beta_n^i\|_{C^{2,\alpha}(\partial\Omega_0)} \leq C_1
\quad \forall n\in\mathbb N.
$$

To prove the claim, let $\rho\in C^{3,\alpha}(\SN)$ be a defining
function for $D_0$ (see \eqref{eq:deffunctrho}).
On $\partial D_0$, the outward unit normal with respect to the metric
$g_n$ is therefore given by
\begin{equation}\label{eq9:prop1sect8}
\nu_{g_n}=-\frac{\nabla_{g_n}\rho}{|\nabla_{g_n}\rho|_{g_n}}.
\end{equation}
In the stereographic coordinates we still denote by $\rho$ the local
representative $\rho\circ\pi^{-1}$. Thus
$$
\Omega_0=\{\rho>0\},\qquad \partial\Omega_0=\{\rho=0\},\qquad d\rho\neq 0 \quad\text{on }\partial\Omega_0.
$$
Passing \eqref{eq9:prop1sect8} to the stereographic coordinates, and using
$$
\nabla_{g_n}\rho=g_n^{ij}\partial_j\rho\,\partial_i, \qquad |\nabla_{g_n}\rho|_{g_n}^2=g_n^{ij}\partial_i\rho\,\partial_j\rho,
$$
we obtain from \eqref{eq10:prop1sect8} that
\begin{equation}\label{eq11:prop1sect8}
\beta_n^i=-\frac{g_n^{ij}\partial_j\rho}{\left(g_n^{rs}\partial_r\rho\,\partial_s\rho\right)^{1/2}}
\qquad\text{on }\partial\Omega_0.
\end{equation}
Let $\nu_0$ be the Euclidean outward unit normal to $\partial\Omega_0$.
By construction we have
$$
\nu_0:=-\frac{\nabla \rho}{|\nabla \rho|},
$$
and thus
\begin{equation}\label{eq12:prop1sect8}
\beta_n\cdot\nu_0
=\beta_n^i\left(-\frac{\partial_i\rho}{| \nabla \rho|}\right)
=\frac{g_n^{ij}\partial_i\rho\,\partial_j\rho}{\left(g_n^{rs}\partial_r\rho\,\partial_s\rho\right)^{1/2}|\nabla \rho|}
=\frac{\left(g_n^{ij}\partial_i\rho\,\partial_j\rho\right)^{1/2}}{|\nabla \rho|}.
\end{equation}
Thanks to \eqref{eq:unifposdefinvG_n}, we have
$$
g_n^{ij}\partial_i\rho\,\partial_j\rho
\ge
\frac1{c_4}|\nabla\rho|^2
\qquad\text{on }\partial\Omega_0,
$$
and by \eqref{eq12:prop1sect8} we obtain
$$
\beta_n\cdot\nu_0
\ge
\frac1{\sqrt{c_4}}
\qquad\text{on }\partial\Omega_0,
$$
which proves the first part of the claim.

For the uniform H\"older bound on the coefficients
$\beta_n^1,\ldots,\beta_n^{N-1}$, by \eqref{eq3:prop1sect8}, and since the
local representative of $\rho$ is in
$C^{3,\alpha}(\overline{\Omega_0})$, we infer that the numerator in
\eqref{eq11:prop1sect8} is uniformly bounded in
$C^{2,\alpha}(\partial\Omega_0)$. Moreover, since $d\rho\neq 0$ on
$\partial\Omega_0$ and
$$
g_n^{ij}\partial_i\rho\,\partial_j\rho
\ge\frac1{c_4}|\nabla\rho|^2,
$$
the denominator is uniformly bounded away from zero and is uniformly
bounded in $C^{2,\alpha}(\partial\Omega_0)$. Hence there exists a positive
constant $c_{10}$, independent of $n$, such that for $i=1,\dots,N-1$
$$
\|\beta_n^i\|_{C^{2,\alpha}(\partial\Omega_0)}\leq c_{10} \quad \forall n\in\mathbb N,
$$
and \textit{Claim III} is thus proved.\\

We now conclude the proof of the proposition. By construction, $\psi=\psi_n$ solves
$$
\begin{cases}
\widetilde{\mathcal{L}}_n \psi=f_n & \hbox{in $\Omega_0$,}\\
\widetilde M_n\psi=0 & \hbox{on $\partial\Omega_0$,}
\end{cases}
$$
and it is a priori only of class $C^{2,\alpha}(\overline{\Omega_0})$. 
By definition (see \eqref{eq0:prop1sect8}), taking into account \eqref{eq6:prop1sect8}, and since $\psi_n \in \overline{\mathcal{B}_l(0)}$, we have
\begin{equation}\label{eq2:finalPropSect8}
\|f_n\|_{C^{1,\alpha}(\overline{\Omega_0})} + \|\psi_n\|_{C^0(\overline{\Omega_0})}\leq C_2,
\end{equation}
for some positive constant $C_2$, independent of $n$, but depending on $l$. 

Fix $\alpha'\in(\alpha,1)$. By \textit{Claim I}--\textit{Claim III} and \eqref{eq2:finalPropSect8}, the
coefficients $a_n^{ij},b_n^k$ and the datum $f_n$ are bounded in
$C^{1,\alpha}(\overline{\Omega_0})$, hence in
$C^{0,\alpha'}(\overline{\Omega_0})$, $\widetilde{\mathcal{L}}_n$ is uniformly elliptic with ellipticity constants independent of $n$, and the oblique field $\beta_n$ is bounded
in $C^{2,\alpha}(\partial\Omega_0)$, hence in $C^{1,\alpha'}(\partial\Omega_0)$,
uniformly in $n$. Therefore \cite[Theorem 4.40]{LiebermanODP}, applied with
$\alpha'$ in place of $\alpha$ and with the sign and normal conventions already
used above, gives $\psi_n\in C^{2,\alpha'}(\overline{\Omega_0})$ and
\begin{equation}\label{eq1:finalPropSect8}
\|\psi_n\|_{C^{2,\alpha'}(\overline{\Omega_0})} \leq C_3\bigl(\|f_n\|_{C^{0,\alpha'}(\overline{\Omega_0})}
+\|\psi_n\|_{C^0(\overline{\Omega_0})}\bigr)\qquad\forall n\in\N,
\end{equation}
with $C_3$ independent of $n$.

Then, combining \eqref{eq1:finalPropSect8} and \eqref{eq2:finalPropSect8}, by the compact embedding
$C^{2,\alpha'}(\overline{\Omega_0})\hookrightarrow C^{2,\alpha}(\overline{\Omega_0})$,
there exists
$\psi\in C^{2,\alpha}(\overline{\Omega_0})$ such that, up to a subsequence, it holds
$$
\psi_n\to\psi
\qquad\text{in }C^{2,\alpha}(\overline{\Omega_0}).
$$
Hence, we obtain convergence in $C^{2,\alpha}(\overline{D_0})$. 
Since $\psi_n\in \overline{\mathcal B_l(0)}$ and the latter is closed in $C^{2,\alpha}(\overline{D_0})$, we have $\psi\in \overline{\mathcal B_l(0)}$. Moreover, since $h_n\to h$ in $\Param$, and $F$ is continuous, passing to the limit in
$$
F(\psi_n,h_n)=(0,0)
$$
gives
$$
F(\psi,h)=(0,0).
$$
Therefore $(\psi,h)\in M_l$, and we have proved that every sequence
$(\psi_n,h_n)_n\subset M_l$ with $h_n\to h$ in $\Param$ admits a subsequence
converging in $C^{2,\alpha}(\overline{D_0})\times\Param$ to an element of $M_l$. Since $l$ is arbitrary, and \eqref{eq:Fpreimageunion} holds, we conclude that $\Pi$ is $\sigma$-proper.
\end{proof}

\section{Preliminary results for transversality}

Let $\psi\in C^{2,\alpha}(\overline{D_0})$ be a solution of
$$
F(\psi,\Id)=(0,0).
$$
As in Section 5, we work at the identity value of the domain parameter. The case of a general parameter will be reduced to the identity case in
Section 8 (see the proof of Proposition \ref{prop1sect8}).

We consider the bounded extension
$$
D_\psi F(\psi,\Id): H^2(D_0)\longrightarrow L^2(D_0)\times L^2(\partial D_0)
$$
introduced in Section 5, and denote by $(D_\psi F(\psi,\Id))'$ its adjoint, whose kernel $\ker\bigl((D_\psi F(\psi,\Id))'\bigr)$ is a subspace of $L^2(D_0)\times L^2(\partial D_0)$.

The main result of this section is the following.

\begin{theorem}\label{thm1sect7}
Assume that $N-1$ does not belong to the Neumann spectrum of $-\Delta_{g_0}$ on $D_0$.
Let $H>0$, and let $\psi\in C^{2,\alpha}(\overline{D_0})$ be a solution of $F(\psi,\Id)=(0,0)$.
Then, for any $(w,\eta)\in\ker\bigl((D_\psi F(\psi,\Id))'\bigr)$ such that
\beq\label{eq:statementteo1sect7}
\bigl(D_hF(\psi,\Id)[Y],(w,\eta)\bigr)_{L^2(D_0)\times L^2(\partial D_0)}=0,
\eeq
for every $Y\in\Y$, we have $(w,\eta)=(0,0)$.
\end{theorem}

We begin with two preliminary results which, for a fixed $Y\in\Y$, give two equivalent formulations of the identity in \eqref{eq:statementteo1sect7}.

\begin{lemma}\label{lem1eq:sect7}
Let $\psi\in C^{2,\alpha}(\overline{D_0})$ be a solution of $F(\psi,\Id)=(0,0)$, let $(w,\eta)\in \ker\bigl((D_\psi F(\psi,\Id))'\bigr)$, $Y \in \Y$ and set
$$\delta_Y g_0:=\mathcal{L}_Yg_0,\quad W_\psi:=\sqrt{1+|\nabla_{g_0}\psi|_{g_0}^2},$$
where $\mathcal L_Yg_0$ is the Lie derivative of $g_0$ along $Y$ (see \eqref{eq:defLieDerivative}). 
Then \eqref{eq:statementteo1sect7} holds if and only if
\begin{equation}\label{eq:trasvcondition2}
\begin{array}{lll}
&&\displaystyle -\int_{D_0}\frac{\delta_Y g_0(\nabla_{g_0}\psi,\nabla_{g_0} w)}{W_\psi}\,d\mu_{g_0} + \int_{D_0}\frac{\delta_Y g_0(\nabla_{g_0}\psi,\nabla_{g_0}\psi)}{2W_\psi^3}\,\langle \nabla_{g_0}\psi,\nabla_{g_0} w\rangle_{g_0}\,d\mu_{g_0}\\[12pt]
&&\displaystyle+\frac12\int_{D_0}(\tr_{g_0} \delta_Y g_0)\, \divg_{g_0}\left(\frac{w}{W_\psi}\nabla_{g_0}\psi\right)\,d\mu_{g_0}+\frac{N-1}{2}\int_{D_0}\frac{\delta_Y g_0(\nabla_{g_0}\psi,\nabla_{g_0}\psi)}{W_\psi^3}\,w \,d\mu_{g_0}=0.
\end{array}
\end{equation}
\end{lemma}

\begin{proof}
Let $\psi\in C^{2,\alpha}(\overline{D_0})$ be such that $F(\psi,\Id)=(0,0)$, let $(w,\eta)\in \ker\bigl((D_\psi F(\psi,\Id))'\bigr)$, and let $Y \in \Y$. By definition of $(\cdot,\cdot)_{L^2(D_0)\times L^2(\partial D_0)}$, the relation \eqref{eq:statementteo1sect7} means that
\begin{equation}\label{eq:trasvcond}
\int_{D_0}  D_hF_1(\psi,\Id)[Y] \;w\,d\mu_{g_0} +\int_{\partial D_0}  D_hF_2(\psi,\Id)[Y]\;\eta\,d\sigma_{g_0}=0.
\end{equation}
By Lemma \ref{lem3sect4}, writing for simplicity $\nabla \psi$, $\nu$, $\langle\cdot,\cdot\rangle$ instead of $\nabla_{g_0} \psi$, $\nu_{g_0}$, $\langle\cdot,\cdot\rangle_{g_0}$ respectively, we have

\begin{equation}\label{eq:partialdrF}
\begin{array}{lll}
\displaystyle D_hF_1(\psi,\Id)[Y]
&=& \displaystyle
{\divg}_{g_0} \left(\frac{(\delta_Y g_0(\nabla\psi,\cdot))^\sharp}{W_\psi}\right) - {\divg}_{g_0}\left(\frac{\delta_Y g_0(\nabla\psi,\nabla\psi)}{2W_\psi^3}\nabla\psi \right)\\[15pt]
&&\displaystyle- \frac{1}{2W_\psi} \langle\nabla (\tr_{g_0} \delta_Y g_0),\nabla\psi\rangle
+\frac{N-1}{2W_\psi^3}\delta_Y g_0(\nabla\psi,\nabla\psi),\\[15pt]
\displaystyle D_hF_2(\psi,\Id)[Y]
&=& \displaystyle -\delta_Y g_0(\nu,\nabla^\top\psi)-\frac12 \delta_Y g_0(\nu,\nu)\partial_{\nu}\psi,
\end{array}
\end{equation}
where $\bigl(\delta_Yg_0(\nabla_{g_0}\psi,\cdot)\bigr)^\sharp$ is the $g_0$-metric dual vector field of the one-form $\delta_Yg_0(\nabla \psi,\cdot)$ (see \eqref{eq:defomegasharp}).
Since $\psi$ satisfies $F(\psi,\Id)=(0,0)$, in particular, it holds that $\partial_\nu\psi=0$ on $\partial D_0$, and thus

$$D_hF_2(\psi,\Id)[Y]= -\delta_Y g_0(\nu,\nabla^\top\psi),$$
where $\nabla^\top\psi$ denotes the tangential component of $\nabla\psi$ on $\partial D_0$.
Moreover, by Theorem \ref{theo1sect5} and Corollary \ref{cor:regwKernel}, we
have $w\in C^{2,\alpha}(\overline{D_0})$, and $\partial_\nu w=0$, $\eta=\frac{w}{W_\psi}$ on $\partial D_0$. In particular, the boundary integral in \eqref{eq:trasvcond} can be rewritten as
\begin{equation}\label{eq:intterm0}
\int_{\partial D_0}  D_hF_2(\psi,\Id)[Y]\;\eta\,d\sigma_{g_0} = -\int_{\partial D_0} \frac{w}{W_\psi}\,\delta_Y g_0(\nu,\nabla^\top\psi)\,d\sigma_{g_0}.
\end{equation}
For the first integral in \eqref{eq:trasvcond}, by \eqref{eq:partialdrF}, integrating by parts and recalling the definition of $(\delta_Y g_0(\nabla\psi,\cdot))^\sharp$ we have
\begin{equation}\label{eq:intterm1}
\begin{array}{lll}
&&\displaystyle \int_{D_0}\divg_{g_0}\left(\frac1{W_\psi} (\delta_Y g_0(\nabla\psi,\cdot))^\sharp\right)w\,d\mu_{g_0}\\[12pt]
&=& \displaystyle
-\int_{D_0}\frac1{W_\psi}\,\delta_Y g_0(\nabla\psi,\nabla w)\,d\mu_{g_0} + \int_{\partial D_0}\frac{w}{W_\psi}\,\delta_Y g_0(\nabla\psi,\nu)\,d\sigma_{g_0}\\[12pt]
&=& \displaystyle -\int_{D_0}\frac1{W_\psi}\,\delta_Y g_0(\nabla\psi,\nabla w)\,d\mu_{g_0} + \int_{\partial D_0}\frac{w}{W_\psi}\,\delta_Y g_0(\nabla^\top\psi,\nu)\,d\sigma_{g_0},
\end{array}
\end{equation}
where for the last equality we used that $\partial_\nu\psi=0$ on $\partial D_0$, which gives that $\nabla\psi=\nabla^\top\psi$ on $\partial D_0$.

For the second term in the expression of $D_hF_1$, integrating again by parts, taking into account that $\partial_\nu\psi=0$ on $\partial D_0$, we obtain
$$
-\int_{D_0}\divg_{g_0}\left(\frac{\delta_Y g_0(\nabla\psi,\nabla\psi)}{2W_\psi^3}\nabla\psi\right)w\,d\mu_{g_0} = \int_{D_0}\frac{\delta_Y g_0(\nabla\psi,\nabla\psi)}{2W_\psi^3}\,\langle \nabla\psi,\nabla w\rangle\,d\mu_{g_0}.
$$
Similarly, for the trace term, we get that
\begin{equation}\label{eq:intterm3}
-\frac12\int_{D_0}\,\frac{\langle \nabla(\tr_{g_0} \delta_Y g_0),\nabla\psi\rangle}{W_\psi}\, w\, d\mu_{g_0}=\frac12\int_{D_0}(\tr_{g_0} \delta_Y g_0)\, \divg_{g_0}\left(\frac{w}{W_\psi}\nabla\psi\right)\,d\mu_{g_0}.
\end{equation}
Summing up, from \eqref{eq:partialdrF}--\eqref{eq:intterm3}, taking into account that the boundary integrals in \eqref{eq:intterm0} and \eqref{eq:intterm1} cancel, we obtain \eqref{eq:trasvcondition2}. Finally, since all the relations  \eqref{eq:trasvcond}--\eqref{eq:intterm3} are equalities, it is clear that \eqref{eq:trasvcondition2} is equivalent to \eqref{eq:statementteo1sect7} and the proof is complete.
\end{proof}
Thanks to the previous result we can prove another equivalent characterization of \eqref{eq:statementteo1sect7} that will be useful in the sequel of this section. More precisely we have the following.

\begin{lemma}\label{lem2eq:sect7}
Let $\psi\in C^{2,\alpha}(\overline{D_0})$ be a solution of $F(\psi,\Id)=(0,0)$, let $(w,\eta)\in \ker\bigl((D_\psi F(\psi,\Id))'\bigr)$, $Y \in \Y$ and let $\delta_Y g_0$, $W_\psi$ be defined as in Lemma \ref{lem1eq:sect7}. 
Then \eqref{eq:statementteo1sect7} holds if and only if

\begin{equation}\label{eq:trasvcondition3}
\int_{D_0}\langle \mathcal T_\psi[w],\, \delta_Y g_0\rangle_{g_0}\,d\mu_{g_0}=0
\end{equation}
where $\mathcal{T}_\psi[w]$ is the $(0,2)$-tensor field defined by
\begin{equation}\label{eq:tensorfieldTpsi}
\displaystyle \mathcal T_\psi[w]:=\displaystyle-\frac1{W_\psi}\,\operatorname{Sym}(d\psi\otimes dw)
+\frac{\langle \nabla_{g_0}\psi,\nabla_{g_0} w\rangle_{g_0}+(N-1) w}{2W_\psi^3}\,d\psi \otimes d\psi+\frac12\divg_{g_0}\left(\frac{w}{W_\psi}\nabla_{g_0}\psi\right)g_0.
\end{equation}
Moreover
\begin{equation}\label{eq:TwH1}
 \mathcal{T}_\psi[w]\in C^{1,\alpha}(\overline{D_0}\,, T^*\SN\otimes T^*\SN).
\end{equation}
\end{lemma}
\begin{proof}
Thanks to Lemma \ref{lem1eq:sect7} the relation \eqref{eq:statementteo1sect7} is equivalent to \eqref{eq:trasvcondition2}.  We now show that \eqref{eq:trasvcondition2} is equivalent to \eqref{eq:trasvcondition3}. 
To ease the notation we write $\nabla \cdot$ instead of $\nabla_{g_0} \cdot$.

We first observe that, since the Lie derivative of a Riemannian metric is again a symmetric $(0,2)$-tensor, then $\delta_Y g_0$ is symmetric. Hence, applying \eqref{eq:tensorpairing2} with
$$\omega=d\psi, \qquad \theta=dw, \qquad T=\delta_Y g_0,$$
and taking into account \eqref{eq1:appendix1}, we obtain
\begin{equation}\label{eq:tensorpairing3}
\bigl\langle \operatorname{Sym}(\dd\psi\otimes \dd w),\delta_Y g_0\bigr\rangle_{g_0} = \delta_Y g_0(\nabla\psi,\nabla w),
\end{equation}
where $\langle \cdot, \cdot\rangle_{g_0}$ denotes here the scalar product between tensor fields induced by $g_0$ (see \eqref{eq:defscalarproducttensors} for the definition).

Similarly, by \eqref{eq:tensorpairing1} we get that
\begin{equation}\label{eq:tensorpairing4}
\bigl\langle\dd\psi\otimes\dd\psi,\delta_Y g_0\bigr\rangle_{g_0}=\delta_Y g_0(\nabla\psi,\nabla\psi).
\end{equation}

Next, recalling that the trace of a $(0,2)$-tensor $T$ with respect to $g_0$ is defined by
$$\tr_{g_0}T=g_0^{ij}T_{ij},$$
then, we deduce that
$$
\langle g_0,T\rangle_{g_0}=\tr_{g_0}T.
$$
In particular, for any scalar function $\varphi$, it holds that
\begin{equation}\label{eq:tensorpairing6}
\langle \varphi g_0,\delta_Y g_0\rangle_{g_0} = \varphi\,\tr_{g_0}\delta_Y g_0.
\end{equation}

Now, thanks to \eqref{eq:tensorpairing3}--\eqref{eq:tensorpairing6} we can rewrite each integrand in \eqref{eq:trasvcondition2} as the scalar product of a tensor field against $\delta_Y g_0$. Indeed, by \eqref{eq:tensorpairing3}, we have
$$
-\frac1{W_\psi}\delta_Y g_0(\nabla\psi,\nabla w) = \left\langle -\frac1{W_\psi}\operatorname{Sym}(\dd\psi \otimes \dd w), \,\delta_Y g_0 \right\rangle_{g_0},
$$
and by \eqref{eq:tensorpairing4}, it holds
$$
\frac{\langle \nabla\psi,\nabla w\rangle_{g_0}+(N-1)w}{2W_\psi^3}\, \delta_Y g_0(\nabla\psi,\nabla\psi)
= \left\langle \frac{\langle \nabla\psi,\nabla w\rangle_{g_0}+(N-1)w}{2W_\psi^3}\, \dd\psi \otimes \dd\psi,\,\delta_Y g_0\right\rangle_{g_0}.
$$
Finally, by \eqref{eq:tensorpairing6}, we get that
$$
\frac12(\tr_{g_0}\delta_Y g_0)\, \divg_{g_0}\left(\frac{w}{W_\psi}\nabla\psi\right)=\left\langle \frac12\divg_{g_0}\left(\frac{w}{W_\psi}\nabla\psi\right)g_0, \,\delta_Y g_0\right\rangle_{g_0}.
$$
Therefore, summing up the three previous identities, the relation \eqref{eq:trasvcondition2} can be equivalently rewritten as
$$
\int_{D_0}\langle \mathcal T_\psi[w],\delta_Y g_0\rangle_{g_0}\,d\mu_{g_0}=0,
$$
where $\mathcal T_\psi[w]$ is the tensor field given by \eqref{eq:tensorfieldTpsi}.

To conclude it remains to prove \eqref{eq:TwH1}. The first two terms in
\eqref{eq:tensorfieldTpsi} belong to
$$
C^{1,\alpha}(\overline{D_0},T^*\SN\otimes T^*\SN),
$$
since $\psi\in C^{2,\alpha}(\overline{D_0})$ and
$w\in C^{2,\alpha}(\overline{D_0})$ (see Corollary \ref{cor:regwKernel}). Let us only discuss
the last term. Since $\psi$ solves $F(\psi,\Id)=(0,0)$, we have
$$
\divg_{g_0}\left(\frac{\nabla\psi}{W_\psi}\right) = \frac{N-1}{W_\psi}-(N-1)He^\psi.
$$
Therefore
$$
\divg_{g_0}\left(\frac{w}{W_\psi}\nabla\psi\right)=\frac{1}{W_\psi}\langle\nabla w,\nabla\psi\rangle_{g_0}+w\left(\frac{N-1}{W_\psi}-(N-1)He^\psi\right).
$$
The right-hand side belongs to $C^{1,\alpha}(\overline{D_0})$, by the regularity of $\psi$ and $w$. Hence also the last term in \eqref{eq:tensorfieldTpsi} belongs to the desired space, and \eqref{eq:TwH1} follows. The proof is complete.
\end{proof}

\begin{lemma}\label{lem:divT-zero}
Let $\psi\in C^{2,\alpha}(\overline{D_0})$ be a solution of $F(\psi,\Id)=(0,0)$, let $(w,\eta)\in \ker\bigl((D_\psi F(\psi,\Id))'\bigr)$, let $\mathcal{T}_\psi[w]$ be the tensor field defined in \eqref{eq:tensorfieldTpsi}, and suppose that \eqref{eq:statementteo1sect7} holds for all $Y\in \Y$.
Then
$$
\divg_{g_0} \mathcal{T}_\psi[w]=0 \qquad\text{in }D_0.
$$
\end{lemma}

\begin{proof}
Let us first observe that by definition $\mathcal{T}_\psi[w]$ is a symmetric $(0,2)$-tensor field and let us denote by $\mathcal{T}_{ij}$ the components of $\mathcal{T}_\psi[w]$ in local coordinates, i.e.
$$\mathcal{T}_\psi[w] =\mathcal{T}_{ij}\, dx^i\otimes dx^j.$$
Moreover, recalling that $\delta_Y g_0=\mathcal{L}_Yg_0$ and \eqref{eq:liederiv}, then, in local coordinates, we have
\beq\label{eq:complocalhY}
(\delta_Y g_0)_{ij}=\nabla_iY_j+\nabla_jY_i.
\eeq
Now, denoting by $\mathcal{T}^{ij}$ the components of $\mathcal{T}_\psi[w]$ with both indices raised by the metric $g_0$, namely $$\mathcal{T}^{ij}=g_0^{ai}g_0^{bj}\mathcal{T}_{ab},$$ 
then, by definition of scalar product between tensors (see \eqref{eq:defscalarproducttensors}), and \eqref{eq:complocalhY}, we get that
$$\langle \mathcal{T}_\psi[w],\delta_Y g_0\rangle_{g_0}=\mathcal{T}^{ij}(\nabla_iY_j+\nabla_jY_i).$$
In addition, since $\mathcal{T}_\psi[w]$ is symmetric then $\mathcal{T}^{ij}=\mathcal{T}^{ji}$, and we conclude that
\beq\label{eq:prodscalTsymmYsmooth}
\langle \mathcal{T}_\psi[w], \delta_Y g_0\rangle_{g_0}=2\mathcal{T}^{ij}\nabla_iY_j.
\eeq
By assumption, the relation \eqref{eq:statementteo1sect7} holds for every $Y\in \Y$. In particular, it holds for every $Y\in C_c^\infty(D_0,TD_0)\subset \Y$. Hence, by Lemma \ref{lem2eq:sect7} for every such $Y$,
\begin{equation}\label{eq1:lem1sect7}
0 = \int_{D_0}\langle \mathcal{T}_\psi[w],\delta_Y g_0\rangle_{g_0}\,d\mu_{g_0} = 2\int_{D_0}\mathcal{T}^{ij}\nabla_iY_j\,d\mu_{g_0}.
\end{equation}
On the other hand, as $\mathcal{T}_\psi[w]\in C^{1,\alpha}(\overline{D_0}, T^*\SN\otimes T^*\SN)$ (see \eqref{eq:TwH1}), then integrating by parts, using that $Y$ is compactly supported in $D_0$, and recalling the definition of divergence of a tensor field (see \eqref{eq:defDivTensor}), we infer that
\begin{equation}\label{eq2:lem1sect7}
\int_{D_0}\mathcal{T}^{ij}\nabla_iY_j\,d\mu_{g_0}
= -\int_{D_0}(\nabla_i\mathcal{T}^{ij})Y_j\,d\mu_{g_0} =-\int_{D_0}\langle \divg_{g_0}\mathcal{T}_\psi[w],Y\rangle_{g_0}\,d\mu_{g_0}.
\end{equation}
Then, combining \eqref{eq1:lem1sect7} and \eqref{eq2:lem1sect7} we conclude that
$$
\int_{D_0}\langle \divg_{g_0}\mathcal{T}_\psi[w],Y\rangle_{g_0}\,d\mu_{g_0}=0 \qquad \forall\,Y\in C_c^\infty(D_0,TD_0).$$
Finally, as the test field $Y$ is arbitrary, we obtain
$$\divg_{g_0}\mathcal{T}_\psi[w]=0 \qquad\text{in }D_0,$$
and the proof is complete.
\end{proof}

\begin{lemma}\label{lem:Tnunuzero}
Under the same assumptions as Lemma \ref{lem:divT-zero}, it holds that
\begin{equation}\label{eq:Tnunu-zero}
\mathcal T_\psi[w](\nu_{g_0},\nu_{g_0})=0 \qquad\text{on }\partial D_0.
\end{equation}
\end{lemma}

\begin{proof}
For every $Y\in\Y$, restricting $Y$ to $\overline{D_0}$, using \eqref{eq:prodscalTsymmYsmooth} and integrating by parts we get
$$
\begin{aligned}
\int_{D_0} \langle\mathcal T_\psi[w],\delta_Yg_0\rangle_{g_0}\,d\mu_{g_0}=
&-2\int_{D_0}\langle\divg_{g_0}\mathcal T_\psi[w],Y\rangle_{g_0}\,d\mu_{g_0}\\
&+2\int_{\partial D_0}\mathcal T_\psi[w](\nu_{g_0},Y)\,d\sigma_{g_0}.
\end{aligned}
$$
By Lemma \ref{lem:divT-zero}, this reduces to
\begin{equation}\label{eq:boundaryidT}
\int_{D_0}
\langle\mathcal T_\psi[w],\delta_Yg_0\rangle_{g_0}
\,d\mu_{g_0}=2\int_{\partial D_0} \mathcal T_\psi[w](\nu_{g_0},Y) \,d\sigma_{g_0}.
\end{equation}
On the other hand, since \eqref{eq:statementteo1sect7} holds for every
$Y\in\Y$, then, by Lemma \ref{lem2eq:sect7}, we have
$$
\int_{D_0} \langle\mathcal T_\psi[w],\delta_Yg_0\rangle_{g_0} \,d\mu_{g_0} =0\qquad \forall\, Y\in\Y.
$$
Hence, by \eqref{eq:boundaryidT}, we get
\begin{equation}\label{eq:boundarytraction0}
\int_{\partial D_0}\mathcal T_\psi[w](\nu_{g_0},Y) \,d\sigma_{g_0}=0\qquad \forall\, Y\in\Y.
\end{equation}

Fix $p\in\partial D_0$. Since $\partial D_0\Subset Q$, we may choose an
open neighborhood $U$ of $p$ such that
$$
\overline U\subset Q.
$$
We may also choose a smooth vector field
$$
Z\in C_c^\infty(U,T\SN)
$$
such that
$$
Z(p)=\nu_{g_0}(p).
$$
For every $\varphi\in C_c^\infty(U)$, the vector field
$$
Y:=\varphi Z
$$
belongs to $\Y$. Hence \eqref{eq:boundarytraction0} gives
$$
\int_{\partial D_0} \varphi\, \mathcal T_\psi[w](\nu_{g_0},Z)\,d\sigma_{g_0}=0\qquad \forall \varphi\in C_c^\infty(U).
$$
Then, since the map
$$
q\longmapsto (\mathcal{T}_\psi[w])_q(\nu_{g_0}(q),Z(q))
$$
is continuous on $\partial D_0\cap \overline{U}$, by the arbitrariness of
$\varphi$, we infer that
$$
\mathcal T_\psi[w](\nu_{g_0},Z)=0
\qquad\text{on }\partial D_0\cap U.
$$
In particular, evaluating at $p$ and recalling that
$Z(p)=\nu_{g_0}(p)$, we obtain
$$
(\mathcal T_\psi[w])_p(\nu_{g_0}(p),\nu_{g_0}(p))=(\mathcal T_\psi[w])_p(\nu_{g_0}(p),Z(p))=0.
$$
Finally, since $p\in\partial D_0$ is arbitrary, \eqref{eq:Tnunu-zero} follows and the proof is complete.
\end{proof}

\begin{lemma}\label{lem:Tnunu-boundary-formula}
Under the assumptions of Lemma \ref{lem:Tnunuzero}, it holds
\begin{equation}\label{eq:TnunudivV}
\mathcal T_\psi[w](\nu_{g_0},\nu_{g_0})=\frac12\divg_{g_0}\left(\frac{w}{W_\psi}\nabla_{g_0}\psi\right)
\qquad\text{on }\partial D_0.
\end{equation}
In particular, by Lemma \ref{lem:Tnunuzero}, we have 
\begin{equation}\label{eq:divV0boundary}
\divg_{g_0}\left(\frac{w}{W_\psi}\nabla_{g_0}\psi\right)=0 \qquad\text{on }\partial D_0.
\end{equation}
\end{lemma}

\begin{proof}
We evaluate each term of the tensor field $\mathcal T_\psi[w]$ on $(\nu_{g_0},\nu_{g_0})$. For the first term, by definition (see \eqref{eq:tensorfieldTpsi}, \eqref{eq:defalphatensorbeta} and \eqref{eq:defsymalphatensorbeta}), we get
$$
\operatorname{Sym}(\dd\psi\otimes\dd w)(\nu_{g_0},\nu_{g_0})=\partial_{\nu_{g_0}}\psi\,\partial_{\nu_{g_0}}w.
$$
Since $\psi$ satisfies $F(\psi,\Id)=(0,0)$, in particular, we have $\partial_{\nu_{g_0}}\psi=0$ on $\partial D_0$. Hence
$$
\operatorname{Sym}(\dd\psi\otimes\dd w)(\nu_{g_0},\nu_{g_0})=0
\qquad\text{on }\partial D_0.
$$

For the second term we have
$$
(\dd\psi\otimes\dd \psi)(\nu_{g_0},\nu_{g_0})=(\partial_{\nu_{g_0}}\psi)^2=0 \qquad\text{on }\partial D_0.
$$
Therefore, on $\partial D_0$, only the last term survives, and we get
$$
\mathcal T_\psi[w](\nu_{g_0},\nu_{g_0})=\frac12\divg_{g_0}\left(\frac{w}{W_\psi}\nabla_{g_0}\psi\right)g_0(\nu_{g_0},\nu_{g_0}).
$$
Since $g_0(\nu_{g_0},\nu_{g_0})=1$, this proves \eqref{eq:TnunudivV}.

Finally, combining \eqref{eq:TnunudivV} with \eqref{eq:Tnunu-zero}, we obtain
\eqref{eq:divV0boundary}, and the proof is complete.
\end{proof}

From the previous results, we deduce the following crucial identity between $w$ and $\psi$ on $\partial D_0$.

\begin{proposition}\label{prop1:sect7}
Let $\psi\in C^{2,\alpha}(\overline{D_0})$ be a solution of $F(\psi,\Id)=(0,0)$, let $(w,\eta)\in \ker\bigl((D_\psi F(\psi,\Id))'\bigr)$, and suppose that \eqref{eq:statementteo1sect7} holds for all $Y\in \Y$.
Then 
\begin{equation}\label{eq1:statprop1sect7}
\langle \nabla_{g_0}^\top w,\nabla_{g_0}^\top\psi\rangle_{g_0}=(N-1)\bigl(He^\psi W_\psi-1\bigr)w \qquad\text{on }\partial D_0.
\end{equation}
\end{proposition}

\begin{proof}
We recall that by Theorem \ref{theo1sect5} and Corollary \ref{cor:regwKernel}, we
have $w\in C^{2,\alpha}(\overline{D_0})$ and $\partial_{\nu_{g_0}} w=0$ on $\partial D_0$.
By Lemma \ref{lem:Tnunu-boundary-formula}, we have
$$
\divg_{g_0}\left(\frac{w}{W_\psi}\nabla_{g_0}\psi\right)=0 \qquad\text{on }\partial D_0.
$$
On the other hand, by the product rule, we get
$$
\divg_{g_0}\left(\frac{w}{W_\psi}\nabla_{g_0}\psi\right)=\frac{1}{W_\psi} \langle\nabla_{g_0}w,\nabla_{g_0}\psi\rangle_{g_0}+w\,\divg_{g_0}\left(\frac{\nabla_{g_0}\psi}{W_\psi}\right).
$$
Now, since the first component of $F(\psi,\Id)=(0,0)$ gives
$$
-\divg_{g_0}\left(\frac{\nabla_{g_0}\psi}{W_\psi}\right)+\frac{N-1}{W_\psi}-(N-1)He^\psi=0\qquad\text{in }D_0,
$$
and all the terms in this identity are continuous on $\overline{D_0}$, the same identity holds on $\partial D_0$. 
Hence
\beq\label{eq1:transpeqlemma}
0=\frac{1}{W_\psi}\langle\nabla_{g_0}w,\nabla_{g_0}\psi\rangle_{g_0}+(N-1)\left(\frac{1}{W_\psi}-He^\psi\right)w\qquad\text{on }\partial D_0.
\eeq
Moreover, since
$$
\partial_{\nu_{g_0}}\psi=\partial_{\nu_{g_0}}w=0\qquad\text{on }\partial D_0,
$$
we have
\beq\label{eq2:transpeqlemma}
\langle\nabla_{g_0}w,\nabla_{g_0}\psi\rangle_{g_0}=\langle\nabla_{g_0}^\top w,\nabla_{g_0}^\top\psi\rangle_{g_0}
\qquad\text{on }\partial D_0.
\eeq
Multiplying \eqref{eq1:transpeqlemma} by $W_\psi$ and using \eqref{eq2:transpeqlemma}, we obtain
$$
\langle\nabla_{g_0}^\top w,\nabla_{g_0}^\top\psi\rangle_{g_0}=(N-1)\bigl(He^\psi W_\psi-1\bigr)w
\qquad\text{on }\partial D_0,
$$
as claimed. The proof is complete.
\end{proof}

We conclude this section by proving Theorem \ref{thm1sect7}.

\begin{proof}[Proof of Theorem \ref{thm1sect7}]
Let $\psi\in C^{2,\alpha}(\overline{D_0})$ be a solution of $F(\psi,\Id)=(0,0)$, and let $(w,\eta)$ be an element of $\ker\bigl((D_\psi F(\psi,\Id))'\bigr)$. 
By Theorem \ref{theo1sect5} and Corollary \ref{cor:regwKernel}, the pair $(w,\eta)$ satisfies \eqref{eq1:statteo1sect5}, with
$$
w\in C^{2,\alpha}(\overline{D_0}), \qquad \eta\in C^{1,\alpha}(\partial D_0).
$$
Let us set
$$
\psi_H:=\log\frac1H
$$
and observe that at least one of the following three alternatives occurs:
\begin{enumerate}
\item[(i)]
$$
\psi\equiv \psi_H
\qquad \text{on }\partial D_0,
$$
\item[(ii)]
$$
\max_{\partial D_0}\psi>\psi_H,
$$
\item[(iii)]
$$
\min_{\partial D_0}\psi<\psi_H.
$$
\end{enumerate}
To conclude, we show that in each of these three cases the only possibility is $(w,\eta)=(0,0)$.
\medskip

\noindent
\textbf{Case (i): $\psi\equiv \psi_H$ on $\partial D_0$.}
\medskip

Set
$$
u:=\psi-\psi_H.
$$
By construction,
$$
u=0
\qquad \text{on }\partial D_0.
$$
Moreover, since $\partial_{\nu_{g_0}}\psi=0$ on $\partial D_0$, we also have
$$
\partial_{\nu_{g_0}}u=0
\qquad \text{on }\partial D_0.
$$
Since $\psi$ solves
\begin{equation}\label{eq:CMCradialgraphs}
-\divg_{g_0}\!\left(\frac{ \nabla_{g_0} \psi}{W_\psi}\right)+\frac{N-1}{W_\psi}-(N-1)He^\psi=0 \qquad \text{in }D_0,
\end{equation}
and the constant function $\psi_H$ is also a solution of \eqref{eq:CMCradialgraphs}, we readily check that $u$ satisfies
\begin{equation}\label{eq:u-difference-self}
-\divg_{g_0}\!\left(\frac{ \nabla_{g_0} u}{W_\psi}\right)+(N-1)\left(\frac1{W_\psi}-1\right)-(N-1)(e^u-1)=0.
\end{equation}
By a straightforward computation, rewriting the last two terms in \eqref{eq:u-difference-self} linearly in $\nabla_{g_0} u$ and $u$, we check that $u$ satisfies a linear uniformly elliptic equation of the form
$$
-\divg_{g_0}(a\, \nabla_{g_0} u)+\langle b, \nabla_{g_0} u\rangle_{g_0}-c\,u=0 \qquad \text{in }D_0,
$$
where
$$
a=\frac{1}{W_\psi},
\qquad
b=-(N-1)\int_0^1 \frac{t\, \nabla_{g_0} u}{(1+t^2| \nabla_{g_0} u|_{g_0}^2)^{3/2}}\,dt, \qquad c=(N-1)\int_0^1e^{tu}\,dt.
$$
Since
$$u=0, \qquad \partial_{\nu_{g_0}}u=0 \qquad \text{on }\partial D_0,$$
by a classical result of Calder\'on, also known as Calder\'on's uniqueness theorem (see \cite{Calderon}), it follows that $u$ vanishes in an interior neighborhood of $\partial D_0$. More precisely, in our framework, up to passing to local coordinates, we apply the version stated in \cite[Theorem 3.3.1]{Isakov}. Hence, by interior unique continuation (see \cite[Section 3.3]{Isakov} or \cite{GarofaloLin}) we get that $u\equiv 0$ in $D_0$. Therefore
$$
\psi\equiv \psi_H \qquad\text{in }D_0.
$$

Thus $ \nabla_{g_0}\psi\equiv 0$, and referring to the notation of \eqref{eq:simplnotation}, we have
$$
W_\psi\equiv 1, \qquad b_\psi=0, \qquad c_\psi=(N-1)He^\psi=N-1,
$$
in particular $A_\psi$ is the identity tensor field.
Therefore the adjoint problem \eqref{eq1:statteo1sect5} reduces to
\begin{equation}\label{eq:reducedadjoint}
\begin{cases}
-\Delta_{g_0} w =(N-1)\,w & \text{in }D_0,\\[0.4em]
\partial_{\nu_{g_0}}w=0 & \text{on }\partial D_0,\\[0.4em]
\eta=w & \text{on }\partial D_0.
\end{cases}
\end{equation}
Thus $w$ is a Neumann eigenfunction of $-\Delta_{g_0}$ with eigenvalue $N-1$, and by hypothesis the only possibility is
$$
w\equiv 0 \qquad \text{in }D_0.
$$
Then, by the last relation in \eqref{eq:reducedadjoint}, we also obtain
$$
\eta=0 \qquad \text{on }\partial D_0.
$$
Hence $(w,\eta)=(0,0)$.

\medskip

\noindent
\textbf{Case (ii): $\max_{\partial D_0}\psi>\psi_H$.}

\medskip

Set
$$
E_+:=\{p\in\partial D_0;\,\psi(p)>\psi_H\}.
$$
By assumption, $E_+$ is a nonempty open subset of $\partial D_0$. Let $p_0\in E_+$, and let $\sigma_{p_0}$ be the integral curve of the vector field $\nabla_{g_0}^\top\psi$ on $\partial D_0$ such that $$\sigma_{p_0}(0)=p_0.$$ 
Since $\partial D_0$ is compact, $\sigma_{p_0}$ is defined for every $t\in\mathbb R$. Then, for all $t\ge 0$, it holds that
$$
\frac{d}{dt}\psi(\sigma_{p_0}(t)) = |\nabla_{g_0}^\top\psi|_{g_0}^2(\sigma_{p_0}(t)) \geq 0,
$$
and, since $\psi(p_0)>\psi_H$, it follows that
$$
He^{\psi(\sigma_{p_0}(t))}\ge He^{\psi(p_0)}>He^{\psi_H}=1 \qquad \forall t\ge 0.
$$
In particular, as $W_\psi\geq 1$, we have
$$
He^{\psi(\sigma_{p_0}(t))}W_\psi(\sigma_{p_0}(t))\ge He^{\psi(p_0)}>1 \qquad \forall t\ge 0.
$$
Setting
$$
q_\psi(p):=(N-1)(He^{\psi(p)} W_\psi(p)-1),
$$
we infer that for all $t\ge 0$
$$
q_\psi(\sigma_{p_0}(t)) \geq (N-1)\bigl(He^{\psi(p_0)}-1\bigr)=:c_*>0.
$$
On the other hand, since \eqref{eq:statementteo1sect7} holds by assumption, then, by Proposition \ref{prop1:sect7} and by definition of $\sigma_{p_0}$, we deduce that
$$
\frac{d}{dt}w(\sigma_{p_0}(t))=\langle \nabla_{g_0}^\top w(\sigma_{p_0}(t)),\sigma_{p_0}^\prime(t)\rangle_{g_0} =q_\psi(\sigma_{p_0}(t))\,w(\sigma_{p_0}(t)).
$$
Therefore, for all $t\ge 0$,
$$
w(\sigma_{p_0}(t)) = w(p_0)e^{\int_0^t q_\psi(\sigma_{p_0}(s))\,ds},
$$
and thus
$$
|w(\sigma_{p_0}(t))| \geq |w(p_0)|e^{c_* t}.
$$
Since $w$ is continuous on the compact set $\partial D_0$, it is bounded there. Hence the previous estimate is possible only if
$$
w(p_0)=0.
$$
Since $p_0\in E_+$ was arbitrary, we conclude that
$$
w\equiv 0 \qquad \text{on } E_+.
$$
Moreover, since $\partial_{\nu_{g_0}}w=0$ on $\partial D_0$, we have
\beq\label{sect7:eqwwnu0}
w=0, \qquad \partial_{\nu_{g_0}}w=0 \qquad \text{on }E_+.
\eeq
Since $w$ solves the linear elliptic equation 
$$
L_\psi^*w=0 \qquad\text{in }D_0,
$$
(see Theorem \ref{theo1sect5}) and \eqref{sect7:eqwwnu0} holds, then, by the same argument of \textit{Case (i)} based on \cite[Theorem 3.3.1]{Isakov}, $w$ vanishes in an interior neighborhood of $E_+$. Hence, by unique continuation, it follows that
$$
w\equiv 0 \qquad \text{in }D_0.
$$
Finally, by \eqref{eq1:statteo1sect5}, we conclude that
$$
\eta=\frac{w}{W_\psi}=0 \qquad \text{on }\partial D_0.
$$
Hence also in this case we obtain $(w,\eta)=(0,0)$.

\medskip

\noindent
\textbf{Case (iii): $\min_{\partial D_0}\psi<\psi_H$.}

\medskip

Set
$$
m:=\min_{\partial D_0}\psi
$$
and let
$$
K_-:=\{p\in\partial D_0;\,\psi(p)=m\}.
$$
By assumption,
$$
m<\psi_H.
$$
Since $\partial D_0$ is compact and $\psi$ is continuous, $K_-$ is a nonempty compact subset of $\partial D_0$.

Let $p\in K_-$. Since $p$ is a minimum point of $\psi$ on $\partial D_0$, we have
$$
 \nabla_{g_0}^\top\psi(p)=0.
$$
Moreover, as $\partial_{\nu_{g_0}}\psi=0$ on $\partial D_0$, we infer that
$$
\nabla_{g_0}\psi(p)=0, \qquad W_\psi(p)=1.
$$
Hence, for every $p\in K_-$,
$$
q_\psi(p) = (N-1)\bigl(He^{\psi(p)}W_\psi(p)-1\bigr) = (N-1)(He^m-1)<0.
$$

Since $q_\psi$ is continuous and $K_-$ is compact, there exist an open neighborhood $U$ of $K_-$ in $\partial D_0$ and a constant $\delta>0$ such that
$$
q_\psi\le -\delta<0 \qquad \text{on }U.
$$
If $U\neq \partial D_0$, then $\partial D_0\setminus U$ is compact and disjoint from $K_-$, hence the number
$$
m_-:=\min_{\partial D_0\setminus U}\psi
$$
is well defined and satisfies
$$
m_->m.
$$
In this case let us fix $\overline{m} \in \mathbb R$ such that
$$
m<\overline{m}<\min\{m_-,\psi_H\}.
$$
If instead $U=\partial D_0$, we choose $\overline{m}\in \mathbb R$ such that
$$
m<\overline{m}<\psi_H.
$$
Setting
$$
E_-:=\{p\in\partial D_0;\,\psi(p)<\overline{m}\},
$$
we obtain a nonempty open subset of $\partial D_0$ such that
$$
E_-\subset U.
$$

Now let $p_0\in E_-$, and let $\sigma_{p_0}$ be the integral curve of the vector field
$- \nabla_{g_0}^\top\psi$
on $\partial D_0$ such that
$$
\sigma_{p_0}(0)=p_0.
$$
Since $\partial D_0$ is compact, $\sigma_{p_0}$ is defined for every $t\in\mathbb R$.
Then, for all $t\ge 0$, it holds
$$
\frac{d}{dt}\psi(\sigma_{p_0}(t))=-|\nabla_{g_0}^\top\psi|_{g_0}^2(\sigma_{p_0}(t))\leq 0.
$$
Since $p_0\in E_-$, we have $\psi(p_0)<\overline{m}$, and thus
$$
\psi(\sigma_{p_0}(t))\le \psi(p_0)<\overline{m}
\qquad \forall t\ge 0.
$$
Hence
$$
\sigma_{p_0}(t)\in E_-\subset U \qquad \forall t\ge 0,
$$
and thus
$$
q_\psi(\sigma_{p_0}(t))\le -\delta<0 \qquad \forall t\ge 0.
$$
On the other hand, by \eqref{eq1:statprop1sect7}, we have
$$
\frac{d}{dt}w(\sigma_{p_0}(t))=\langle \nabla_{g_0}^\top w,\sigma_{p_0}^\prime(t)\rangle_{g_0}=-q_\psi(\sigma_{p_0}(t))\,w(\sigma_{p_0}(t)).
$$
Therefore
$$
w(\sigma_{p_0}(t)) = w(p_0)e^{\int_0^t -q_\psi(\sigma_{p_0}(s))\,ds},
$$
and, since $-q_\psi\ge \delta>0$ on $E_-$, we obtain
$$
|w(\sigma_{p_0}(t))| \geq |w(p_0)|e^{\delta t}.
$$
Again, since $w$ is continuous on the compact set $\partial D_0$, it is bounded there. Hence the previous estimate is possible only if
$$
w(p_0)=0.
$$
Since $p_0\in E_-$ is arbitrary, then $w\equiv0$ on $E_-$. This, together with
$\partial_{\nu_{g_0}}w=0$ on $\partial D_0$ implies that $w\equiv0$
in $D_0$, by the same unique continuation argument as in \textit{Cases (i)} and \textit{(ii)}.

Finally, since $\eta=w/W_\psi$ on $\partial D_0$, we obtain $\eta=0$, and so $(w,\eta)=(0,0)$. The proof is complete.
\end{proof}

\section{Proof of Theorem \ref{thm:main}}

We apply a classical transversality argument, stated in many books and papers, in particular we refer to the version stated in \cite[Theorem 5.4]{Henry}. We recall here this theorem in our framework and for Fredholm maps of index zero.

\begin{theorem}\label{thm:Henryspecial}
Let $\mathsf X,\mathsf Y,\mathsf Z$ be $C^1$ Banach manifolds, let $\mathsf O\subset \mathsf X\times \mathsf Y$ be open, let $F:\mathsf O\to \mathsf Z$ be a $C^1$ map, and let $z\in \mathsf Z$.
Assume that, at every $(x,y)\in F^{-1}(z)$:
\begin{enumerate}[label=(\roman*),leftmargin=3em]
\item $D_xF(x,y)$ is Fredholm of index zero;
\item $D F(x,y):T_x\mathsf X\times T_y\mathsf Y\to T_z\mathsf Z$ is surjective.
\end{enumerate}
Further assume that
   \begin{enumerate}[label=(\roman*),leftmargin=3em,start=3]
\item the projection $F^{-1}(z)\to \mathsf Y$ is $\sigma$-proper.
\end{enumerate}
Then the set of $y\in \mathsf Y$ for which $z$ is a regular value of $F(\cdot,y)$ is residual in $\mathsf Y$.
\end{theorem}

Let $D_0$ be the domain fixed in Section 2 and assume also that $N-1 \notin \SpecNeu(-\Delta_{g_0};D_0)$. Recall that for $h\in \DiffQ$ we have set $D_h:=h(D_0)$ and $g_h:=h^*g_0$, and denote by
$$
\mathscr A:=\bigl\{h\in\DiffQ;\ N-1\notin\SpecNeu(-\Delta_{g_0};D_h)\bigr\}
$$
the set of admissible diffeomorphisms. By Lemma \ref{lem1:sect8} below, $\mathscr A$ is an open subset of $\DiffQ$ containing $\Id$, and thus it is an admissible choice for $\Param$ in Sections 2--7. Moreover $\mathscr A$ is a $C^1$ Banach manifold modeled on $\Y$, hence it is locally completely metrizable and thus a Baire space. In particular, the residual subsets of $\mathscr A$ are dense in $\mathscr A$. 

We shall apply Theorem
\ref{thm:Henryspecial} with
\beq\label{eq:choiceofspaces}
\mathsf X:=C^{2,\alpha}(\overline{D_0}), \quad \mathsf Y:=\mathscr A, \quad \mathsf Z:=C^{0,\alpha}(\overline{D_0})\times C^{1,\alpha}(\partial D_0), \quad \mathsf O:=\mathsf X\times\mathsf Y,
\eeq
to the map $F$ defined in \eqref{eq:defF}, with $\Param=\mathscr A$, and to the
value $z=(0,0)$.

\begin{lemma}\label{lem1:sect8}
$\mathscr A$ is an open subset of $\DiffQ$, and $\Id \in \mathscr A$.
\end{lemma}
\begin{proof}
By assumption $N-1 \notin \SpecNeu(-\Delta_{g_0};D_0)$, and this means that $\Id\in\mathscr A$. To prove that $\mathscr A$ is open we first observe that, for any $h \in \DiffQ$, since $h:(D_0,g_h)\to(D_h,g_0)$ is an isometry, we have
$$
\SpecNeu(-\Delta_{g_0};D_h)=\SpecNeu(-\Delta_{g_h};D_0).
$$
The spectrum $\SpecNeu(-\Delta_{g_h};D_0)$ consists of a nondecreasing sequence $\bigl(\lambda_j(D_0,g_h)\bigr)_{j\in\N}$ of Neumann eigenvalues of $-\Delta_{g_h}$ on $D_0$, repeated according to their multiplicity. 
Since $h\mapsto g_h$ is continuous from $\DiffQ$ into the space of $C^{2,\alpha}$ symmetric $2$-tensor fields on $\overline{D_0}$, the min--max characterization of each ordered eigenvalue implies that the map $h\mapsto\lambda_j(D_0,g_h)$ is continuous on $\DiffQ$ for every $j\in\N$. Now let $h_0\in\mathscr A$. 
Since the eigenvalues are ordered and $\lambda_0(D_0,g_{h_0})=0<N-1$, there is $j\in\N$ such that $\lambda_j(D_0,g_{h_0})<N-1<\lambda_{j+1}(D_0,g_{h_0})$. 
By continuity, these two inequalities hold for every $h$ in a neighborhood of $h_0$, and for such $h$ no eigenvalue equals $N-1$. Hence $\mathscr A$ is open and the proof is complete.
\end{proof}
Next we establish the surjectivity of the total differential of $F$.
\begin{proposition}\label{prop1sect8}
Let $(\psi,h)\in C^{2,\alpha}(\overline{D_0})\times\mathscr A$ be such that
$F(\psi,h)=(0,0)$. Then the total differential
$$
DF(\psi,h):C^{2,\alpha}(\overline{D_0})\times T_h\mathscr A\longrightarrow
C^{0,\alpha}(\overline{D_0})\times C^{1,\alpha}(\partial D_0)
$$
is surjective.
\end{proposition}
\begin{proof}
We divide the proof into three steps.

\medskip

\noindent\textbf{Step 1}: Reduction to surjectivity of $DF(\psi,\Id)$.\\

This is an easy consequence of the fact that if we consider a domain $D_h$, with $h \in \mathscr A$, we could start from the map $G$ analogous to $F$, using as reference domain $D_h$ and then prove the surjectivity of $DG(\varphi,\Id)$, for some $\varphi \in C^{2,\alpha}(\overline{D_h})$. 

More precisely, let $(\psi,h)\in C^{2,\alpha}(\overline{D_0})\times\mathscr A$ and
let
$$
\mathscr A_h:=\bigl\{k\in\DiffQ;\
N-1\notin\SpecNeu(-\Delta_{g_0};k(D_h))\bigr\}
$$
be the set of admissible diffeomorphisms associated with the reference domain
$D_h$. Since $k(D_h)=(k\circ h)(D_0)$, we can also write
\beq\label{eq:Ah-vs-A}
\mathscr A_h=\bigl\{k\in\DiffQ;\ k\circ h\in\mathscr A\bigr\}.
\eeq
In particular, $\mathscr A_h$ is open, and $\Id\in\mathscr A_h$ if and only if $h\in\mathscr A$.
Let 
$$
G:C^{2,\alpha}(\overline{D_h})\times\mathscr A_h \longrightarrow
C^{0,\alpha}(\overline{D_h})\times C^{1,\alpha}(\partial D_h),
$$
$G=G(u,k)$, be the map defined by the same expressions as $F$ in \eqref{eq:defF}, with $D_h$
in place of $D_0$, thus the metric associated with the parameter $k$ is
$g_k=k^*g_0$, now on $D_h$. As pointed out in Section 2, $D_h$ is a domain of class $C^{3,\alpha}$ satisfying \eqref{eq:standing-assumptions-D0}.
We also introduce the Banach space isomorphisms
$$
\mathcal Q_h: C^{2,\alpha}(\overline{D_0})\longrightarrow
C^{2,\alpha}(\overline{D_h}),
\qquad
\mathcal Q_hu:=u\circ h^{-1},
$$
$$
\mathcal P_h: C^{0,\alpha}(\overline{D_h})\times C^{1,\alpha}(\partial D_h) \longrightarrow C^{0,\alpha}(\overline{D_0})\times C^{1,\alpha}(\partial D_0), \qquad \mathcal P_h(f,g):=(f\circ h,g\circ h).
$$
Let $k\in\mathscr A_h$. Then $k\circ h\in\mathscr A$ by \eqref{eq:Ah-vs-A}, so that $F(u\circ h,k\circ h)$ is well defined. Moreover the map
$$
h:\bigl(D_0,(k\circ h)^*g_0\bigr)\longrightarrow\bigl(D_h,k^*g_0\bigr)
$$
is an isometry, and arguing as in the proof of Lemma
\ref{lem1:sect2}, with $g_0$ replaced by $k^*g_0$, we obtain
\beq\label{eq:relationFG}
F(u\circ h,k\circ h)=\mathcal P_h\,G(u,k)\quad \forall u\in C^{2,\alpha}(\overline{D_h}).
\eeq
In particular, setting
\beq\label{eq1addprop1sect8}
\varphi:=\psi\circ h^{-1},
\eeq
it holds $F(\psi,h)=\mathcal P_hG(\varphi,\Id)$, and thus $G(\varphi,\Id)=(0,0)$.

Differentiating \eqref{eq:relationFG} with respect to the first variable and computing at $(\varphi,\Id)$, and taking into account \eqref{eq1addprop1sect8}, we check that
\beq\label{eq:relation-Dpsi}
D_\psi F(\psi,h)[v]=\mathcal P_h\,D_\varphi G(\varphi,\Id)\bigl[\mathcal Q_hv\bigr] \qquad \forall v\in C^{2,\alpha}(\overline{D_0}).
\eeq
On the other hand, let $Y\in\Y$ and let $h_t:=\Theta_h(tY\circ h)$ and
$k_t:=\Theta_{\Id}(tY)$ be the curves considered at the end of Section 4, defined for all $|t|$ sufficiently small, so that $k_t\in \mathscr A_h$.
Then by definition $h_t=k_t\circ h \in  \mathscr A$ and $\frac{d}{dt}\big|_{t=0}h_t=Y\circ h$. Taking $u=\varphi$ and
$k=k_t$ in \eqref{eq:relationFG}, differentiating with respect to $t$, and computing at $t=0$, we get
\beq\label{eq:relation-Dh}
D_hF(\psi,h)[Y\circ h]=\mathcal P_h\,D_kG(\varphi,\Id)[Y].
\eeq
Finally, since $\mathcal Q_h$ and $\mathcal P_h$ are isomorphisms and, by
\eqref{eq:tangent_space_DiffQ}, every element of $T_h \mathscr A$ is of the form $Y\circ h$
with $Y\in\Y$, the identities \eqref{eq:relation-Dpsi} and \eqref{eq:relation-Dh} show
that in order to prove that $DF(\psi,h)$ is surjective it is sufficient to prove that
$$
\operatorname{Range}D_\varphi G(\varphi,\Id)+\operatorname{Range}D_kG(\varphi,\Id) =C^{0,\alpha}(\overline{D_h})\times C^{1,\alpha}(\partial D_h),
$$
and this completes the proof of \textit{Step 1}.\\

Thanks to the above argument, and since $D_h$ is a domain of class $C^{3,\alpha}$ with $D_h\Subset Q$, $\mathscr A_h$ is an open neighborhood of $\Id$ in $\DiffQ$, and $\Y$ depends only on $Q$, all the results of Sections 2--7 hold with $(D_h,\mathscr A_h,G,\varphi)$ in place of $(D_0,\mathscr A,F,\psi)$. 
Hence, up to this renaming, we are reduced to proving the following statement, which is the content of \textit{Step 2} and \textit{Step 3}: if $(\psi,\Id)\in C^{2,\alpha}(\overline{D_0})\times\mathscr A$ satisfies $F(\psi,\Id)=(0,0)$, then $DF(\psi,\Id)$ is surjective. 
In particular, the spectral hypothesis on $D_0$ invoked in \textit{Step 2} below is, after this renaming, exactly the hypothesis $h\in\mathscr A$ of the statement.
\medskip 

\noindent\textbf{Step 2}: Proof of the surjectivity of $DF(\psi,\Id)$ in the
Sobolev framework.\\

We first observe that, since $T_\Id \mathscr A=\Y$ (see \eqref{eq:tangent_space_DiffQ}), the total differential $DF(\psi,\Id)$ is defined on $C^{2,\alpha}(\overline{D_0}) \times \Y$ by
$$DF(\psi,\Id)(u,Y)=D_\psi F(\psi,\Id)[u] + D_h F(\psi,\Id)[Y],\qquad
\forall(u,Y)\in C^{2,\alpha}(\overline{D_0}) \times \Y,$$
with values in $C^{0,\alpha}(\overline{D_0})\times C^{1,\alpha}(\partial D_0)$. As pointed out at the beginning of Section 5 the operator $D_\psi F(\psi,\Id)$ extends to a bounded linear operator from $H^2(D_0)$ to $L^2(D_0)\times L^2(\partial D_0)$, still denoted by $D_\psi F(\psi,\Id)$.

Note that, since $F(\psi,\Id)=(0,0)$, we have $\partial_{\nu_{g_0}}\psi=0$, so that $\psi$ satisfies
\eqref{eq:Neumpsi}, and thus Theorem \ref{theo1sect5}, Corollary \ref{cor:regwKernel} and Remark \ref{rem:H12-realization} are applicable.

By Remark \ref{rem:H12-realization}, $D_\psi F(\psi,\Id)$, regarded as the bounded linear operator
$$
\mathbb D_\psi F(\psi,\Id):H^2(D_0)\longrightarrow L^2(D_0)\times H^{1/2}(\partial D_0), \qquad \mathbb D_\psi F(\psi,\Id) u=\bigl(L_\psi u,\partial_{\nu_{g_0}}u\bigr),
$$
is Fredholm of index zero. 
Moreover, as $C^{0,\alpha}(\overline{D_0})\times C^{1,\alpha}(\partial D_0)$ embeds continuously into $L^2(D_0)\times H^{1/2}(\partial D_0)$, we regard $D_h F(\psi,\Id)$ as a bounded linear operator from $\Y$ into $L^2(D_0)\times H^{1/2}(\partial D_0)$.

Since $\operatorname{Range}\mathbb D_\psi F(\psi,\Id)$ is closed and of finite codimension in $L^2(D_0)\times H^{1/2}(\partial D_0)$, the subspace
$$\operatorname{Range}\mathbb D_\psi F(\psi,\Id)+\operatorname{Range}D_h F(\psi,\Id)$$ 
is closed as well. By the Hahn--Banach theorem, in order to prove that
\beq\label{eq:sobolev-surjectivity}
\operatorname{Range}\mathbb D_\psi F(\psi,\Id)+\operatorname{Range}D_h F(\psi,\Id)=
L^2(D_0)\times H^{1/2}(\partial D_0)
\eeq
it is sufficient to show that every $\ell\in(L^2(D_0)\times H^{1/2}(\partial D_0))'$ vanishing on both ranges vanishes identically.

Let $\ell\in(L^2(D_0)\times H^{1/2}(\partial D_0))'$ annihilate $\operatorname{Range}\mathbb D_\psi F(\psi,\Id)$ and $\operatorname{Range}D_h F(\psi,\Id)$, and identify $\ell$ with a pair
$$
(w,\eta)\in L^2(D_0)\times H^{-1/2}(\partial D_0).
$$
Annihilating $\operatorname{Range}\mathbb D_\psi F(\psi,\Id)$ means precisely that $(w,\eta)\in\ker\bigl((\mathbb D_\psi F(\psi,\Id))'\bigr)$. Hence, by Remark \ref{rem:H12-realization}, $\eta=\iota\eta_0$ with $(w,\eta_0)\in\ker\bigl((D_\psi F(\psi,\Id))'\bigr)$, where $\iota:L^2(\partial D_0)\hookrightarrow H^{-1/2}(\partial D_0)$ is the canonical injection (see \eqref{eq:caninj}) and, by Theorem \ref{theo1sect5} and Corollary \ref{cor:regwKernel}, we infer that
$$
w\in C^{2,\alpha}(\overline{D_0}), \qquad L_\psi^*w=0\ \text{ in }D_0, \qquad \partial_{\nu_{g_0}}w=0\ \text{ on }\partial D_0, \qquad \eta_0=\frac{w}{W_\psi}\in C^{1,\alpha}(\partial D_0).
$$
Since $\eta=\iota\eta_0$ acts on $H^{1/2}(\partial D_0)$ by integration against $\eta_0$, the annihilation of $\operatorname{Range}D_h F(\psi,\Id)$ reads as
$$
\bigl(D_h F(\psi,\Id)[Y],(w,\eta_0)\bigr)_{L^2(D_0)\times L^2(\partial D_0)}=0
\qquad\forall\,Y\in\Y ,
$$
that is, \eqref{eq:statementteo1sect7} holds for $(w,\eta_0)$. Since
$\Id\in\mathscr A$, that is,
$$
N-1\notin\SpecNeu(-\Delta_{g_0};D_0),
$$
Theorem \ref{thm1sect7} gives $(w,\eta_0)=(0,0)$, hence $\eta=0$ and $\ell=0$. The proof of \textit{Step 2} is complete.
\medskip

\noindent\textbf{Step 3}: Conclusion.\\

Let $(f,g)\in C^{0,\alpha}(\overline{D_0})\times C^{1,\alpha}(\partial D_0)$. By \eqref{eq:sobolev-surjectivity} there exist $u\in H^2(D_0)$ and $Y\in\Y$ such that
$$
\mathbb D_\psi F(\psi,\Id) u+D_h F(\psi,\Id)[Y]=(f,g).
$$
Since $D_h F(\psi,\Id)[Y]\in C^{0,\alpha}(\overline{D_0})\times C^{1,\alpha}(\partial D_0)$, we have
$$
\mathbb D_\psi F(\psi,\Id) u=(f,g)-D_h F(\psi,\Id)[Y]\in C^{0,\alpha}(\overline{D_0})\times C^{1,\alpha}(\partial D_0).
$$
Then, by Remark \ref{rem:Schauder-lifting}, it follows that $u\in C^{2,\alpha}(\overline{D_0})$, so that $(u,Y)$ belongs to the domain
$C^{2,\alpha}(\overline{D_0})\times\Y$ of $DF(\psi,\Id)$ and
$$DF(\psi,\Id)(u,Y)=D_\psi F(\psi,\Id)[u]+D_h F(\psi,\Id)[Y]=(f,g).$$
Since $(f,g)$ was arbitrary, $DF(\psi,\Id)$ is surjective, and the conclusion follows from the reduction performed in \textit{Step 1}. The proof is complete.
\end{proof}

\begin{proof}[Proof of Theorem \ref{thm:main}]
We apply Theorem \ref{thm:Henryspecial}, with $\mathsf X$, $\mathsf Y$, $\mathsf Z$ and $\mathsf O$ as in \eqref{eq:choiceofspaces}, to the map $F$ defined in \eqref{eq:defF}, with $\Param=\mathscr A$, and to the value $z=(0,0)$. By Proposition \ref{prop:FC1} the map $F:\mathsf O\to\mathsf Z$ is of class $C^1$. We check that assumptions (i) and (ii) hold at every $(\psi,h)\in F^{-1}(0,0)$, and that (iii) holds as well.

Assumption (i) holds by Proposition \ref{prop:Fredholm-general-parameter}, which gives that $D_\psi F(\psi,h):\mathsf X\to\mathsf Z$ is Fredholm of index zero. Assumption (ii) is the content of Proposition \ref{prop1sect8}, taking into account that $T_\psi\mathsf X=C^{2,\alpha}(\overline{D_0})$ and $T_h\mathsf Y=T_h\mathscr A$. 
As for assumption (iii), Proposition \ref{prop:sigmaproper} gives that the projection $\Pi:F^{-1}(0,0)\to\mathscr A$ is $\sigma$-proper. More precisely, denoting by $\overline{\mathcal B_l(0)}$ the closed ball of radius $l$ centered
at $0$ in $C^{2,\alpha}(\overline{D_0})$ and setting
$$
M_l:=\bigl(\overline{\mathcal B_l(0)}\times\mathscr A\bigr)\cap F^{-1}(0,0),
\qquad l\in\N^+,
$$
its proof shows that $F^{-1}(0,0)=\bigcup_{l\ge1}M_l$, that each $M_l$ is closed in $F^{-1}(0,0)$, and that $\Pi|_{M_l}$ is proper for every $l\in\N^+$, in the sense that for any sequence $(\psi_n,h_n)_{n\in\N}$ in $M_l$ such that $h_n\longrightarrow h$ in $\mathscr A$, as $n\to+\infty$, there exists a subsequence $(\psi_{n_j},h_{n_j})_{j\in\N}$ converging in $C^{2,\alpha}(\overline{D_0})\times\mathscr A$ to an element $(\psi,h)\in M_l$.

By Theorem \ref{thm:Henryspecial} we have that
$$
\mathscr G:=\bigl\{h\in\mathscr A;\ (0,0)\text{ is a regular value of } F(\cdot,h)\bigr\}
$$
is residual in $\mathscr A$, hence it is dense in $\mathscr A$ because, as pointed out at the beginning of this section, $\mathscr A$ is a Baire space. Let $h\in\mathscr A$ and let $\varphi\in C^{2,\alpha}(\overline{D_h})$ be a solution of \eqref{eq:probh}. 
By Lemma \ref{lem1:sect2}, $\psi:=\varphi\circ h$ satisfies $F(\psi,h)=(0,0)$ and, by \eqref{eq:relation-Dpsi}, the operator $D_\psi F(\psi,h)$ is injective if and only if $\varphi$ is nondegenerate. Since $D_\psi F(\psi,h)$ is Fredholm of index zero, it is onto if and only if it is injective. Hence $\mathscr G$ coincides with the set defined in the statement of Theorem \ref{thm:main}.

Now, let $h\in\mathscr G$. We claim that, for every $l\in\N^+$, the fiber
$$
\bigl(\Pi|_{M_l}\bigr)^{-1}(\{h\})=\bigl\{\psi\in\overline{\mathcal B_l(0)};\ F(\psi,h)=(0,0)\bigr\}\times\{h\}
$$
is a finite set. Fix $l\in\N^+$ and assume by contradiction that there exists a sequence $(\psi_n,h)_{n\in\N}$ in $\bigl(\Pi|_{M_l}\bigr)^{-1}(\{h\})$ of pairwise distinct elements. Since $\Pi|_{M_l}$ is proper and the parameters are constantly equal to $h$, there exists a subsequence $(\psi_{n_j},h)_{j\in\N}$ converging in $C^{2,\alpha}(\overline{D_0})\times\mathscr A$ to some $(\bar\psi,h)\in M_l$. 
Moreover, as the $\psi_n$ are pairwise distinct, at most one of them equals $\bar\psi$, so that, discarding it if necessary, we may assume that $\psi_{n_j}\neq\bar\psi$ for all $j\in\N$. 
On the other hand, $h\in\mathscr G$, so that $D_{\psi}F(\bar\psi,h)$ is onto and, since it is Fredholm of index zero, it is an isomorphism. 
Then, by the implicit function theorem, there exists a neighborhood $U$ of $\bar\psi$ in $C^{2,\alpha}(\overline{D_0})$ such that $\bar\psi$ is the only element of $U$ satisfying $F(\cdot,h)=(0,0)$. This contradicts $\psi_{n_j}\to\bar\psi$ with $\psi_{n_j}\neq\bar\psi$ and $F(\psi_{n_j},h)=(0,0)$, and proves the claim.

Finally, the map $\varphi\mapsto\varphi\circ h$ is a bounded linear operator from $C^{2,\alpha}(\overline{D_h})$ into $C^{2,\alpha}(\overline{D_0})$ and, by Lemma \ref{lem1:sect2}, it restricts to a bijection between the solutions of \eqref{eq:probh} and the functions $\psi\in C^{2,\alpha}(\overline{D_0})$ satisfying $F(\psi,h)=(0,0)$. 
Hence, for every $R>0$, there exists $l\in\N^+$ such that the image under this map of the set of solutions $\varphi$ with $\|\varphi\|_{C^{2,\alpha}(\overline{D_h})}\leq R$ is contained in $\bigl\{\psi\in\overline{\mathcal B_l(0)};\ F(\psi,h)=(0,0)\bigr\}$, which is finite by the claim. Since the map is injective, this proves the finiteness assertion in the statement of the theorem. The proof is complete.
\end{proof}

\appendix

\section{Preliminaries of Riemannian Geometry}

In this section, we recall some basic definitions and results from Riemannian geometry that are used throughout the paper. For further details, we refer to \cite{LeeRM} and, for the divergence of tensor fields, to \cite{ONeill}.

Let $(M,g)$ be a Riemannian manifold of dimension $m$, possibly with boundary. Let $h:M\to M$ be a diffeomorphism. The pull-back of $g$ by $h$ is the Riemannian metric $h^*g$ on $M$ defined by
\beq\label{eq:defpullbackmetric}
(h^*g)_p(X,Y):=g_{h(p)}\bigl(dh_p(X),dh_p(Y)\bigr)
\eeq
for every $p\in M$ and $X,Y\in T_pM$. Equivalently,
$$h:(M,h^*g)\longrightarrow (M,g)$$
is an isometry.

Given a one-form $\omega$ on $M$, we denote by $\omega^{\sharp_g}$ its metric dual vector field,
characterized by
\beq\label{eq:defomegasharp}
g(\omega^{\sharp_g},X)=\omega(X)
\eeq
for every vector field $X$ on $M$. When the reference metric is clear from the context,
we omit the subscript and simply write $\omega^\sharp$.

Let $f\in C^1(M)$. The gradient $\nabla_g f$ is the unique vector field on
$M$ characterized by
\beq\label{eq:defgradient}
\dd f(X)=g(X,\nabla_g f)
\eeq
for every vector field $X$ on $M$. By \eqref{eq:defomegasharp} we also have
\beq\label{eq1:appendix1}
\nabla_g f=(\dd f)^{\sharp_g},
\eeq
and, in local coordinates $(x^1,\dots,x^m)$,
\beq\label{eq1bis:appendix1}
\nabla_g f=g^{ij}\partial_i f\,\partial_j .
\eeq

If $A$ is a $(1,1)$-tensor field on $M$, we denote by
$$
\tr A
$$
its trace, namely the pointwise trace of the endomorphism $A_p:T_pM\to T_pM$. In particular, if
$\{e_i\}_{i=1,\ldots,m}$ is a local $g$-orthonormal frame, then
$$
\tr A=\sum_{i=1}^m g(Ae_i,e_i).
$$

If $T$ is a $(0,2)$-tensor field, its trace with respect to $g$ is denoted by $\tr_gT$ and is defined by
$$
\tr_g T = \sum_{i=1}^m T(e_i,e_i),
$$
where $\{e_i\}_{i=1,\ldots,m}$ is any local $g$-orthonormal frame. If, in
local coordinates,
\begin{equation}\label{eq:T_localcoordinates}
T=T_{ij}\,\dd x^i\otimes \dd x^j,
\end{equation}
then
\beq\label{eq:traceloccord}
\tr_g T=g^{ij}T_{ij}.
\eeq

Let $\nabla$ be the Levi-Civita connection of $(M,g)$. For a vector field
$X$ on $M$, we regard $\nabla X$ as the $(1,1)$-tensor field defined by
$$
(\nabla X)(Z):=\nabla_ZX
$$
for every vector field $Z$. The divergence of $X$ with respect to $g$ is
defined by
$$
\divg_g X:=\tr(\nabla X).
$$
Hence, in a local $g$-orthonormal frame, we have
$$
\divg_g X=\sum_{i=1}^m g(\nabla_{e_i}X,e_i),
$$
while if \(X=X^i\partial_i\) in a local coordinate chart, then
\begin{equation}\label{eq:divlocalcoordinates}
\divg_g X=\frac{1}{\sqrt{\det g}} \partial_i\left(\sqrt{\det g}\,X^i\right),
\end{equation}
where $\det g$ denotes the determinant of the matrix $(g_{ij})$.
If $f\in C^2(M)$, the Hessian of $f$ with respect to $g$ is the symmetric $(0,2)$-tensor defined by
$$
D_g^2f(X,Z):=g(\nabla_X\nabla_g f,Z)
$$
for every pair of vector fields $X,Z$ on $M$. It can also be written as
$$
D_g^2f(X,Z)=X(Zf)-(\nabla_XZ)f.
$$
In local coordinates,
\beq\label{eq:Hesslocalcoord}
(D_g^2f)_{ij}=\nabla_i\nabla_j f=\partial_{ij}f-\Gamma_{ij}^k\partial_k f,
\eeq
where $\Gamma_{ij}^k$ are the Christoffel symbols of $g$, which are given by
\beq\label{eq:christoffel}
\Gamma_{ij}^k=\frac12 g^{ks}\left(\partial_ig_{js}+\partial_jg_{is}-\partial_sg_{ij}\right).
\eeq

The Laplace--Beltrami operator is defined by
$$
\Delta_g f:=\divg_g(\nabla_g f).
$$
It also satisfies
$$
\Delta_g f=\tr_g(D_g^2f),
$$
and, in local coordinates,
\beq\label{eq:Laplacianlocalcoord}
\Delta_g f=g^{ij}\left(\partial_{ij}f-\Gamma_{ij}^k\partial_k f\right).
\eeq

We next recall the divergence of a $(0,2)$-tensor field. If $T$ is a $(0,2)$-tensor field on $M$, we define $\divg_g T$ as the vector field
characterized by
$$
g(\divg_g T,X)=\sum_{i=1}^m (\nabla_{e_i}T)(e_i,X)
$$
for every vector field $X$, where $\{e_i\}_{i=1,\ldots,m}$ is any local
$g$-orthonormal frame.

In local coordinates, writing $T$ as in \eqref{eq:T_localcoordinates}, the
covariant components of $\divg_g T$ are
\beq\label{eq:defDivTensor}
(\divg_g T)_j=\nabla^iT_{ij}=g^{ik}\nabla_kT_{ij}.
\eeq
After raising the index with the metric, we obtain
$$
(\divg_g T)^j=\nabla_iT^{ij},
\qquad\text{where}\qquad
T^{ij}:=g^{ia}g^{jb}T_{ab}.
$$

Let $S$ and $T$ be $(0,2)$-tensor fields on $M$. Their pointwise scalar product induced by $g$ is defined by
\begin{equation}\label{eq:defscalarproducttensors}
\langle S,T\rangle_g:=g^{ia}g^{jb}S_{ij}T_{ab}.
\end{equation}
In a local $g$-orthonormal frame $(e_1,\dots,e_m)$, this reads
$$
\langle S,T\rangle_g=\sum_{i,j=1}^{m} S(e_i,e_j)\,T(e_i,e_j).
$$
If $M$ has no boundary (in this paper we only need $M=\SN$) and $Y$ is a $C^1$ vector field on $M$, we denote by $\mathcal{L}_Yg$ the Lie derivative of $g$ along $Y$, namely
\begin{equation}\label{eq:defLieDerivative}
\mathcal L_Y g:=\left.\frac{d}{dt}\right|_{t=0}(\Phi_Y^t)^*g,
\end{equation}
where $\Phi_Y$ is the flow of $Y$ and $\Phi_Y^t=\Phi_Y(t,\cdot)$. It is well known (see \cite[Corollary 12.33]{Lee}, combined with the
compatibility of $\nabla$ with $g$ and the symmetry of $\nabla$) that
\begin{equation}\label{eq:liederiv}
(\mathcal L_Y g)(X,Z)=g(\nabla_XY,Z)+g(X,\nabla_ZY)
\end{equation}
for every pair of vector fields $X,Z$.

Finally, let $\omega$ and $\theta$ be one-forms on $M$. The tensor
$\omega\otimes\theta$ is defined by
\begin{equation}\label{eq:defalphatensorbeta}
(\omega\otimes\theta)(X,Z):=\omega(X)\theta(Z),
\end{equation}
for any vector fields $X,Z$. We also set
\begin{equation}\label{eq:defsymalphatensorbeta}
\operatorname{Sym}(\omega\otimes\theta) := \frac12(\omega\otimes\theta+\theta\otimes\omega).
\end{equation}
If $T$ is a $(0,2)$-tensor field, we check that
\begin{equation}\label{eq:tensorpairing1}
\langle \omega\otimes\theta,T\rangle_g =T(\omega^{\sharp_g},\theta^{\sharp_g}),
\end{equation}
and, if in addition, $T$ is symmetric, then
\begin{equation}\label{eq:tensorpairing2}
\bigl\langle \operatorname{Sym}(\omega\otimes\theta),T\bigr\rangle_g =T(\omega^{\sharp_g},\theta^{\sharp_g}).
\end{equation}
\section{Some useful identities}
In this section we prove a technical result that is used in Section 4. We begin with the following preliminary Lemma.
\begin{lemma}\label{lem1appendix}
Let $M$ be a compact manifold of dimension $m$, possibly with
boundary. Let $(g_t)_{t\in (-\varepsilon,\varepsilon)}$ be a $C^1$ one-parameter family of
$C^{1,\alpha}$ Riemannian metrics on $M$ and set
$$
\dot g_0=\left.\frac{d}{dt}\right|_{t=0}g_t.
$$
In a fixed local coordinate system $(x^1,\dots,x^m)$, define
\beq\label{eqstat:lem1appB}
J_t:=\sqrt{\det(g_t)}.
\eeq
Then
$$
\left.\frac{d}{dt}\right|_{t=0}\log J_t=\frac12\,\tr_{g_0} \dot g_0.
$$
\end{lemma}

\begin{proof}
Let $(x^1,\dots,x^m)$ be a fixed local coordinate system. Let us write
$$
g_t=(g_{ij}(t))_{i,j=1}^m,
$$
and define $J_t$ as in \eqref{eqstat:lem1appB}. By Jacobi's formula for the derivative of the determinant, we have
$$
\frac{d}{dt}\det(g_t)=\det(g_t)\, \tr\left(g_t^{-1}\frac{d}{dt}g_t
\right).
$$
Evaluating at $t=0$, we obtain
$$
\left.\frac{d}{dt}\right|_{t=0}\det(g_t) =\det(g_0)\,\tr(g_0^{-1} \dot g_0).
$$
Since $\dot g_0$ is symmetric, and recalling \eqref{eq:traceloccord}, we get
$$
\tr(g_0^{-1} \dot g_0)= g_0^{ij}(\dot g_0)_{ji}= g_0^{ij}(\dot g_0)_{ij}=\tr_{g_0}\dot g_0.
$$
Therefore,
$$
\left.\frac{d}{dt}\right|_{t=0}\det(g_t)=\det(g_0)\,\tr_{g_0}\dot g_0.
$$
Finally, since $J_t=\sqrt{\det(g_t)}$ (see \eqref{eqstat:lem1appB}), it follows that
$$
\begin{aligned}
\left.\frac{d}{dt}\right|_{t=0}\log J_t
&=\frac12 \left.\frac{d}{dt}\right|_{t=0}\log\det(g_t)\\
&=\frac12\,\tr_{g_0}\dot g_0,
\end{aligned}
$$
and the proof is complete.
\end{proof}

\begin{lemma}\label{lem2appendix}
Let $M$ be a compact manifold of dimension $m$, possibly with
boundary. Let $(g_t)_{t\in (-\varepsilon,\varepsilon)}$ be a $C^1$ one-parameter family of
$C^{1,\alpha}$ Riemannian metrics on $M$, and let $(X_t)_{t\in (-\varepsilon,\varepsilon)}$ be a $C^1$ one-parameter family of
$C^{1,\alpha}$ vector fields on $M$. 
Set
$$
\dot g_0=\left.\frac{d}{dt}\right|_{t=0}g_t,\qquad\dot X_0:=\left.\frac{d}{dt}\right|_{t=0}X_t.
$$
Then
\begin{equation}\label{eq1:lem2appendix}
\left.\frac{d}{dt}\right|_{t=0}\divg_{g_t}(X_t) = \divg_{g_0}(\dot X_0) + \frac12 \left\langle \nabla_{g_0}(\tr_{g_0}\dot g_0), X_0 \right\rangle_{g_0}.
\end{equation}
\end{lemma}

\begin{proof}
Let $(x^1,\dots,x^m)$ be a fixed local coordinate system and write
$$
X_t= X_t^i\,\partial_i.
$$
Let $J_t$ be defined as in \eqref{eqstat:lem1appB}. Then, recalling the coordinate expression for the divergence of a vector
field (see \eqref{eq:divlocalcoordinates}), we have
$$
\divg_{g_t}(X_t)=\frac1{J_t}\,\partial_i(J_tX_t^i).
$$
Hence
\begin{eqnarray*}
\left.\frac{d}{dt}\right|_{t=0}\divg_{g_t}(X_t)
&=& \left.\frac{d}{dt}\right|_{t=0}
\left(\frac1{J_t}\,\partial_i(J_tX_t^i)\right)\\[12pt]
&=& \left[-\frac{\dot J_0}{J_0^2}\,\partial_i(J_0X_0^i)+\frac1{J_0}\,\partial_i(\dot J_0X_0^i)+\frac1{J_0}\,\partial_i(J_0\dot X_0^i) \right].
\end{eqnarray*}
Since
$$
\frac1{J_0}\,\partial_i(J_0X_0^i)=\divg_{g_0}(X_0)
$$
and
$$
\frac1{J_0}\,\partial_i(J_0\dot X_0^i)=\divg_{g_0}(\dot X_0),
$$
we obtain
\begin{equation}\label{eq2:lem2appendix}
\left.\frac{d}{dt}\right|_{t=0}\divg_{g_t}(X_t) =-\frac{\dot J_0}{J_0}\,\divg_{g_0}(X_0)+\frac1{J_0}\,\partial_i(\dot J_0X_0^i)+\divg_{g_0}(\dot X_0).
\end{equation}
Set
$$
k_0:=\frac{\dot J_0}{J_0}=\left.\frac{d}{dt}\right|_{t=0}\log J_t.
$$
Then $\dot J_0=J_0k_0$, and thus
$$
\frac1{J_0}\,\partial_i(\dot J_0X_0^i)=\frac1{J_0}\,\partial_i(J_0 k_0X_0^i)=\divg_{g_0}(k_0X_0).
$$
Substituting this identity into
\eqref{eq2:lem2appendix}, we obtain
\begin{equation}\label{eq3:lem2appendix}
\left.\frac{d}{dt}\right|_{t=0}\divg_{g_t}(X_t)
=-k_0\,\divg_{g_0}(X_0)+\divg_{g_0}(k_0 X_0)+\divg_{g_0}(\dot X_0).
\end{equation}
Now, by the product rule, we have
\begin{equation}\label{eq4:lem2appendix}
\begin{aligned}
\divg_{g_0}(k_0X_0)
&=\frac1{J_0}\,\partial_i(J_0 k_0X_0^i)\\
&=k_0\frac1{J_0}\,\partial_i(J_0X_0^i)
+(\partial_i k_0)X_0^i\\
&=k_0\,\divg_{g_0}(X_0)+\left\langle\nabla_{g_0}k_0,X_0\right\rangle_{g_0},
\end{aligned}
\end{equation}
and thus, from \eqref{eq3:lem2appendix} and \eqref{eq4:lem2appendix}, we get that
\begin{equation}\label{eq5:lem2appendix}
\left.\frac{d}{dt}\right|_{t=0}\divg_{g_t}(X_t)=\divg_{g_0}(\dot X_0)+\left\langle \nabla_{g_0} k_0, X_0\right\rangle_{g_0}.
\end{equation}
Moreover, by Lemma \ref{lem1appendix}, it holds
$$
k_0=\left.\frac{d}{dt}\right|_{t=0}\log J_t=\frac12\,\tr_{g_0} \dot g_0,
$$
so that
\begin{equation}\label{eq6:lem2appendix}
\left\langle \nabla_{g_0} k_0, X_0 \right\rangle_{g_0}=\frac12 \left\langle \nabla_{g_0}(\tr_{g_0} \dot g_0), X_0 \right\rangle_{g_0}.
\end{equation}
Finally, combining \eqref{eq5:lem2appendix} and
\eqref{eq6:lem2appendix}, we obtain
\eqref{eq1:lem2appendix}. The proof is complete.
\end{proof}


\begin{thebibliography}{10}

\bibitem{Calderon}
A. P. Calder\'on, \emph{Uniqueness in the Cauchy problem for partial differential equations},  Amer. J. Math. \textbf{80}, 16--36 (1958).

\bibitem{ChoePark2011}
J. Choe, S.-H. Park,
\emph{Capillary surfaces in a convex cone},
Math. Z. \textbf{267}, 875--886 (2011).

\bibitem{GarofaloLin} N. Garofalo, F.-H. Lin, \emph{Unique continuation for elliptic operators: a geometric-variational approach}, Comm. Pure Appl. Math. \textbf{40}, 347--366 (1987).

\bibitem{Grisvard} P. Grisvard, \emph{Elliptic Problems in Nonsmooth Domains},
Monographs and Studies in Mathematics, vol. 24, Pitman, Boston, MA, 1985.

\bibitem{Henry} D. Henry, \emph{Perturbation of the Boundary in Boundary-Value Problems of Partial Differential Equations},
London Mathematical Society Lecture Note Series, vol. 318,
Cambridge University Press, Cambridge, 2005.

\bibitem{Hirsch}
M. W. Hirsch, \emph{Differential Topology}, Graduate Texts in Mathematics, vol. 33, Springer-Verlag, New York, 1976.

 \bibitem{IPW} A. Iacopetti, F. Pacella, T. Weth, \emph{Existence of nonradial domains for overdetermined and isoperimetric problems in nonconvex cones}, Arch. Ration. Mech. Anal. \textbf{245}, 1005--1058 (2022).
 
 \bibitem{Isakov}
V. Isakov, \emph{Inverse Problems for Partial Differential Equations},
2nd ed., Applied Mathematical Sciences, vol. 127, Springer, New York, 2006.
 
\bibitem{Lee} J. M. Lee, \emph{Introduction to Smooth Manifolds}, 2nd ed., Graduate Texts in Mathematics, vol. 218, Springer, New York, 2013.

\bibitem{LeeRM}
J. M. Lee,
\emph{Introduction to Riemannian Manifolds}, 2nd ed., Graduate Texts in Mathematics, vol. 176, Springer, Cham, 2018.
 
\bibitem{LiebermanODP} G. M. Lieberman, \emph{Oblique Derivative Problems for Elliptic Equations}, World Scientific Publishing Co. Pte. Ltd., Hackensack, NJ, 2013.

\bibitem{LionsMagenes} J.-L. Lions, E. Magenes, \emph{Non-Homogeneous Boundary Value Problems and Applications. Vol. I},
Die Grundlehren der mathematischen Wissenschaften, vol. 181, Springer-Verlag, Berlin--Heidelberg--New York, 1972.

\bibitem{Micheletti1972} A. M. Micheletti, \emph{Perturbazione dello spettro dell'operatore di Laplace, in relazione ad una variazione del campo},
Ann. Scuola Norm. Sup. Pisa Cl. Sci. (3) \textbf{26}, 151--169 (1972).

\bibitem{Micheletti1973}
A. M. Micheletti, \emph{Perturbazione dello spettro di un operatore ellittico di tipo variazionale, in relazione ad una variazione del campo}, Ann. Mat. Pura Appl. (4) \textbf{97}, 267--281 (1973).

\bibitem{ONeill} B. O'Neill, \emph{Semi-Riemannian Geometry with Applications to Relativity},
Pure and Applied Mathematics, vol. 103, Academic Press, New York, 1983.

\bibitem{PacellaTralli2020}
F. Pacella, G. Tralli,
\emph{Overdetermined problems and constant mean curvature surfaces in cones},
Rev. Mat. Iberoam. \textbf{36}, 841--867 (2020).


\bibitem{Quinn1970}
F. Quinn, \emph{Transversal approximation on Banach manifolds}, Proc. Sympos. Pure Math. \textbf{15}, Amer. Math. Soc., Providence, RI, 213--222 (1970).

\bibitem{RitoreRosales2004}
M. Ritor\'e, C. Rosales,
\emph{Existence and characterization of regions minimizing perimeter under a volume constraint inside Euclidean cones},
Trans. Amer. Math. Soc. \textbf{356}, 4601--4622 (2004).

\bibitem{SautTemam1979} J.-C. Saut, R. Temam, \emph{Generic properties of
nonlinear boundary value problems}, Comm. Partial Differential Equations
\textbf{4}, 293--319 (1979).

\bibitem{Uhlenbeck1972}
K. Uhlenbeck,
\emph{Eigenfunctions of Laplace operators},
Bull. Amer. Math. Soc. \textbf{78}, 1073--1076 (1972).

\bibitem{Uhlenbeck1976}
K. Uhlenbeck,
\emph{Generic properties of eigenfunctions},
Amer. J. Math. \textbf{98}, 1059--1078 (1976).


\bibitem{Wittmann}
J. Wittmann,
\emph{The Banach manifold $C^k(M,N)$},
Differential Geom. Appl. \textbf{63}, 166--185 (2019).

\end{thebibliography}
\end{document}